\documentclass[11pt]{amsart}

\usepackage{amssymb,amsfonts,amsthm,amsmath}

\usepackage{booktabs}

\usepackage{graphicx, verbatim, centernot}
\usepackage{caption}
\usepackage{xcolor}

\usepackage{multicol}
\usepackage{enumitem}
\usepackage[normalem]{ulem}

\usepackage[a4paper, hmargin=2cm, vmargin={3cm, 3cm}]{geometry}
\usepackage{fancyhdr}
\usepackage{hyperref}
\hypersetup{
	colorlinks=true,
	linkcolor=blue,
	citecolor=purple,
	filecolor=magenta,      
	urlcolor=purple,
}

\usepackage[
backend=bibtex,
style=alphabetic,
sortlocale=auto,
uniquelist=false,
natbib=false,
url=false, 
doi=false,
eprint=false,
safeinputenc=false,
giveninits=true,
isbn=false,
maxbibnames=10,
]{biblatex}

\renewbibmacro{in:}{}

\DeclareFieldFormat*[inbook]{citetitle}{#1}
\DeclareFieldFormat*{title}{#1}

\bibliography{JLM_HCG}

\newtheorem{theorem}{Theorem}

\newtheorem{lemma}[theorem]{Lemma}

\newtheorem{proposition}[theorem]{Proposition}
\newtheorem{corollary}[theorem]{Corollary}

\newtheorem{claim}[theorem]{Claim}

\theoremstyle{definition}
\newtheorem*{definition*}{Definition}
\newtheorem{definition}[theorem]{Definition}

\newcommand{\theoremname}{testing}
\theoremstyle{remark}
\newtheorem*{remark*}{Remark}

\numberwithin{equation}{section}

\def\eps{{\varepsilon}}
\def\La{{\Lambda}}

\def\la{{\lambda}}

\newcommand{\bR}{\mathbb R}

\newcommand{\bZ}{\mathbb Z}
\newcommand{\bD}{\mathbb D}

\newcommand{\bP}{\mathbb P}
\newcommand{\bS}{\mathbb S}
\newcommand{\bE}{\mathbb E}

\newcommand{\Var}{{\sf Var}}

\begin{document}
	
	\captionsetup[figure]{labelfont={bf},labelformat={default},labelsep=period,name={Fig.}}
	
	\title[Large fluctuations of the Hierarchical Coulomb gas]{Large charge fluctuations in the Hierarchical Coulomb gas}
	
	\author{Alon Nishry}
	\address{\tiny{Alon Nishry, School of Mathematical Sciences, Tel Aviv University}}
	\email{alonish@tauex.tau.ac.il}
	
	\author{Oren Yakir}
	\address{\tiny{Oren Yakir, School of Mathematical Sciences, Tel Aviv University}}
	\email{oren.yakir@gmail.com}
	
	\begin{abstract}
		The two-dimensional one-component plasma (OCP) is a model of electrically charged particles which are embedded in a uniform background of the opposite charge, and interact through a logarithmic potential. More than 30 years ago, Jancovici, Lebowitz and Manificat discovered an asymptotic law for probabilities of large charge fluctuations in the OCP. We prove that this law holds for the \emph{hierarchical} counterpart of the OCP. The hierarchical model was recently introduced  by Chatterjee, and is inspired by Dyson’s hierarchical model of the Ising ferromagnet.
	\end{abstract}
	\maketitle 
	\thispagestyle{empty}
	\section{Introduction}
	The two-dimensional ``one-component plasma'' (sometimes also called the Coulomb gas, or the jellium model) is a system of `electrically charged' particles which interact through a logarithmic potential and are immersed in a uniform background of the opposite charge. Mathematically, this is a point process with $n\ge 2$ points $x_1,\ldots,x_n\in \bR^2$, whose joint probability density function is given by
	\begin{equation}
		\label{eq:density_for_euclidean_coulomb}
		\frac{1}{Z_E(n,\beta)} \exp\Big(-\beta\mathcal{H}_n^E(x_1,\ldots,x_n)\Big)
	\end{equation}
	where
	\begin{equation*}
		\mathcal{H}_n^E(x_1,\ldots,x_n) = \sum_{1\le j<k\le n} \log\frac{1}{|x_j-x_k|} + \frac{1}{2}\sum_{j=1}^{n} |x_j|^2 \, .
	\end{equation*}
	Here, $\beta>0$ is a parameter known as the \emph{inverse temperature}, and the normalizing constant $Z_E(n,\beta)$ is called the \emph{partition function}. The `external field' $\frac{1}{2}|\cdot|^2$ keeps the particles from escaping to infinity.
	
	The one-component plasma has attracted the interest of mathematicians and physicist alike, see for example the recent surveys by Serfaty~\cite{SerfatyICM} and Lewin~\cite{Lewin}. It is known since the classical work of Ginibre~\cite{Ginibre} that the joint density~\eqref{eq:density_for_euclidean_coulomb} with $\beta=2$ is the density of the eigenvalues of a random matrix with i.i.d.\ complex Gaussian entries (which form a determinantal point process, e.g. \cite[Sec.~6.4]{GAFbook}). Generally speaking, the one-component plasma with $\beta \ne 2$ is much less understood than the $\beta=2$ case (e.g.~\cite[Sec.~4]{GinibreSurvey22}). 
	
	\subsection{Charge fluctuations}
	It is well known that with high probability (as $n\to \infty$), the points sampled from the density~\eqref{eq:density_for_euclidean_coulomb} are uniformly distributed in a disk of radius $\sqrt{n/\pi}$ centered at the origin. Denoting by $\bD(0,R)$ the disk of radius $R$ centered at the origin, the \emph{charge fluctuation} is the random variable given by
	\begin{equation*}
		\label{eq:charge_fluc_for_euclidean_coulomb}
		\nu_n(R) = \# \text{ of points in } \bD(0,R) - \pi R^2\, .
	\end{equation*}
	Fluctuations of $\nu_{n}(R)$ have been of interest to physicists, see e.g. in~\cite{MartinYalcin,Lebowitz,JLM}. In these papers, the existence of a limit law as $n\to \infty$ for the point process~\eqref{eq:density_for_euclidean_coulomb} is taken for granted, and the authors study $\nu_\infty(R)$ as $R$ becomes large. The existence of a unique infinite system is known  only in the special case $\beta=2$. 
	
	In~\cite{MartinYalcin}, Martin and Yalcin argued that 
	\begin{equation}
		\label{eq:variance_for_charge_fluctuation}
		\qquad \Var\big[\nu_\infty(R)\big] \sim c_\beta R \qquad \text{as   } R\to \infty \, ,
	\end{equation}
	for some $c_\beta>0$.
	It is worth mentioning that for the Poisson point process (i.e.\ the limit of the $\beta=0$ case), the variance in~\eqref{eq:variance_for_charge_fluctuation} scales like $R^2$, which reflects on the fact that points sampled according to~\eqref{eq:density_for_euclidean_coulomb} with $\beta>0$ tend to have smaller discrepancy than points which do not interact. Later on, Lebowitz~\cite{Lebowitz} argued the fluctuations of $\nu_\infty(R)/\sqrt{R}$ are asymptotically normal as $R\to \infty$.
	These predictions can be verified for the special case $\beta=2$ through exact computations. 
	In a recent preprint, Lebl\'e~\cite{Leble} shows there exists a constant $c>0$ so that for all $R$ large enough
	\[
	\limsup_{n\to \infty} \Var\big[\nu_n(R)\big] \le \frac{R^2}{(\log R)^c}\, .
	\]
	\subsection{The JLM Law}
	While the Martin-Yalcin prediction~\eqref{eq:variance_for_charge_fluctuation} is about typical behavior for the charge fluctuations, one can also ask about rare events. Jancovici, Lebowitz and Manificat~\cite{JLM} discovered that for $\alpha>\tfrac{1}{2}$,
	\begin{equation}
		\label{eq:JLM_prediction}
		\log \bP\Big[|\nu_\infty(R)| \ge R^\alpha\Big] \asymp -R^{\varphi(\alpha) + o(1)}, \quad \text{ as } R \to \infty,
	\end{equation}
	where $\varphi$ is the piece-wise linear function
	\begin{equation}
		\label{eq:def_of_phi_alpha}
		\varphi(\alpha) = \begin{cases}
			2\alpha - 1, & \frac{1}{2}<\alpha \le 1; \\ 3\alpha -2, & 1 \le \alpha \le 2; \\ 2\alpha & \alpha \ge 2.
		\end{cases}
	\end{equation}
	In a usual probabilistic terminology, $\alpha\in(\tfrac{1}{2},1)$ is the `moderate deviations' regime, the case $\alpha=1$ is the classical large deviations, the hole event is a special case of $\alpha=2$ and $\alpha>2$ is the extreme overcrowding regime. The asymptotics~\eqref{eq:JLM_prediction}, which we call the \emph{JLM law}, is proved in~\cite{FenzlLambert, JLM, Shirai} for the case $\beta=2$. Very recently, Thoma~\cite{Thoma} obtained the upper bound in the overcrowding regime $\alpha>2$ for general $\beta>0$. 
	
	\subsection{The hierarchical Coulomb gas}
	In this paper we study a closely related model to the one-component plasma, namely, its hierarchical counterpart as introduced recently by Chatterjee in~\cite{chatterjee}. In this model, $x_1,\ldots,x_n\in [0,1]^2$ are sampled according to the density
	\begin{equation}
		\label{eq:density_of_HCG_intro}
		\frac{1}{Z(n,\beta)} \exp\bigg(-\beta \sum_{1\le i <j\le n}w(x_i,x_j)\bigg)
	\end{equation}
	where $w(x,y)$ is the minimal $k$ such that $x,y$ belong to different dyadic sub-squares of $[0,1]^2$ of side length $2^{-k}$, see Figure~\ref{fig:illustration_of_w}.     
	\begin{figure}[!htbp]
		\begin{center}
			\qquad \qquad \scalebox{0.3}{\includegraphics{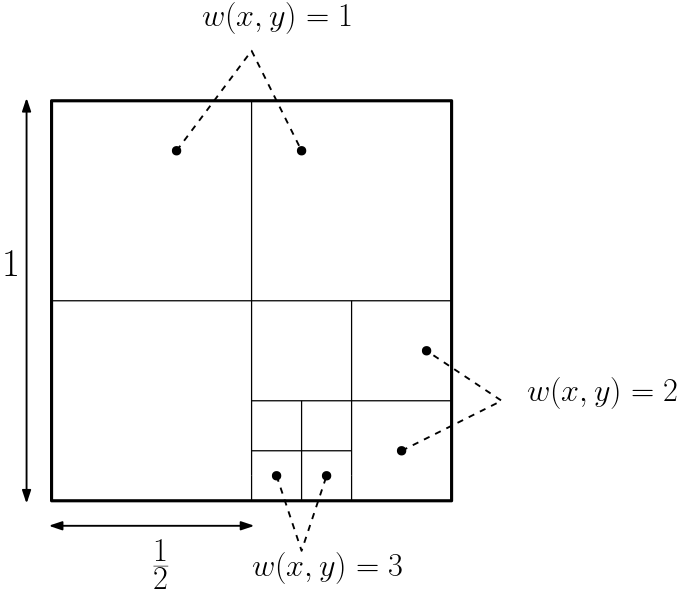}}
		\end{center}
		\caption{Illustration of the hierarchical interacting potential $w(x,y)$.}
		\label{fig:illustration_of_w}
	\end{figure}
	There might be some ambiguity about points which lie on boundaries of dyadic squares, but since these point form a set of measure zero, they do not matter. It is evident that there exists an absolute constant $C>0$ so that
	\[
	w(x,y) \le C \log\big(|x-y|^{-1}\big), \qquad \text{for } |x-y| \le 1.
	\]
	On the other hand, there is another absolute constant $c>0$ so that the average of $w(x,y)$ over all points $|x-y|=\delta$ is bounded from below by
	$
	c \log\big(\delta^{-1}\big),
	$
	for all $\delta\in(0,1)$.  
	Changing the Euclidean distance to a hierarchical distance as above is an old idea of Dyson~\cite{Dyson}, which implemented it to study a hierarchical version of the Ising model with long-range interactions.
	
	Throughout, we denote by $\mu_n$ the random counting measure associated with the joint density~\eqref{eq:density_of_HCG_intro}, omitting the dependence on the inverse temperature $\beta>0$ which will be fixed throughout. We denote by ${\tt Leb}$ the Lebesgue measure on $[0,1]^2$. For any Borel subset $U\subset [0,1]^2$, simple symmetry consideration shows that
	\begin{equation}
		\label{eq:introdution_expectation_of_mu_n}
		\bE[\mu_n(U)] = n \, {\tt Leb} (U) \, .
	\end{equation}
	As the typical point separation of the gas is of order $n^{-1/2}$, it is natural to study the ``blown-up'' version. That is, for a fixed $z\in (0,1)^2$, we consider the disk
	\begin{equation*}
		\bD_n(z,R) = \bD_n(R) = \bD\Big(z,\frac{R}{\sqrt{n}}\Big) = \Big\{x\in \bR^2 : \, |x-z|\le \frac{R}{\sqrt{n}}\Big\}\, .
	\end{equation*}
	Clearly, for all $z,R$ fixed and $n$ large enough, we have $\bD_n(z,R)\subset [0,1]^2$. In~\cite{chatterjee}, Chatterjee proved (among other things) that for all $\beta>0$ we have
	\begin{equation*}
		c_\beta R \le \Var\big[\mu_n(\bD_n(z,R))\big] \le C_\beta R \, (\log R)^2
	\end{equation*}
	uniformly as $n\to \infty$, thereby confirming the Martin-Yalcin prediction~\eqref{eq:variance_for_charge_fluctuation} for the hierarchical model (up to logarithmic factors in $R$). The main result of this paper is a verification of the JLM law~\eqref{eq:JLM_prediction} for the hierarchical Coulomb gas.
	
	\begin{theorem}
		\label{thm:Micro_JLM}
		The Jancovici-Lebowitz-Manificat law holds for the hierarchical Coulomb gas model in two-dimensions, at all scales. More precisely, for all $z\in(0,1)^2$, $\alpha>\tfrac{1}{2}$ and $\eps>0$, there exists $R_0=R_0(\alpha,\beta,z,\eps)$ large enough so that for all $R>R_0$ we have,
		\[
		\exp\Big(-R^{\varphi(\alpha)+\eps}\Big) \le \bP\Big[\big|\mu_n(\bD_n(z,R)) - \pi R^2\big| \ge R^\alpha \Big] \le \exp\Big(-R^{\varphi(\alpha)-\eps}\Big)
		\]
		uniformly as $n\to\infty$. Here $\varphi$ is the piecewise linear function defined by~\eqref{eq:def_of_phi_alpha}.
	\end{theorem}
	\begin{remark*}
		Although Theorem~\ref{thm:Micro_JLM} is stated for large microscopic disks, the proof does not rely on the specific geometry of the domain in any way. A similar result for more general domains with regular boundary can be obtained via standard geometric arguments (see, for example, \cite[Lemma~3.12]{chatterjee}). For clarity and consistency, we chose to assume the domain is a disk, as assumed in~\cite{JLM}, and do not discuss this extension in further detail.
	\end{remark*}
	\subsection{Related works}
	One could study local properties of the one-component plasma via \emph{smooth} test functions instead of counting points. In recent years a better understanding of the typical fluctuations was obtained, see Bauerschmidt-Bourgade-Nikula-Yau~\cite{BauerschmidtBourgadeNikulaYau,BauerschmidtBourgadeNikulaYauQuasiFree} and Lebl\'e-Serfaty~\cite{LebleSerfatyGAFA}. Another related direction of research focuses on obtaining tail-estimates which matches the JLM prediction on a macroscopic scale (that is, when $R$ grows like $c\sqrt{n}$ for some $c>0$). This was considered for the hole probability ($\alpha=2$) in~\cite{Adhikari,Charlier} and for all $\alpha \in (4/3,2)$ in~\cite{GhoshNishry}. We also mention some related results regarding almost sure uniform discrepancy bounds of the charge fluctuations when $\beta=\beta_n$ increases with the number of points. This phenomena, known as the  ``freezing regime", is discussed in~\cite{Ameur-Romero,Marceca-Romero}.

	The hierarchical Coulomb gas model~\eqref{eq:density_of_HCG_intro} was introduced by Chatterjee in~\cite{chatterjee}, which further defined an analogous hierarchical model in all dimensions. In~\cite{chatterjee}, Chatterjee proved reduced charge fluctuations for $d\in\{1,2,3\}$ and gave upper and lower bounds for fluctuations of smooth test functions, at all scales. Later on, Ganguly and Sarkar~\cite{GangulySarkar} studied the hierarchical model for $d\ge 3$, and established the correct fluctuations for smooth test functions (up to logarithmic factors) in all dimensions bigger than three. We mention that an important intermediate step in~\cite[Eq. 3.11]{GangulySarkar} is proving that for $d\ge3$ there is a strong concentration of the partition function around the contribution coming from the ground states (i.e. those configurations that minimize the potential energy). Such strong concentration does not hold in the cases $d\in\{1,2\}$ due to the different nature of the interacting potential in these cases. Consequently, the analogue of the JLM law for $d\ge 3$ does not apply to this hierarchical model. The work~\cite{JLM} deals with a different temperature regime than the one in~\cite{chatterjee,GangulySarkar}, and we will not address the extension of the JLM law for $d\ge 3$ in this paper. 
		
	Interestingly, although originally conjectured for the one-component plasma, the JLM law~\eqref{eq:JLM_prediction} holds for other two-dimensional point processes with reduced fluctuations. For instance, in~\cite{Nazarov-Sodin-Volberg} Nazarov, Sodin and Volberg proved the JLM law holds for the zero set of the \emph{Gaussian Entire Function} (GEF), a random entire function which is also a Gaussian process. We mention that partial results towards the JLM law for GEF zeros were already obtained by Sodin-Tsirelson~\cite{SodinTsirelson3} and Krishnapur~\cite{Krishnapur}. Another example is the random set obtained by perturbing the points of the integer lattice $\bZ^2$ with i.i.d.\ Gaussian random vectors. Checking that the JLM law holds for the latter is a routine computation. 
	
	Ever since the classical work of Dyson~\cite{Dyson}, the study of hierarchical counterparts for models in statistical mechanics (as done in this paper) has become quite popular. This is largely due to the ``tree structure" inherent in hierarchical models, which often makes the analysis more tractable. For instance, in Dyson's hierarchical model, the critical exponents were computed exactly by Bleher and Sinai~\cite{BleherSinai,BleherSinai2}, confirming some of the predictions from the physics literature. 
	Other notable examples are the hierarchical $\phi^4$-model studied by Gawedzki-Kupiainen~\cite{GawedzkiKupiainen1,GawedzkiKupiainen}, hierarchical directed polymers studied by Derrida-Griffiths~\cite{DerridaGriffiths}, hierarchical Anderson model~\cite{Bovier,Kritchevski,FyodorovOssipovRodriguez}; by no means this is a complete list.
	While we do not delve into these examples here, we refer the interested reader to the original papers and the books by Collet-Eckmann~\cite{ColletEckmann}, Sinai~\cite{Sinai} and Bauerschmidt-Brydges-Slade~\cite{BauerschmidtBrydgesSlade} for a more comprehensive study of hierarchical models.
	
	\subsection{Structure of the proof}\label{sec:ideas_from_the_proof}
	The proof of Theorem~\ref{thm:Micro_JLM} splits into the different regimes $\alpha\in(\tfrac{1}{2},1)$, $\alpha\in(1,2)$ and $\alpha>2$, where in each regime we need to prove the corresponding upper and lower bounds. The paper is broken up accordingly:
	\begin{itemize}
		\item In Section~\ref{sec:preliminaries} we collect some preliminaries that are used throughout the paper;
		\item In Section~\ref{sec:UB_large_deviations} we prove the upper bound in the range $1<\alpha<2$;
		\item In Section~\ref{sec:UB_moderate_deviations} we prove the upper bound in the range $\frac{1}{2}<\alpha<1$;
		\item In Section~\ref{sec:LB_large_deviations} we prove the lower bound in the range $1<\alpha<2$;
		\item In Section~\ref{sec:LB_moderate_deviations} we prove the lower bound in the range $\frac{1}{2}<\alpha<1$;
		\item Finally, in Section~\ref{sec:overcrowding}, we conclude the proof of Theorem~\ref{thm:Micro_JLM} by proving both the upper and lower bounds in the overcrowding regime $\alpha>2$.
	\end{itemize}
	Denote by $\mathcal{D}_k$ the collection of all dyadic sub-squares of $[0,1]^2$ and recall that $\mu_n$ is the random counting measure for the point sampled according to~\eqref{eq:density_of_HCG_intro}. Towards the proof of Theorem~\ref{thm:Micro_JLM}, a key observation is that point counts on dyadic squares are concentrated. First, we will show (by combining Lemma~\ref{lemma:tail_probability_for_disc_at_top_level} with Lemma~\ref{lemma:overcrowding_probability_in_dyadic_cube}) that for any square at the top level $Q\in \mathcal{D}_1$ and $\eps>0$ we have
	\begin{equation}
		\label{eq:introduction_concentration_top_cube}
		\bP\Big[ \big|\mu_n(Q) - \tfrac{n}{4}\big| \ge \eps n\Big] \asymp \exp\Big( - c(\beta,\eps) \, n^2 \Big)
	\end{equation}
	By a certain self-similarity property of the model (see Proposition~\ref{prop:self_similarity_and_conditional_independence} for the exact formulation), we can use induction and ``propagate down'' the strong concentration~\eqref{eq:introduction_concentration_top_cube} to other dyadic squares in $\mathcal{D}_k$, provided that $k$ is not too large. More formally, we will show (with the aid of Lemma~\ref{lemma:concentration_on_dyadic_cubes}) that if $L=n/4^k$ and $L\gg 1$ is large enough, then for all $Q\in \mathcal{D}_k$ and $\eps>0$
	\begin{equation}
		\label{eq:introduction_concentration_general_cube}
		\bP\Big[ \big|\mu_n(Q) - L\big| \ge \eps L\Big] \asymp \exp\Big( - c(\beta,\eps) \, L^2 \Big)\, .
	\end{equation}
	While the extreme overcrowding regime $\alpha>2$ follows by a different (and easier) argument, we conclude the introduction by indicating how~\eqref{eq:introduction_concentration_general_cube} is used in proving the JLM law for $\tfrac{1}{2}<\alpha<2$. We remark that the structure of our proof in this regime is somewhat similar to the structure of proof of the JLM law for the zero set of the GEF as in~\cite{Nazarov-Sodin-Volberg}. In particular, the strong concentration estimate~\eqref{eq:introduction_concentration_general_cube} is analogous to a similar strong concentration result proved for the change of argument of the Gaussian Entire Function along nice curves, see Lemmas~9 and~10 in~\cite{Nazarov-Sodin-Volberg}. 
	
	Following the physical intuition from Jancovici-Lebowitz-Manificat~\cite{JLM}, it is reasonable to assume that the large fluctuations are present in a thin annulus around the disk $\bD_{n}(R)$, where the width of the annuli depends on the size of the charge fluctuations. Starting with the range $\alpha\in(1,2)$, we recall that $\bD_n(R)$ is a disk of radius $R/\sqrt{n}$ and cover $\partial \bD_n(R)$ by dyadic squares in $\mathcal{D}_j$, with $j\ge 1$ chosen so that $$2^{-j} \approx R^{\alpha-1}/\sqrt{n}\, ,$$ see Figure~\ref{fig:jlm_intuition}. We note that the number of \emph{boundary squares} (i.e. those squares in $\mathcal{D}_j$ which intersect $\partial \bD_n(R)$) is $\approx R^{2-\alpha}$. Next, we observe that by~\eqref{eq:introdution_expectation_of_mu_n} for each boundary square $Q\in \mathcal{D}_j$ we have
	\[
	\bE\big[\mu_n(Q)\big] = \frac{n}{4^j} \approx R^{2\alpha-2}.
	\]
	
	\begin{figure}
		\begin{center}
			\scalebox{0.25}{\includegraphics{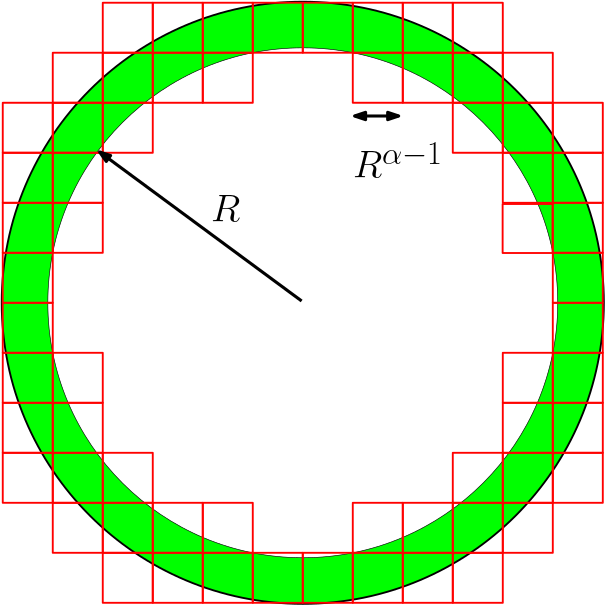}}
		\end{center}
		\caption{The JLM intuition for $1<\alpha<2$: the discrepancy will likely occur in a small annuli (in green) of width $R^{\alpha-1}$ around the boundary of the disk. For $\tfrac{1}{2}<\alpha<1$ the width of the annuli is $\mathcal{O}(1)$.}
		\label{fig:jlm_intuition}
	\end{figure}
	
	That is, the expected number of points in all boundary squares is $\approx R^{2\alpha-2} R^{2-\alpha} = R^{\alpha},$ which is precisely the amount of discrepancy we want to create in the disk $\bD_n(R)$. Assuming that the discrepancy is equally spread among the boundary squares, we conclude from~\eqref{eq:introduction_concentration_general_cube} the probability of having a large discrepancy from each boundary square is roughly (we neglect the constants in front of exponent)
	\[
	\exp\Big(-(R^{2\alpha-2})^2\Big) = \exp\Big(-R^{4\alpha-4}\Big) \, .
	\]
	Combining all the above, while assuming that the point count in the different boundary squares are asymptotically independent, we conclude that
	\begin{align*}
		\bP\Big[\big|\mu_n(\bD_n(R)) - \pi R^2\big| \ge R^\alpha \Big] 
		& \approx \prod_{Q \text{ boundary square }} \bP\Big[ \text{there is large discrepancy in } Q \Big] \\
		& \approx \exp\Big(-R^{4\alpha-4} R^{2-\alpha} \Big) = \exp\Big(-R^{3\alpha-2}\Big)\, .
	\end{align*}
	In practice, the upper bound in the range $\alpha\in(1,2)$ will follow by estimating from above the moment generating function of $\mu_n\big(\bD_n(R)\big)$ is a certain range, while keeping the above intuition in mind, see Section~\ref{sec:UB_large_deviations}. For the lower bound in the range $\alpha\in(1,2)$, we will estimate from below the probability that a certain subset of well-separated boundary squares created the desired discrepancy, while all other boundary squares had a ``typical'' point count, see Section~\ref{sec:LB_large_deviations} for more details.

	For the `moderate deviations' range $\alpha\in(\tfrac{1}{2},1)$, the basic intuition is similar to the range $\alpha\in(1,2)$, only that here we choose $2^{-j}\approx 1/\sqrt{n}$ so there are $\approx R$ boundary squares and the expected number of points per square is of constant order. Therefore, there are typically $\mathcal{O}(R)$ points in all boundary squares, and to create the $R^{\alpha}$ discrepancy in the disk $\bD_n(R)$, one needs to ``choose'' a subset of boundary squares in which there will be a large discrepancy. This extra combinatorial term is the reason we see the different exponent $\varphi(\alpha) = 2\alpha-1$ in this range (in physical language, this corresponds to the \emph{entropy} in the system).
	
	While getting the upper bound in the range $\alpha\in(\tfrac{1}{2},1)$ requires an estimate of the moment generating function of $\mu_n\big(\bD_n(R)\big)$ (see Section~\ref{sec:UB_moderate_deviations}) and is in some sense similar to the upper bound in the range $\alpha\in(1,2)$, the lower bound requires a completely different argument compared to the lower bound in the range $\alpha\in(1,2)$. Indeed, in Section~\ref{sec:LB_moderate_deviations} we reduce the problem of proving the lower bound in the range $\alpha\in(\tfrac{1}{2},1)$ to a question about moderate deviations of independent, bounded, non-degenerate random variables. To tackle the latter, we appeal to a `change of measure' argument, as given by Proposition~\ref{prop:moderate_deviations_lower_bound_bounded_random_variables} (the proof of which in given in Appendix~\ref{sec:moderate_deviations_lower_bound}). 
	
	\subsubsection*{Policy on constants}
	We will use the standard Landau $\mathcal{O}$-notation and the symbol $\lesssim$ interchangeably. Throughout, $C,c>0$ will symbolize (large and small) constants which can depend only on $\beta>0$, but not on $n$ or $R$. Sometimes, to emphasis the dependence in $\beta,$ we will denote such constants by $C_\beta,c_\beta>0$. The value of all constants can change from line to line.
	\subsubsection*{Acknowledgments}
	We would like to thank Sourav Chatterjee, Gaultier Lambert, Thomas Lebl\'e and Mikhail Sodin for very helpful discussions. The research was funded in part by ISF Grant 1903/18. O.Y. is also supported in part by ISF Grant 1288/21 and BSF Grant 202019.
	\section{Preliminaries}
	\label{sec:preliminaries}
	In this section we collect some basic facts on the hierarchical Coulomb gas that will be used throughout the paper. First, we introduce some notations that will be used across sections of this paper. 
	\begin{itemize}
		\item $\bP,\bE,\Var$; the probability, expectation and variance with respect to~\eqref{eq:density_of_HCG_intro};	
		\item $\mu_n$ the random empirical measure of the hierarchical Coulomb gas, as given by~\eqref{eq:def_random_counting_measure_hcg};
		\item $\Delta_n = \mu_n - n \, {\tt Leb}$, where ${\tt Leb}$ is the Lebesgue measure in $[0,1]^2$;
		\item $\bD_n(R) = \bD_n(z,R)$ the disk of radius $R/\sqrt{n}$, centered around a point $z\in (0,1)^2$;
		\item $\mathcal{D}_{k}$ all dyadic sub-squares of $[0,1]^2$ with side lengths $2^{-k}$;
		\item $\mathcal{F}_k$ the sigma-algebra generated by $\big\{ \mu_n(Q) : Q\in \mathcal{D}_k \big\}$;
		\item $\mathcal{U}_k,\mathcal{V}_k$ the maximal and boundary squares (respectively) of order $k$ for $\bD_n(R)$, see Definition~\ref{definition:maximal_and_boundary_cubes};
		\item $p(Q)$ the relative area of $\bD_n(R)$ inside a dyadic square $Q$, that is 
		\[
		p(Q) = \frac{{\tt Leb}\big(\bD_n(R) \cap Q\big)}{{\tt Leb}(Q)} \, .
		\]
		
		\item $\ell\ge 1$ is an integer such that 
		\[
		\frac{R}{\sqrt{n}} \le 2^{-\ell} < 2\, \frac{R}{\sqrt{n}} \, .
		\]
		We will also use in specific locations several different dyadic scales, which we will denote by $j_1,j_2,\ldots$ etc., depending on the order at which they appear. 
	\end{itemize}  
	\subsection{Basic facts from\texorpdfstring{~\cite{chatterjee}}{Cha19}}
	For $n\ge 2$, let $\Sigma_n$ denote the set of all tuples of $n$ points from $[0,1]^2$. Recall that the energy of a configuration $(x_1,\ldots,x_n)\in\Sigma_n$ is given by
	\begin{equation}
		\label{eq:def_of_hamiltonian}
		\mathcal{H}_n(x_1,\ldots,x_n) = \sum_{1\le i <j\le n} w(x_i,x_j) \, ,
	\end{equation}
	where $w(x,y) = \min \{ k \ge 1 : x,y \ \text{belong to distinct sub-squares in } \mathcal{D}_k\}$. For $\beta>0$, let $\nu_{n,\beta}$ denote the probability measure on $\Sigma_n$ with the density
	\[
	\frac{1}{Z(n,\beta)} e^{-\beta \mathcal{H}_n(x_1,\ldots,x_n)}
	\] 
	with respect to Lebesgue measure on $\Sigma_n$. Throughout the paper, we will refer to the point process $\nu_{n,\beta}$ as the \emph{hierarchical Coulomb gas} with $n$ points and inverse temperature $\beta>0$. Let $(p_1,\ldots,p_n)\in \Sigma_n$ be a realization of the point process $\nu_{n,\beta}$, and denote by
	\[
	\mathcal{P}(Q) = \{p_j : \, p_j \in Q\}
	\]
	the random points which fall inside $Q\in \bigcup_{k\ge 1}\mathcal{D}_k$. We will denote the random counting measure associated with $(p_1,\ldots,p_n)$ by
	\begin{equation}
		\label{eq:def_random_counting_measure_hcg}
		\mu_n = \sum_{j=1}^{n} \delta_{p_j} \, .
	\end{equation}
	Clearly $\mu_n(Q) = |\mathcal{P}(Q)|$, and by exchangeability of the point count in dyadic squares we have $\bE\big[\mu_n(Q)\big] = n/4^k$ for all $Q\in \mathcal{D}_k$. Using the linearity of the expectation, the last observation implies that for any Borel set $U\subset[0,1]^2$, we have
	\begin{equation*}
		\label{eq:formula_for_expectation}
		\bE\big[\mu_n(U)\big] = n \, {\tt Leb}(U) \, .
	\end{equation*} 
	The following proposition is the main reason as to why the hierarchical structure helps with the analysis. For the reader's convenience, we provide the proof below.
	\begin{proposition}
		[{\cite[Lemma~3.6]{chatterjee}}]	\label{prop:self_similarity_and_conditional_independence}
		For each $k\ge 1$, conditional on $\mathcal{F}_k$ the random sets $\{ \mathcal{P}(Q) : \, Q\in \mathcal{D}_k\}$ are mutually independent. Furthermore, conditional on $\mathcal{F}_k$, for each $Q\in \mathcal{D}_k$ the random set $\mathcal{P}(Q)$ has the same law as the scaled version of $\nu_{\mu_n(Q),\beta}$.
	\end{proposition}
		\begin{proof}
		By the definition of the potential $w(x,y)$, we can rewrite the energy~\eqref{eq:def_of_hamiltonian} of a given configuration $(x_1,\ldots,x_n)\in \Sigma_n$ as
		\begin{equation}
			\label{eq:alternative_form_for_the_hamiltonian}
			\mathcal{H}_n(x_1,\ldots,x_n) = \binom{n}{2} + \sum_{j=1}^{\infty} \sum_{Q\in \mathcal{D}_j} \binom{\mu_n(Q)}{2}
		\end{equation}
		with the convention that $\binom{0}{2} = \binom{1}{2} = 0$. Indeed, each pair of points contributes one unit of energy for each dyadic square that contains both points.
		With~\eqref{eq:alternative_form_for_the_hamiltonian}, we can write for all $k\ge 1$
		\begin{equation}
			\label{eq:hamiltonian_with_conditioning}
			\mathcal{H}_n(x_1,\ldots,x_n) = C_k(x_1,\ldots,x_n) + \sum_{Q \in \mathcal{D}_k} \Big[\sum_{\substack{i<j \\ x_i,x_j\in \mathcal P(Q)}} w(x_i,x_j) \, \Big]
		\end{equation}
		with $C_k$ measurable with respect to $\mathcal{F}_{k-1}$. Now, conditional on $\mathcal{F}_k$, the distribution of $$\big\{\mu_n(Q) \, : \, Q\in \bigcup_{j=k+1}^{\infty} \mathcal{D}_j\big\}$$ depends only on $\big\{\mu_n(Q) \, : \, Q \in \mathcal{D}_k \big\}.$ Hence, in view of~\eqref{eq:hamiltonian_with_conditioning}, the conditional density of $\{ \mathcal{P}(Q) : \, Q \in \mathcal{D}_k\}$ is proportional to
		\begin{equation*}
			\prod_{Q \in \mathcal{D}_k} \exp\bigg(-\beta \sum_{\substack{i<j \\ x_i,x_j\in \mathcal P(Q)}} w(x_i,x_j) \bigg)\, .
		\end{equation*} 
		The conditional independence property is evident from the above. Furthermore, for each $Q\in \mathcal{D}_k$ we observe that conditional on $\mathcal{F}_k$, the marginal density of $\mathcal{P}(Q)$ is proportional to
		\[
		\exp\bigg(-\beta \sum_{\substack{i<j \\ x_i,x_j\in \mathcal P(Q)}} w(x_i,x_j) \bigg) \, , 
		\]
		which, since $\mu_n(Q) = |\mathcal P(Q)|$, has exactly the same law as a scaled version of $\nu_{\mu_n(Q),\beta}$.
	\end{proof}
	We will also need a basic estimate on the partition function $Z(n,\beta)$, which follows from the proof of~\cite[Lemma 3.3]{chatterjee}. Again, for the reader's convenience, we provide the proof below.
	\begin{claim}
		\label{claim:estimates_on_partition_function}
		There exists $C_\beta>0$ so that for all $n\ge 2$
		\[
		\Big|\log Z(n,\beta) + \frac{2\beta}{3} n^2\Big| \le C_\beta \, n\log n\, .
		\]
	\end{claim}
	\begin{proof}
		Note that
		\[
		Z(n,\beta) = \int_{[0,1]^{2n}} e^{-\beta \mathcal{H}_n(x_1,\ldots,x_n)} \, {\rm d} \, {\tt Leb}(x_1)\cdots {\rm d} \, {\tt Leb}(x_n)
		\]
		with $x_1,\ldots,x_n\in[0,1]^2$. To bound $Z(n,\beta)$ from above, denote by $k\ge 1$ the maximal integer such that $2^k \le \sqrt{n}$. Then, for all $(x_1,\ldots,x_n)\in\Sigma_n$, we have
		\begin{align}
			\label{eq:estimate_on_partition_function_lower_bound_on_hamiltonian} \nonumber
			\mathcal{H}_{n}(x_1,\ldots,x_n) &\stackrel{\eqref{eq:alternative_form_for_the_hamiltonian}}{\ge}  \binom{n}{2} + \sum_{j=1}^{k} \sum_{Q\in \mathcal{D}_j} \binom{\mu_n(Q)}{2} \\ & = \binom{n}{2} +\frac{1}{2} \sum_{j=1}^{k} \sum_{Q\in \mathcal{D}_j} \big(\mu_n(Q)^2 - \mu_{n}(Q)\big) =  \binom{n}{2} -\frac{nk}{2} +\frac{1}{2} \sum_{j=1}^{k} \sum_{Q\in \mathcal{D}_j} \mu_n(Q)^2 \, .  
		\end{align}
		By Cauchy-Schwarz inequality, we have
		\[
		\sum_{Q\in \mathcal{D}_j} \mu_n(Q)^2 \ge \frac{1}{|\mathcal{D}_j|} \bigg(\sum_{Q\in \mathcal{D}_j} \mu_n(Q)\bigg)^2 = \frac{n^2}{4^j} \, .
		\]
		Plugging into~\eqref{eq:estimate_on_partition_function_lower_bound_on_hamiltonian}, we see that for all $x_1,\ldots,x_n\in[0,1]^2$ 
		\[
		\mathcal{H}_n(x_1,\ldots,x_n) \ge \binom{n}{2} -\frac{nk}{2} +\frac{n^2}{2} \sum_{j=1}^{k} 4^{-j} = \frac{2n^2}{3} -\frac{nk}{2} -\frac{n}{2} -\frac{n^2}{6 \times 4^k} 
		\]
		which, by our choice of $k$, implies that
		\begin{equation*}
			Z(n,\beta) \le \exp\Big(-\frac{2\beta}{3}n^2 + C_\beta \, n \log n \Big)
		\end{equation*}
		for all $n\ge 2$. To get the complementary lower bound, we apply Jensen's inequality to get
		\[
		Z(n,\beta) \ge \exp\bigg(-\beta \int_{[0,1]^{2n}} \mathcal{H}_n(x_1,\ldots,x_n) \, {\rm d} \, {\tt Leb}(x_1)\cdots {\rm d} \, {\tt Leb}(x_n)\bigg)\, .
		\]
		We have
		\begin{align*}
			\int_{[0,1]^{2n}}      \mathcal{H}_n(x_1,\ldots,x_n) \, {\rm d} \, {\tt Leb}(x_1)\cdots {\rm d} \, {\tt Leb}(x_n)
			& = \int_{[0,1]^{2n}} \sum_{1\le i <j\le n} w(x_i,x_j) \, {\rm d} \, {\tt Leb}(x_1)\cdots {\rm d} \, {\tt Leb}(x_n) \\
			& = \binom{n}{2} \iint_{[0,1]^2\times[0,1]^2} w(x,y) \, {\rm d} \, {\tt Leb}(x) \, {\rm d} \, {\tt Leb}(y) \, .
		\end{align*}
		It is straightforward to check that (see~\cite[Lemma~3.1]{chatterjee}) 
		\[
		\iint_{[0,1]^2\times[0,1]^2} w(x,y) \, {\rm d} \, {\tt Leb}(x) \, {\rm d} \, {\tt Leb}(y) = 3\sum_{j=1}^{\infty} \frac{j}{4^{j}} = \frac{4}{3}\, ,
		\]
		which implies that
		\[
		Z(n,\beta) \ge \exp\Big(- \frac{4\beta}{3} \, \binom{n}{2}  \Big) = \exp\Big(-\frac{2\beta}{3} n^2 + \frac{2\beta}{3}n\Big)\, .
		\]
		Combining the upper and lower bound, the claim follows.
	\end{proof}
	To conclude this section, we cite another claim from~\cite{chatterjee}. 
	\begin{claim}[{\cite[Corollary~3.4]{chatterjee}}]
		\label{claim:bound_on_ratio_of_partitions}
		For all $n\ge 1$, $\beta>0$ and $k\ge -n+1$ we have
		\[
		\frac{Z(n+k,\beta)}{Z(n,\beta)} \le \exp\bigg(-\frac{4}{3} \beta n k - \frac{2}{3}\beta k(k-1) + C^\prime \beta|k|\log(n+|k|+1)\bigg)
		\]
		where $C^\prime$ is some positive absolute constant.
	\end{claim}
	\subsection{Discrepancy at top level of the tree}
	Towards the proof of Theorem~\ref{thm:Micro_JLM}, the first key observation is that the fluctuations in dyadic squares has sub-Gaussian tail (recall~\eqref{eq:introduction_concentration_general_cube} from the introduction). We start by proving this observation for squares at the top level $\mathcal{D}_1$. Throughout, let $Q_1,\ldots,Q_4$ be the four dyadic squares in $\mathcal{D}_1$ and denote by $X_j = \Delta_n(Q_j)$ the corresponding discrepancy for $j=1,\ldots,4$. We note that the support of the random vector $(X_1,\ldots,X_4)$ is given by
	\begin{equation}
		\label{eq:support_of_top_level_disc}
		\text{Supp}({\bf X}) = \Big\{(k_1,\ldots,k_4) \in \big[\tfrac{n}{4},\tfrac{3n}{4}\big] \cap \big(\bZ-\tfrac{n}{4}\big) \, : \, k_1+\ldots+k_4 = 0 \Big\} \, .
	\end{equation}
	\begin{lemma}
		\label{lemma:tail_probability_for_disc_at_top_level}
		There exist an absolute constant $C>0$ so that for all $n\ge 1$ and for all $(k_1,\ldots,k_4)\in\text{\normalfont Supp}({\bf X})$ we have
		\begin{equation*}
			\bP\Big[\big(X_1,\ldots,X_4\big) = \big(k_1,\ldots,k_4\big)\Big] \le \exp\Big(-\frac{2\beta}{3}\sum_{j=1}^{4} k_j^2 + C\beta \log(n+1) \sum_{j=1}^4 |k_j|\Big)\, .
		\end{equation*}
	\end{lemma}
	\begin{proof}
		The starting point of the proof is the following recursive formula for the probability
		\begin{align}
			\nonumber
			\label{eq:probability_of_disc_top_level_equality}
			\bP\Big[\big(X_1,&\ldots,X_4\big) = \big(k_1,\ldots,k_4\big)\Big] \\ &= 4^{-n}e^{-\beta\binom{n}{2}} \frac{n!}{(\frac{n}{4}+k_1)!(\frac{n}{4}+k_2)! (\frac{n}{4}+k_3)! (\frac{n}{4}+k_4)!} \frac{\prod_{j=1}^{4}Z\big(\frac{n}{4} + k_j,\beta\big)}{Z\big(n,\beta\big)} \, .
		\end{align}
		Indeed, denoting by $t_j = \frac{n}{4} + k_j$ for $j=1,\ldots,4$, Proposition~\ref{prop:self_similarity_and_conditional_independence} implies that 
		\begin{align*}
			\bP\Big[\big(X_1,&\ldots,X_4\big) = \big(k_1,\ldots,k_4\big)\Big] \\ &= \frac{1}{Z(n,\beta)} \int_{[0,1]^{2n}} \mathbf{1}_{\big\{(X_1,\ldots,X_4 = (k_1,\ldots,k_4)\big\}} e^{-\beta \mathcal{H}_n(x_1,\ldots,x_n)} \, {\rm d} \, {\tt Leb}(x_1)\cdots {\rm d} \, {\tt Leb}(x_n) \\ &= \frac{n!}{t_1!t_2! t_3! t_4!} e^{-\beta\binom{n}{2}} \frac{1}{Z(n,\beta)} \prod_{j=1}^{4} \bigg( \int_{[0,\frac{1}{2}]^{2t_j}} e^{-\beta \mathcal{H}_{t_j}(x_1,\ldots,x_{t_j})} \, {\rm d} \, {\tt Leb}(x_1)\cdots {\rm d} \, {\tt Leb}(x_{t_j}) \bigg) \, .
		\end{align*}
		By changing variables $x_j \mapsto 2x_j$ we see that
		\[
		\int_{[0,\frac{1}{2}]^{2t_j}} e^{-\beta \mathcal{H}_{t_j}(x_1,\ldots,x_{t_j})} \, {\rm d} \, {\tt Leb}(x_1)\cdots {\rm d} \, {\tt Leb}(x_{t_j}) = 4^{-t_j} Z(t_j,\beta) 
		\]
		and since $t_1+\ldots+t_4 = n$, the equality~\eqref{eq:probability_of_disc_top_level_equality} follows. Now, for integers $t_1,\ldots,t_4\ge 0$ such that $t_1+\ldots+t_4= n$, we denote the multinomial coefficient by
		\[
		\psi(t_1,\ldots,t_4) = \frac{n!}{t_1! \, t_2! \, t_3 ! \,  t_4 !}\, .
		\]
		Let $m_1,\ldots,m_4$ be non-negative integers such that $m_1+\ldots+m_4 = n$ and $|m_j - \frac{n}{4}| \le 1$ for all $j=1,\ldots,4$. As the multinomial coefficient is maximized when $t_j = m_j$, we see that
		\[
		\frac{\psi\big(t_1 , t_2 , t_3, t_4 \big)}{\psi\big(m_1,m_2,m_3,m_4\big)} \le 1 
		\]
		for all $(k_1,\ldots,k_4)\in \text{Supp}(\bf{X})$. This, together with~\eqref{eq:probability_of_disc_top_level_equality}, implies that
		\begin{align*}
			\bP\Big[\big(X_1,\ldots,X_4\big) = & \big(k_1,\ldots,k_4\big)\Big] \\ &\le \frac{\bP\Big[\big(X_1,\ldots,X_4\big) = \big(k_1,\ldots,k_4\big)\Big]}{\bP\Big[\big(X_1,\ldots,X_4\big) = \big(m_1 - \frac{n}{4},\ldots,m_4 - \frac{n}{4}\big)\Big]} \\ & \le \frac{\psi\big(\frac{n}{4} + k_1 , \frac{n}{4} + k_2 , \frac{n}{4} + k_3, \frac{n}{4} + k_4 \big)}{\psi\big(m_1,m_2,m_3,m_4\big)} \, \prod_{j=1}^{4} \frac{Z\big(\frac{n}{4} + k_j,\beta\big)}{Z\big(m_j ,\beta\big)} \le \prod_{j=1}^{4} \frac{Z\big(\frac{n}{4} + k_j,\beta\big)}{Z\big(m_j ,\beta\big)}\, .
		\end{align*}
		To bound the product above, we apply Claim~\ref{claim:bound_on_ratio_of_partitions} to each of its element. Together with the fact that $\sum_{j=1}^{4} k_j = 0$, we arrive at the inequality
		\begin{align*}
			\bP\Big[&\big(X_1,\ldots,X_4\big) = \big(k_1,\ldots,k_4\big)\Big] \\ &\le \prod_{j=1}^{4} \frac{Z\big(\frac{n}{4} + k_j,\beta\big)}{Z\big( m_j ,\beta\big)}
			\le \exp\bigg(-\frac{\beta n}{3}\sum_{j=1}^{4}k_j - \frac{2\beta}{3}\sum_{j=1}^4 k_j(k_j-1) + C^\prime \beta \sum_{j=1}^4 |k_j| \log\big(\frac{n}{4} + |k_j|+1\big) \bigg) \\
			& \le \exp\bigg(-\frac{2\beta}{3} \sum_{j=1}^4 k_j^2 + C^\prime \beta \log(n+1) \sum_{j=1}^4 |k_j|\bigg)
		\end{align*}
		as desired.
	\end{proof}
	Using Lemma~\ref{lemma:tail_probability_for_disc_at_top_level}, we obtain upper bounds on the moment generating function of the discrepancy at the top level of the tree. The following lemma would be essential in proving the upper bound in Theorem~\ref{thm:Micro_JLM} for the range $1<\alpha<2$, see Section~\ref{sec:UB_large_deviations} below.
	\begin{lemma}
		\label{lemma:moment_generating_funtion_upper_bound_top_level_no_assumption}
		There exist $C_\beta>0$ so that for all $n\ge 2$ and $\boldsymbol{\la} = (\lambda_1,\ldots,\la_4)\in \bR^4$ we have
		\begin{equation*}
			\log\bE\bigg[\exp\Big(\sum_{j=1}^{4} \la_j X_j \Big)\bigg] \le  C_\beta \big(1 + \| \boldsymbol{\la} \|^2\big)\log^2(n+1) \, .
		\end{equation*} 
	\end{lemma}
	To prove Lemma~\ref{lemma:moment_generating_funtion_upper_bound_top_level_no_assumption}, we will need the following simple claim, which will also be helpful later on.
	\begin{claim}
		\label{claim:bound_on_MGF_gaussian_and_folded}
		Let $a>0$ and $b_1,b_2\in \bR$, then we have
		\[
		\sum_{k\in \mathbb{Z}} \exp\Big(-ak^2+b_1 k + b_2 |k| \Big) \le 1+ \frac{C}{\sqrt{a}} \exp\Big(\frac{1}{4a} \big(|b_1| + |b_2|\big)^2\Big) \, , 
		\]
	\end{claim}
	where $C>0$ is some absolute constant.
	\begin{proof}
		We start by proving the bound
		\begin{equation}
			\label{eq:bound_on_MGF_gaussian_and_folded_one_side}
			\sum_{k\ge 1} \exp\Big(-ak^2+(b_1 + b_2)k \Big) \le \frac{C}{\sqrt{a}} \exp\Big(\frac{(b_1+b_2)^2}{4a}\, \Big) \, .
		\end{equation}
		Indeed, we have the simple bound
		\begin{align*}
			\sum_{k\ge 1} \exp\Big(-ak^2+b_1|k| + b_2 k \Big) \le \exp\Big(-\min_{k\ge 1} \Big[ \, \frac{a}{2}k^2 - (b_1+b_2) k\Big]\Big) \times\Big(\sum_{k\ge 1} e^{-ak^2/2}\Big) \, .
		\end{align*}
		A computation shows that
		\[
		\min_{k\ge 1} \Big[ \, \frac{a}{2}k^2 - (b_1+b_2) k\Big] \ge -\frac{(b_1+b_2)^2}{4a} \, ,
		\]
		and furthermore
		\[
		\sum_{k\ge 1} e^{-ak^2/2} \le \int_{0}^{\infty} e^{-ax^2/2} \, {\rm d}x \lesssim \frac{1}{\sqrt{a}} 
		\]
		which gives~\eqref{eq:bound_on_MGF_gaussian_and_folded_one_side}. Splitting the sum to positive and negative integers and using~\eqref{eq:bound_on_MGF_gaussian_and_folded_one_side} we get
		\begin{align*}
			\sum_{k\in \mathbb{Z}} \exp\Big(-ak^2+b_1|k| + b_2 k \Big) &\le 1 + \frac{C}{\sqrt{a}} \Big(\exp\Big(\frac{(b_1+b_2)^2}{4a}\, \Big) + \exp\Big(\frac{(b_1-b_2)^2}{4a}\, \Big)\Big) \\ &= 1 + \frac{C}{\sqrt{a}} \exp\Big(\frac{b_1^2 + b_2^2}{4a}\Big) \times\Big(e^{\frac{b_1b_2}{2a}} + e^{-\frac{b_1b_2}{2a}}\Big) \\ & \le 1+ \frac{C}{\sqrt{a}} \exp\Big(\frac{b_1^2 + b_2^2}{4a} + \frac{|b_1 b_2|}{2a}\Big)
		\end{align*}
		as desired.
	\end{proof}
	\begin{proof}[Proof of Lemma~\ref{lemma:moment_generating_funtion_upper_bound_top_level_no_assumption}]
		To start, we apply Lemma~\ref{lemma:tail_probability_for_disc_at_top_level} and see that
		\begin{align}
			\label{eq:mgf_upper_bound_top_level_no_assumption_after_tail_estimate}
			\nonumber
			\bE\Big[\exp\big(\sum_{j=1}^{4}\la_j X_j\big)\Big] &= \sum_{(k_1,\ldots,k_4)\in \text{Supp}(\bf{X})} \exp\Big(\sum_{j=1}^4 \la_j k_j\Big) \, \bP\Big[\big(X_1,\ldots,X_4\big) = \big(k_1,\ldots,k_4\big)\Big] \\
			& \le \sum_{(k_1,\ldots,k_4)\in \text{Supp}(\bf{X})}\prod_{j=1}^{4}\exp\Big(\la_j k_j-\frac{2\beta}{3}k_j^2 + C^\prime \beta \log(n+1)|k_j| \Big) \, .
		\end{align}
		Recalling that $\text{Supp}(\bf{X})$ is given by~\eqref{eq:support_of_top_level_disc}, we change variables $k\mapsto 4k$ in the sum~\eqref{eq:mgf_upper_bound_top_level_no_assumption_after_tail_estimate} and interchange the sum and product to get the bound 
		\begin{equation}
			\label{eq:mgf_upper_bound_top_level_no_assumption_after_change_of_variables}
			\bE\Big[\exp\big(\sum_{j=1}^{4}\la_j X_j\big)\Big] \le  \prod_{j=1}^{4} \Bigg(\sum_{k\in \bZ} \exp\Big(-\frac{\beta}{24}k^2 + \frac{\la_j}{4} k + \frac{1}{4}C^\prime \beta \log(n+1)|k| \Big)\Bigg) \, .
		\end{equation}
		To conclude the desired bound, we apply Claim~\ref{claim:bound_on_MGF_gaussian_and_folded} with 
		\[
		a = \frac{\beta}{24} \, , \quad  b_1 = \frac{\lambda_j}{4} \, , \quad  b_2= \frac{1}{4}C^\prime \beta\log(n+1) \, ,
		\]
		and use the elementary inequality $(1+\la)^2\le 2(1+\la^2)$ to get that
		\begin{equation*}
			\sum_{k\in \bZ}\exp\Big(-\frac{\beta}{24}k^2 + \frac{\la_j}{4} k + \frac{1}{4}C^\prime \beta \log(n+1)|k| \Big) \le \exp\Big(C_\beta \,(1 + \la_j^2) \log^2(n+1)\Big)
		\end{equation*}
		for some $C_\beta>0$. Plugging this into~\eqref{eq:mgf_upper_bound_top_level_no_assumption_after_change_of_variables} we get that
		\[
		\bE\Big[\exp\big(\sum_{j=1}^{4}\la_j X_j\big)\Big] \le \prod_{j=1}^{4} \exp\Big(C_\beta \,(1 + \la_j^2) \log^2(n+1)\Big) \le \exp\Big(C_\beta \big(1 + \| \boldsymbol{\la} \|^2\big) \log^2(n+1) \Big)
		\]
		which is what we wanted.
	\end{proof}
	The bound in Lemma~\ref{lemma:moment_generating_funtion_upper_bound_top_level_no_assumption} holds for all $\boldsymbol{\la} \in \bR^4$ and will be used for $\| \boldsymbol{\la} \|$ large. To prove the upper bound in Theorem~\ref{thm:Micro_JLM} for the range $\tfrac{1}{2} < \alpha<1$ (which we do in Section~\ref{sec:UB_moderate_deviations}), we will also need a similar bound with $\| \boldsymbol{\la} \|$ small, as given by the following lemma.
	\begin{lemma}
		\label{lemma:upper_bound_on_MGF_top_level_small_input}
		There exist $C_\beta>0$ so that for any $n\ge 3$ and $\boldsymbol{\la} = (\la_1,\ldots,\la_4)\in \bR^4$ with  \(
		\|\boldsymbol{\la}\|\le (\log n)^{-2}\) we have
		\begin{equation*}
			\log\bE\bigg[\exp\Big(\sum_{j=1}^{4} \la_j X_j\Big)\bigg] \le  C_\beta \log^4(n+1)\|\boldsymbol{\la} \|^2 \, .
		\end{equation*} 
	\end{lemma}
	\begin{proof}
		Let us denote the moment generating function by $f_n:\bR^4 \to \bR$, namely,
		\begin{equation*}
			f_n(\boldsymbol{\la}) = \bE\bigg[\exp\Big(\sum_{j=1}^{4} \la_j X_j\Big)\bigg]\, .
		\end{equation*}
		Then $f_n(\boldsymbol{0}) = 1$, where $\boldsymbol{0} = (0,0,0,0)$ and $f_n$ is twice differentiable near the origin. We have
		\[
		\frac{\partial f_n}{\partial \la_j} (\boldsymbol{0}) = \bE\big[X_j\big] = 0
		\]
		for all $j=1,\ldots,4$, and so
		\[
		\big|f_n(\boldsymbol{\la}) - 1\big|\le  C \|\boldsymbol{\la}\|^2 \max_{\substack{i=1,\ldots,4 \\ \ell=1,\ldots,4}} \max_{\|\boldsymbol{t} \|\le a} \Big|\frac{\partial^2 f_n}{\partial \la_i \partial \la_\ell} (\boldsymbol{t})\Big| 
		\]
		with $a = 2\, (\log n)^{-2}$. By the Cauchy-Schwarz inequality and the symmetry between the squares, we also have that
		\begin{align*}
			\Big|\frac{\partial^2 f_n}{\partial \la_i \partial \la_\ell} (t_1,\ldots,t_4)\Big| &= \Big|\bE\Big[X_i X_\ell \, e^{\sum_{j=1}^{4}t_j X_j}\Big]\Big| \le \bE\Big[X_1^2 \, e^{\sum_{j=1}^{4}t_j X_j}\Big]\, .
		\end{align*}
		In view of the above, the lemma would follow once we show that
		\begin{equation}
			\label{eq:bound_on_the_second_derivative_mgf_top_level}
			\max_{\|\boldsymbol{t} \|\le a} \bE\Big[X_1^2 \exp\Big(\sum_{j=1}^{4}t_j X_j\Big)\Big] \le C_\beta \log^4(n+1) \, . 
		\end{equation}
		Indeed, consider the event
		\begin{equation*}
			\mathcal{E} = \Big\{ \max_{j=1,\ldots,4} |X_j| \ge \log^2 (n+1) \Big\}
		\end{equation*}
		and split the expectation in~\eqref{eq:bound_on_the_second_derivative_mgf_top_level} according to
		\begin{equation*}
			\bE\Big[X_1^2 \exp\Big(\sum_{j=1}^{4}t_j X_j\Big)\Big] = \bE\Big[\big(\cdots\big) \mathbf{1}_{\mathcal{E}}\Big] + \bE\Big[\big(\cdots\big) \mathbf{1}_{\mathcal{E}^c}\Big] \, .
		\end{equation*}
		On $\mathcal{E}^c$, we apply the naive bound
		\[
		\max_{\|t\|\le a}\bE\Big[X_1^2 \exp\Big(\sum_{j=1}^{4}t_j X_j\Big)\mathbf{1}_{\mathcal{E}^c}\Big] \le \log^4(n+1) e^{4 a \log^2(n+1) } \le e^8\log^4(n+1)
		\]
		so it remains to show that
		\begin{equation}
			\label{eq:bound_on_the_second_derivative_mgf_top_level_good_event}
			\max_{\|t\|\le a}\bE\Big[X_1^2 \exp\Big(\sum_{j=1}^{4}t_j X_j\Big)\mathbf{1}_{\mathcal{E}}\Big] \le C_\beta \, .
		\end{equation}
		By definition, we have
		\begin{equation*}
			\bE\Big[X_1^2 \exp\Big(\sum_{j=1}^{4}t_j X_j\Big)\mathbf{1}_{\mathcal{E}}\Big] = \sum_{\substack{(k_1,\ldots,k_4)\in \text{Supp}(\bf{X}) \\ \max_{j}|k_j|\ge \log^2(n+1)}} k_1^2  \, e^{\sum_{j=1}^{4} k_j t_j} \, \bP\Big[\big(X_1,\ldots,X_4\big) = \big(k_1,\ldots,k_4\big)\Big] \, .
		\end{equation*}
		Furthermore, for $\|t\| \le (\log n)^{-2}$, Lemma~\ref{lemma:tail_probability_for_disc_at_top_level} implies the bound
		\begin{align*}
			e^{\sum_{j=1}^{4} k_j t_j} \, \bP\Big[\big(X_1,\ldots,X_4\big) &= \big(k_1,\ldots,k_4\big)\Big] \\ & \le \exp\bigg((\log n)^{-2} \, \sum_{j=1}^{4} |k_j|-\frac{2\beta}{3}\sum_{j=1}^{4} k_j^2 + C\beta \log(n+1) \sum_{j=1}^4 |k_j|\bigg) \\ &\le \exp\bigg(-\frac{2\beta}{3}\sum_{j=1}^{4} k_j^2 + C\beta \log(n+1) \sum_{j=1}^4 |k_j|\bigg) \, .
		\end{align*}
		Combining these observations together, we arrive at
		\begin{equation*}
			\max_{\|\boldsymbol{t}\|\le a} \,\bE\Big[X_1^2 \exp\Big(\sum_{j=1}^{4}t_j X_j\Big)\mathbf{1}_{\mathcal{E}}\Big] \le \sum_{\substack{(k_1,\ldots,k_4)\in \text{Supp}(\bf{X}) \\ \max_{j}|k_j|\ge \log^2(n+1)}} k_1^2 \exp\bigg(-\frac{2\beta}{3}\sum_{j=1}^{4} k_j^2 + C\beta \log(n+1) \sum_{j=1}^4 |k_j|\bigg) \, .
		\end{equation*}
		As the number of terms in the sum is $\lesssim n^4$, and since
		\[
		\max_{k\in \bZ} \Big[C \beta \log(n+1) |k| - \frac{\beta}{3}k^2 \Big] \le C_\beta \log^2(n+1)\, ,
		\]
		we finally arrive at the inequality 
		\[
		\max_{\|\boldsymbol{t}\|\le a} \,\bE\Big[X_1^2 \exp\Big(\sum_{j=1}^{4}t_j X_j\Big)\mathbf{1}_{\mathcal{E}}\Big] \le Cn^6 \exp\Big(-c_\beta \log^4(n+1)\Big)\, .
		\]
		In particular, this observation proves~\eqref{eq:bound_on_the_second_derivative_mgf_top_level_good_event}, and hence relation~\eqref{eq:bound_on_the_second_derivative_mgf_top_level} holds. Altogether, we have shown that
		\[
		\big|f_n(\boldsymbol{\la}) - 1\big|\le  C_\beta  \|\boldsymbol{\la}\|^2 \log^4(n+1) \, .
		\]
		The lemma follows from the elementary inequality $1+x\le e^x$, valid for all $x>0$.
	\end{proof}
	Finally, we conclude this section with another bound on a moment generating function, this time for the discrepancy squared.
	\begin{claim}
		\label{claim:upper_bound_on_moment_generating_function_disc_squared_top_level}
		There exist constants $C_\beta,c_\beta>0$ so that for all $n\ge 2$ and for all $\boldsymbol{s}=(s_1,\ldots,s_4) \in \bR^4$ with $\|\boldsymbol{s}\|< \beta/2$ we have
		\[
		\log \bE \bigg[\exp \Big(\sum_{j=1}^{4} s_j X_j^2\Big)\bigg] \le C_\beta \log^2(n+1) \, .
		\]
	\end{claim}
	\begin{proof}
		Applying Lemma~\ref{lemma:tail_probability_for_disc_at_top_level} once more, we see that
		\begin{align*}
			\bE\Big[ e^{\sum_{j=1}^{4} s_j X_j^2} \Big]  & = \sum_{(k_1,\ldots,k_4)\in \text{Supp}(\bf{X})} e^{\sum_{j=1}^{4} s_j k_j^2} \, \bP\Big[\big(X_1,\ldots,X_4\big) = \big(k_1,\ldots,k_4\big)\Big] \\
			& \le \sum_{(k_1,\ldots,k_4)\in \text{Supp}(\bf{X})} \exp \Big(\sum_{j=1}^{4} \big(s_j -\frac{2\beta}{3} \big)k_j^2  + C_\beta \log(n+1) \sum_{j=1}^4 |k_j|\Big)\, .
		\end{align*}
		By our assumption, $|s_j| < \beta/2$ for all $j=1,\ldots,4$, and we get that
		\begin{align*}
			\bE\Big[ e^{\sum_{j=1}^{4} s_j X_j^2} \Big] &\le \sum_{(k_1,\ldots,k_4)\in \text{Supp}(\bf{X})} \exp \Big(-\frac{\beta}{6}\sum_{j=1}^{4} k_j^2  + C_\beta \log(n+1) \sum_{j=1}^4 |k_j|\Big) \\ & \le \prod_{j=1}^{4} \sum_{k_j\in \bZ} \exp \Big(-\frac{\beta}{96} k_j^2  + C_\beta \log(n+1) |k_j|\Big)\, .
		\end{align*}
		The desired bound follows immediately from Claim~\ref{claim:bound_on_MGF_gaussian_and_folded}, applied with
		\[
		a=\frac{\beta}{96}\, , \quad b_1 = 0\, , \quad \text{and }\ b_2= C_\beta \log(n+1).
		\]
	\end{proof}
	\subsection{Sub-Gaussian tails and overcrowding in dyadic squares}
	As already discussed in the introduction (see~\eqref{eq:introduction_concentration_general_cube} therein), a key step towards the proof of Theorem~\ref{thm:Micro_JLM} is to show that the discrepancy in dyadic squares has sub-Gaussian tails, all the way down to the microscopic scale. As we already obtained such a concentration result for the squares in $\mathcal{D}_1$ (see Lemma~\ref{lemma:tail_probability_for_disc_at_top_level} above), now we show how this strong concentration ``propagates down'' the tree using induction and the (conditional) self-similarity in law, given by Proposition~\ref{prop:self_similarity_and_conditional_independence}. We remark that a similar strategy was recently applied in~\cite[Proposition 2.2]{Thoma} to obtain upper bounds on the extreme overcrowding probability, all the way down to the microscopic scale.
	\begin{lemma}
		\label{lemma:concentration_on_dyadic_cubes}
		For all $\kappa\in(0,1)$ and for all $L\ge L_0(\beta,\kappa)$ large enough there exists $c_{\kappa,\beta}>0$ such that the following holds. Let $j\ge 1$ be such that $n = 4^j L$, then for all $Q\in \mathcal{D}_j$ we have 
		\[
		\bP\big[|\Delta_n(Q)| \ge L^\kappa\big] \le \exp\Big(-c_{\kappa,\beta} L^{2\kappa}\Big) \, .
		\]
	\end{lemma}
	Note that the choice of $j\ge 1$ in Lemma~\ref{lemma:concentration_on_dyadic_cubes} implies that $\bE[\mu_n(Q)] = L$, so Lemma~\ref{lemma:concentration_on_dyadic_cubes} shows that on large dyadic squares (on a microscopic scale) the tails for polynomial deviations are sub-Gaussian.
	\begin{proof}
		Let $a\in (1,4)$ be such that $a = 4^\kappa$. For $Q=Q_j\in\mathcal{D}_j$, we denote for the proof by $Q_i\in \mathcal{D}_i$ the ancestor of $Q$ at level $i<j$, that is, the dyadic square in $\mathcal{D}_i$ with $Q\subset Q_i$. We have
		\begin{equation}
			\label{eq:concetration_on_dyadic_cubes_first_split}
			\bP\big[|\Delta_n(Q_{j})| \ge L^\kappa\big] \le \bP\big[|\Delta_n(Q_j)| \ge L^\kappa, \, |\Delta_n(Q_{j-1})| \le aL^\kappa \big] + \bP\big[|\Delta_n(Q_{j-1})| \ge aL^\kappa\big] \, .
		\end{equation}
		In addition, Proposition~\ref{prop:self_similarity_and_conditional_independence} gives
		\begin{equation*}
			\bE\big[\Delta_n(Q_j) \mid \mathcal{F}_{j-1}\big] = \bE\big[\mu_n(Q_j) \mid \mathcal{F}_{j-1}\big] - \frac{n}{4^j} = \frac{\mu_n(Q_{j-1})}{4}  - \frac{n}{4^j} = \frac{\Delta_n(Q_{j-1})}{4}  \, ,
		\end{equation*}
		which implies that
		\[
		\mathbf{1}_{\{|\Delta_n(Q_{j-1})| \le aL^\kappa\}}\big|\bE\big[\Delta_n(Q_j) \mid \mathcal{F}_{j-1}\big]\big| \le \frac{a}{4} L^{\kappa} \, .
		\]
		Hence, the triangle inequality implies the bound 
		\begin{align*}
			\bP\big[|\Delta_n(Q_j)| \ge L^\kappa, \, &|\Delta_n(Q_{j-1})| \le aL^\kappa \big] \\ &= \bE\big[ \mathbf{1}_{\{|\Delta_n(Q_{j-1})| \le aL^\kappa\}} \bP\big[|\Delta_n(Q_j)| \ge L^\kappa \mid\mathcal{F}_j\big] \big] \\ &\le \bE\Big[ \mathbf{1}_{\{|\Delta_n(Q_{j-1})| \le aL^\kappa\}} \bP\Big[ \big|\Delta_n(Q_j)-\bE\big[\Delta_n(Q_j)\mid\mathcal{F}_{j-1}\big]\big| \ge (1-\tfrac{a}{4}) L^{\kappa} \mid \mathcal{F}_{j-1}\Big]\Big] \, .
		\end{align*}
		As we are working on the event $\{|\Delta_n(Q_{j-1})| \le aL^\kappa\}$, we also have $\mu_n(Q_{j-1}) \in \big[2L,8L\big]$ for all $L$ large enough. Furthermore, Lemma~\ref{lemma:tail_probability_for_disc_at_top_level} implies that 
		\begin{equation}
			\label{eq:sub-gaussian_tail_of_tilde_L_many_points}
			\sup_{\widetilde{L} \in \big[2L,8L\big]} \bP\Big[\big|\Delta_{\widetilde{L}}(Q_1) \big| \ge \big(1-\tfrac{a}{4}\big) L^{\kappa} \Big] \le \exp\Big(-c_{\beta,\kappa} L^{2\kappa}\Big) \, .
		\end{equation}
		Utilizing Proposition~\ref{prop:self_similarity_and_conditional_independence} once more, we see that  
		\[
		\mathbf{1}_{\{\mu_n(Q_{j-1}) = \widetilde{L}\}} \bP\Big[ \big|\Delta_n(Q_j) - \bE\big[\Delta_n(Q_j)\mid\mathcal{F}_{j-1}\big]\big| \ge (1-\tfrac{a}{4}) L^{\kappa} \mid \mathcal{F}_{j-1}\, \Big] = \bP\Big[\big|\Delta_{\widetilde{L}}(Q_1) \big| \ge (1-\tfrac{a}{4}) L^{\kappa} \Big]\, .
		\]
		Combining this observation with~\eqref{eq:sub-gaussian_tail_of_tilde_L_many_points}, we get
		\begin{align*}
			\bP\big[|\Delta_n(Q_j)| &\ge L^\kappa, \, |\Delta_n(Q_{j-1})| \le aL^\kappa \big] \\ &\le \bE\Big[ \mathbf{1}_{\{|\Delta_n(Q_{j-1})| \le aL^\kappa\}} \bP\Big[ \big|\Delta_n(Q_j)-\bE\big[\Delta_n(Q_j)\mid\mathcal{F}_{j-1}\big]\big| \ge (1-\tfrac{a}{4}) L^{\kappa} \mid \mathcal{F}_{j-1}\Big]\Big] \\ &\le \sup_{\widetilde{L} \in \big[2L,8L\big]} \bP\Big[\big|\Delta_{\widetilde{L}}(Q_1) \big| \ge \big(1-\tfrac{a}{4}\big) L^{\kappa} \Big] \le \exp\Big(-c_{\beta,\kappa} L^{2\kappa}\Big) \, .
		\end{align*}
		Plugging back into~\eqref{eq:concetration_on_dyadic_cubes_first_split} and using that $a^{1/\kappa} = 4$, we arrive at the inequality
		\begin{equation}
			\label{eq:concetration_on_dyadic_cubes_first_split_after_bound}
			\bP\big[|\Delta_n(Q_{j})| \ge L^\kappa\big] \le \exp\Big(-c L^{2\kappa}\Big) + \bP\big[|\Delta_n(Q_{j-1})| \ge (4L)^\kappa\big] \, .
		\end{equation}
		Now, iterating~\eqref{eq:concetration_on_dyadic_cubes_first_split_after_bound} until we reach the top level of the tree we see that
		\begin{align*}
			\bP\big[|\Delta_n(Q_{j})| \ge L^\kappa\big] & \le \exp\Big(-c L^{2\kappa}\Big) + \bP\big[|\Delta_n(Q_{j-1})| \ge (4L)^\kappa\big] \\
			& \le \exp\Big(-c L^{2\kappa}\Big) + \exp\Big(-c\, 4^{2\kappa} L^{2\kappa}\Big) + \bP\big[|\Delta_n(Q_{j-2})| \ge (4^2L)^\kappa\big] \\
			& \le \ldots \le \sum_{j\ge 0} \exp\Big(-c  (4^jL)^{2\kappa}\Big) \le \exp\big(-cL^{2\kappa}\big) \, ,
		\end{align*}
		as desired.
	\end{proof}
	Recall that $\bD_n(R)$ is a disk of radius $R/\sqrt{n}$ centered around the point $z\in(0,1)^2$. The following definition will be used throughout the paper.
	\begin{definition}
		\label{definition:maximal_and_boundary_cubes}
		We say that a square $Q\in \mathcal{D}_j$ is \emph{maximal} in $\bD_n(R)$ if $Q\subset \bD_n(R)$ but its parent $Q^\prime\in \mathcal{D}_{j-1}$ has $Q^\prime\cap \bD_n(R)^c \not=\emptyset$. We denote the set of all maximal squares in $\mathcal{D}_j$ by $\mathcal{U}_j$. We say that $Q\in \mathcal{D}_j$ is a \emph{boundary square} of $\bD_n(R)$ if both $Q\cap \bD_n(R)\not=\emptyset$ and $Q\cap \bD_n(R)^c\not=\emptyset$, and denote by $\mathcal{V}_j$ all boundary squares of $\bD_n(R)$ in $\mathcal{D}_j$.
	\end{definition}
	Let $\ell\ge 1$ be an integer such that 
	\begin{equation}
		\label{eq:def_of_ell}
		\frac{R}{\sqrt{n}} \le 2^{-\ell} < 2\frac{R}{\sqrt{n}} \, .
	\end{equation}
	By our choice of $\ell$, for all $j< \ell$ we have $\mathcal{U}_{j} = \emptyset$, and furthermore $|\mathcal{U}_\ell \cup \mathcal{V}_\ell| \le 10$. The following corollary follows immediately from Lemma~\ref{lemma:concentration_on_dyadic_cubes}, by applying the union bound.
	\begin{corollary}
		\label{corollary:concentration_of_four_cubes_containing_the_disk}
		For $\ell$ given by~\eqref{eq:def_of_ell} and for all $\gamma\in(0,2)$ we have
		\[
		\bP\Big[ \bigcup_{Q\in \mathcal{U}_{\ell} \cup \mathcal{V}_\ell } \big\{ |\Delta_n(Q) |  \ge R^\gamma \big\} \Big] \le 10\exp\big( - c_{\gamma,\beta}R^{2\gamma} \big)
		\]
		for all $R>R_0(\beta,\gamma)$ large enough, uniformly as $n\to \infty$.
	\end{corollary}
	For $\alpha\in(\frac{1}{2},2)$ and $\eta>0$ small enough, we will frequently consider the event 
	\begin{equation}
		\label{eq:def_of_event_no_disc_in_top_cubes_containing_the_disk}
		\mathcal{A} = \Big\{ \forall  Q\in \mathcal{U}_\ell \cup \mathcal{V}_\ell \, : \, |\Delta_n(Q)| \le R^{\alpha-\eta} \Big\} \in \mathcal{F}_{\ell} \, .
	\end{equation} 
	Note that, by Corollary~\ref{corollary:concentration_of_four_cubes_containing_the_disk}, $\bP\big[\mathcal{A}^c\big] \leq \exp(-R^{2\alpha - 3\eta})$ for all $R$ large enough. Since $2\alpha>\varphi(\alpha)$ for the range $\alpha\in(\tfrac{1}{2},2)$, with $\varphi$ given by~\eqref{eq:def_of_phi_alpha}, we can always assume that the event $\mathcal{A}$ holds in this range.
	
	We conclude this section by obtaining upper and lower bounds on the probability of seeing an extreme overcrowding of points in a given dyadic square. The following lemma would be used, among other things, to prove Theorem~\ref{thm:Micro_JLM} in the range $\alpha> 2$ (see Section~\ref{sec:overcrowding} for the details).
	\begin{lemma}
		\label{lemma:overcrowding_probability_in_dyadic_cube}
		There exist $C_\beta>0$ so that for all $Q\in \mathcal{D}_j$ we have
		\[
		\bP\big[\mu_n(Q) = n\big] \ge \exp\Big(-C_\beta \, j n^2\Big)\, .
		\] 
		Furthermore, for all $\delta\in(0,1)$ there exists $c_{\beta,\delta}>0$ such that
		\[
		\bP\big[\mu_n(Q) \ge \delta n\big] \le \exp\Big(-c_{\beta,\delta} \, j n^2\Big)
		\]
		for all $j\ge j_0(\beta,\delta)$ large enough.
	\end{lemma}
	\begin{proof}
		The first inequality is almost immediate from the representation~\eqref{eq:alternative_form_for_the_hamiltonian}, as
		\begin{align}
			\label{eq:equality_for_maximal_overcrowding_in_square}
			\nonumber
			\bP\big[&\mu_n(Q) = n\big] \\ \nonumber &= \frac{1}{Z(n,\beta)} \int_{[0,1]^{2n}} \mathbf{1}_{\{\mu_n(Q) = n\}}e^{-\beta\mathcal{H}_n(x_1,\ldots,x_n)} \, {\rm d} \, {\tt Leb}(x_1)\cdots {\rm d} \, {\tt Leb}(x_n) \\ &=  \frac{e^{-\beta\sum_{i=1}^{j-1} \binom{n}{2}}}{Z(n,\beta)}\int_{[0,2^{-j}]^{2n}} \exp\Big(-\beta\Big[\binom{n}{2} + \sum_{i=j+1}^{\infty} \sum_{\substack{Q\in \mathcal{D}_i\\ Q\subset [0,2^{j}]^2}}\binom{\mu_n(Q)}{2}\Big]\Big) \, {\rm d} \, {\tt Leb}(x_1)\cdots {\rm d} \, {\tt Leb}(x_n) \, . 
		\end{align}
		By the change of variables $x\mapsto 2^{-j}x$ we have 
		\[
		\int_{[0,2^{-j}]^{2n}} \exp\Big(-\beta\Big[\binom{n}{2} + \sum_{i=j+1}^{\infty} \sum_{\substack{Q\in \mathcal{D}_i\\ Q\subset [0,2^{j}]^2}}\binom{\mu_n(Q)}{2}\Big]\Big) \, {\rm d} \, {\tt Leb}(x_1)\cdots {\rm d} \, {\tt Leb}(x_n)= (4^{-j})^n \, Z(n,\beta)
		\]
		which, in view of~\eqref{eq:equality_for_maximal_overcrowding_in_square} gives
		\[
		\bP\big[\mu_n(Q) = n\big] = e^{-\beta (j-1) \binom{n}{2}} (4^{-j})^n = \exp\Big(-\beta (j-1) \binom{n}{2} - jn \log(4)\Big)
		\]
		and the desired bound follows. For the second inequality, we note that for each configuration $(x_1,\ldots,x_n)\in\Sigma_n$ which is contained in the event $\{\mu_n(Q) \ge \delta n\}$ we have
		\[
		\mathcal{H}_n(x_1,\ldots,x_n) \ge \sum_{i=1}^{j} \binom{\lfloor\delta n\rfloor}{2} \ge j \frac{\delta^2n^2}{2} \, . 
		\]
		Hence,
		\begin{align*}
			\bP\big[\mu_n(Q) \ge \delta n\big] & \le \frac{1}{Z(n,\beta)} \int_{[0,1]^{2n}}\mathbf{1}_{\{\mu_n(Q)\ge \delta n\}}e^{-\beta\mathcal{H}_n(x_1,\ldots,x_n)} \, {\rm d}x_1\ldots{\rm d}x_n \\ 
			& \le \frac{1}{Z(n,\beta)} \exp\Big(-\beta j \frac{\delta^2n^2}{2}\Big) \stackrel{\text{Claim}~\ref{claim:estimates_on_partition_function}}{\le} \exp\Big(C_\beta n^2 - \beta j \frac{\delta^2n^2}{2} \Big) \, ,
		\end{align*}
		and for $j\ge j_0(\beta,\delta)$ we get the bound	\[
		\bP\big[\mu_n(Q) \ge \delta n\big] \le \exp\Big(-c j n^2\Big)
		\]
		as desired.
	\end{proof}
	\subsection{A few geometric claims}
	Recall our choice of $\ell \ge 1$ given by~\eqref{eq:def_of_ell} and that $\bD_n(R)$ is a disk of radius $R/\sqrt{n}$. We conclude Section~\ref{sec:preliminaries} with a few simple geometric claims. Denote for this part by
	$$\mathcal{T} = \big\{ \xi + [0,1)^2 \, : \, \xi \in \bZ^2 \big\}$$ a partition of $\bR^2$ into unit squares. 
	The first claim we will need is included from Chatterjee~\cite{chatterjee}, and has to do with dyadic squares which lie on the boundary of the disk $\bD_n(R)$. In our set-up, it will be most convenient to state it as follows.
	\begin{claim}
		\label{claim:enough boundary_cubes_are_good}
		There exists an absolute constant $c_1>0$ so that for all $j\ge \ell$, we have at least $c_1\, 2^{j-\ell}$ squares in $Q\in \mathcal{D}_j$ such that
		\begin{equation*}
			10^{-5} \le \frac{{\tt Leb}\big(\bD_n(R) \cap Q\big)}{{\tt Leb}(Q)} \le 1- 10^{-5} \, .
		\end{equation*} 
	\end{claim}
	The proof of Claim~\ref{claim:enough boundary_cubes_are_good} follows from the following elementary statement from~\cite{chatterjee}. 
	\begin{claim}[{\cite[Lemma~3.12]{chatterjee}}]
		\label{claim:geometric_claim_from_chatterjee}
		Let $x\in \bR^2$ and suppose that $\mathbf{n} =(n_1,n_2) \in \bS^1$ is a unit vector which satisfies
		\begin{equation}
			\label{eq:condition_on_unit_normal_not_parralel}
			\min\big\{ |n_1|, |n_2| \big\} \ge 0.1 \, .
		\end{equation}
		Let $L$ be the line that contains $x$ and is perpendicular to $\mathbf{n}$. Then there is a square $T \in \mathcal{T}$, within Euclidean distance $\sqrt{101}$ from $x$, such that the line $L$ splits $T$ into two parts in such a way that each part has area at least $6\times 10^{-5}$.
	\end{claim}
	\begin{proof}[Proof of Claim~\ref{claim:enough boundary_cubes_are_good}]
		For any $x\in \partial \bD_n(R)$, we denote by $T_x$ the tangent line to $\partial \bD_n(R)$ at the point $x$. A simple computation shows that for all $\eps>0$, the set $$\{ y\in\partial \bD_n(R) \, : |y-x|\le \eps \}$$ lies inside a slab of width $2 \eps^2$ around $T_x$. Applying this observation with $\eps = 2^{-j}$ and scaling, Claim~\ref{claim:geometric_claim_from_chatterjee} implies that provided the unit normal $\mathbf{n}_x\in \bS^1$ at $x$ satisfies~\eqref{eq:condition_on_unit_normal_not_parralel}, there exists $Q\in \mathcal{D}_j$ at distance at most $\sqrt{202} \, 2^{-j}$ from $x$ such that
		\[
		10^{-5} \le \frac{{\tt Leb}\big(\bD_n(R) \cap Q\big)}{{\tt Leb}(Q)} \le 1- 10^{-5} 
		\]
		holds. The desired claim now follows from taking a net of size $\lfloor2^{j-\ell}/1000 \rfloor + 2$ of equidistant points on $\partial \bD_n(R)$, noting that, since $\bD_n(R)$ is a disk, at least half (say) of the points in the net have unit normal which satisfies~\eqref{eq:condition_on_unit_normal_not_parralel}.
	\end{proof}
	Finally, we will frequently use a Gauss-type estimate for the number of squares near the boundary of the circle, which we formalize by the following claim.
	\begin{claim}
		\label{claim:gauss_type_estimate}
		There exists an absolute constant $C>0$ so that the following holds. For any disk $U\subset \bR^2$ of radius $r>0$, the number of squares from $\mathcal{T}$ which intersect $\partial U$ is at most $Cr$. 
	\end{claim}
	\begin{proof}
		Recall that $\mathcal{T}$ is a collection of disjoint squares of unit area. It is evident that each square $T\in \mathcal{T}$ which intersects $\partial U$ must be contained in the annulus 
		\[
		\big\{ x\in \bR^2 : \, \text{dist}(x,\partial U) \le 2\big\}\, .
		\]
		Since the area of the above annulus is bounded from above by $Cr$ for some absolute constant $C>0$, the claim follows.
	\end{proof}
	\section{Upper bound for \texorpdfstring{$1< \alpha < 2$}{1 < alpha < 2}}
	\label{sec:UB_large_deviations}
	In this section we prove the upper bound in Theorem~\ref{thm:Micro_JLM} for the range $\alpha\in(1,2)$. Here, the upper bound would follow from an estimate on the moment generating function of $\bE\big[\la \Delta_n(\bD_n(R)) \big]$ for large values of $|\la|$. Broadly speaking, such an estimate would follow from a similar estimate for the moment generating function of point count in $\mathcal{D}_1$ (Lemma~\ref{lemma:moment_generating_funtion_upper_bound_top_level_no_assumption} above) and an induction which exploits the tree structure of the model (via an appropriate martingale). We remark that the dominant contribution to our estimate will come from boundary squares $\mathcal{V}_k$ with $2^{-k} \approx R^{\alpha-1}/\sqrt{n}$, as explained already in Section~\ref{sec:ideas_from_the_proof}.
	
	Recall that $p(Q)$ is the relative area of $\bD_n(R)$ inside a dyadic square $Q$, that is
	\[
	p(Q) = \frac{{\tt Leb}\big(Q\cap \bD_n(R)\big)}{{\tt Leb}(Q)} \, .
	\]	
	The induction step in our argument is given by the following claim.
	\begin{claim}
		\label{claim:conditional_upper_bound_on_MGF_one_level_down}
		Let $Q^\prime\in \mathcal{D}_j$ for some $j\ge 0$, then for all $n\ge 2$ and $\la\in \bR$ we have
		\begin{equation*}
			\bE\Bigg[\prod_{\substack{Q\subset Q^\prime \\ Q\in \mathcal{D}_{j+1}}} \exp\Big(\la  p(Q)  \Delta_n(Q)\Big)  \, \Big| \, \mathcal{F}_j \, \Bigg] \le \exp\bigg(\la p(Q^\prime) \Delta_n(Q^\prime) + 4C_\beta\big(1+\la^2\big)\log^2(\mu_n(Q^\prime)+1)\bigg) \, ,
		\end{equation*}
		where $C_\beta$ is the constant from Lemma~\ref{lemma:moment_generating_funtion_upper_bound_top_level_no_assumption}.
	\end{claim}
	\begin{proof}
		By Proposition~\ref{prop:self_similarity_and_conditional_independence}, for all dyadic squares $Q\subset Q^\prime$ with $Q\in \mathcal{D}_{j+1}$ we have
		\[
		\bE\big[\Delta_n(Q) \mid \mathcal{F}_j\big] = \frac{\Delta_n(Q^\prime)}{4} \, .
		\]
		Furthermore, we have
		\begin{equation}
			\label{eq:avarage_of_relative_area_is_relative_area_above}
			\frac{1}{4} \sum_{\substack{Q\subset Q^\prime \\ Q\in \mathcal{D}_{j+1}}} p(Q) = p(Q^\prime)
		\end{equation}
		which we can use to factor out the expectation term from both sides of the inequality, i.e., the claim would follow once we prove that 
		\begin{equation}
			\label{eq:conditional_upper_bound_on_MGF_one_level_down_after_reduction}
			\bE\Bigg[\prod_{\substack{Q\subset Q^\prime \\ Q\in \mathcal{D}_{j+1}}} \exp\Big(\la  p(Q)  \mu_n(Q)\Big)  \, \Big| \, \mathcal{F}_j \, \Bigg]  \le \exp\bigg(\la p(Q^\prime) \mu_n(Q^\prime) + 4C_\beta\big(1+\la^2\big)\log^2(\mu_n(Q^\prime)+1)\bigg) \, .
		\end{equation}
		By Proposition~\ref{prop:self_similarity_and_conditional_independence}, conditional on $\mathcal{F}_j$ the point counts on the left-hand side of~\eqref{eq:conditional_upper_bound_on_MGF_one_level_down_after_reduction} have the same distribution as the point count at the top level of the tree with $\mu_n(Q^\prime)$ initial points. By applying Lemma~\ref{lemma:moment_generating_funtion_upper_bound_top_level_no_assumption} with
		\[
		\boldsymbol{\la} = \la \cdot \big( p(Q_1), p(Q_2), p(Q_3), p(Q_4)\big)\, ,
		\]
		we get that
		\begin{align*}
			\bE&\Bigg[\prod_{\substack{Q\subset Q^\prime \\ Q\in \mathcal{D}_{j+1}}} \exp\Big(\la  p(Q)  \mu_n(Q)\Big)  \, \Big| \, \mathcal{F}_j \, \Bigg] \\ &\le \exp \Bigg(\la \bigg(\frac{1}{4}\sum_{\substack{Q\subset Q^\prime \\ Q\in \mathcal{D}_{j+1}}} p(Q) \bigg)\, \mu_n(Q^\prime) + C_\beta\Big(1+\la^2\sum_{\substack{Q\subset Q^\prime \\ Q\in \mathcal{D}_{j+1}}} p(Q)^2\Big)\log^2(\mu_n(Q^\prime)+1)\Bigg) \\ & \stackrel{\eqref{eq:avarage_of_relative_area_is_relative_area_above}}{\le} \exp\bigg(\la p(Q^\prime) \mu_n(Q^\prime) + 4C_\beta\big(1+\la^2\big)\log^2(\mu_n(Q^\prime)+1)\bigg)
		\end{align*}
		which proves~\eqref{eq:conditional_upper_bound_on_MGF_one_level_down_after_reduction}, and with that the claim.
	\end{proof}
	\subsection{The martingale term}
	Recall that $\ell \ge 1$ is the minimal integer such that $2^{-\ell} \ge R/\sqrt{n}$, and also recall Definition~\ref{definition:maximal_and_boundary_cubes} of maximal and boundary squares of $\bD_n(R)$, denoted $\mathcal{U}_k$ and $\mathcal{V}_k$ respectively. Finally, recall the definition of the event $\mathcal{A}$, given by~\eqref{eq:def_of_event_no_disc_in_top_cubes_containing_the_disk}. As the discrepancy is additive, we have
	\begin{align}
		\label{eq:disc_is_additive}
		\nonumber
		\Delta_n\big(\bD_n(R)\big) = \sum_{j=1}^{\infty} \sum_{Q\in \mathcal{U}_j} \Delta_n(Q) &=\sum_{j=\ell}^{\infty} \sum_{Q\in \mathcal{U}_j} \Delta_n(Q) \\ &= \sum_{j=\ell}^{k} \sum_{Q\in \mathcal{U}_j} \Delta_n(Q) + \sum_{Q\in \mathcal{V}_k} \Delta_n\big(\bD_n(R) \cap Q\big)\, .
	\end{align} 
	Letting $\mathcal{F}_k$ be the sigma-algebra generated by $\big\{\mu_n(Q) : Q\in \mathcal{D}_k\big\}$, we will consider the exposure martingale of $\Delta_n(\bD_n(R))$ with respect to the filtration $\{\mathcal{F}_k\}$, that is
	\[
	M_k = \bE\big[\Delta_n(\bD_n(R)) \mid \mathcal{F}_k\big] \, .
	\]
	By Proposition~\ref{prop:self_similarity_and_conditional_independence}, it is evident that for all $Q\in \mathcal{V}_k$ we have
	\[
	\bE\big[\Delta_n\big(\bD_n(R) \cap Q\big) \mid \mathcal{F}_k\big] = p(Q) \Delta_n(Q) \, .
	\]
	which, in view of~\eqref{eq:disc_is_additive}, implies that
	\begin{equation}
		\label{eq:exposure_martingale_after_notations}
		M_k = \sum_{j=\ell}^{k} \sum_{Q\in\mathcal{U}_j} \Delta_n(Q) + \sum_{Q\in \mathcal{V}_k} p(Q)\Delta_n(Q) \, .
	\end{equation}
	Let $\eta>0$ be a small fixed constant. We will denote by $j_1$ the integer such that 
	\begin{equation}
		\label{eq:def_of_j_1}
		\frac{R^{\alpha-1-\eta}}{\sqrt{n}} \le 2^{-j_1} < 2 \, \frac{R^{\alpha -1 -\eta}}{\sqrt{n}} \, ,
	\end{equation}
	and note that $\ell < j_1$, where $\ell$ is given by~\eqref{eq:def_of_ell}. We split the discrepancy as
	\begin{equation*}
		\Delta_n\big(\bD_n(R)\big) = M_{j_1} + B_{j_1}
	\end{equation*} 
	where $M_{j_1}$ is given by~\eqref{eq:exposure_martingale_after_notations} and
	\begin{align}
		\label{eq:large_deviation_upper_bound_def_of_boundary_term}
		\nonumber
		B_{j_1} = \Delta_n\big(\bD_n(R)\big) - M_{j_1} &= \sum_{Q\in \mathcal{V}_{j_1}} \Delta_n\big(Q\cap \bD_n(R)\big) - p(Q) \Delta_n(Q) \\ &= \sum_{Q\in \mathcal{V}_{j_1}} \mu_n\big(Q\cap \bD_n(R)\big) - p(Q) \mu_n(Q) \, .
	\end{align}
	The last equality follows by canceling out the term which correspond to ${\tt Leb}$, which is the same for both terms in the summation. By the union bound, to prove the upper bound in Theorem~\ref{thm:Micro_JLM} it will be enough to prove the corresponding tail estimate for $M_{j_1}$ and $B_{j_1}$ separately. We start with proving this estimate for $M_{j_1}$, as given by the next lemma.
	\begin{lemma}
		\label{lemma:large_deviations_upper_bound_martingale_term}
		For all $\alpha\in(1,2)$ and for all $R$ large enough we have
		\[
		\mathbf{1}_{\mathcal{A}} \, \bP\Big[\big| M_{j_1} \big| \ge \tfrac{1}{2} R^{\alpha} \mid \mathcal{F}_\ell\Big] \le 2\exp\Big(-R^{3\alpha -2 - 3\eta}\Big) 
		\]
		uniformly as $n\to \infty$, where $\mathcal{A}\in \mathcal{F}_{\ell}$ is given by~\eqref{eq:def_of_event_no_disc_in_top_cubes_containing_the_disk}.
	\end{lemma}
	\begin{proof}
		We start by proving the bound for the right tail
		\begin{equation}
			\label{eq:large_deviations_upper_bound_martingale_term_right_tail}
			\mathbf{1}_{\mathcal{A}} \, \bP\Big[ M_{j_1}  \ge \tfrac{1}{2} R^{\alpha} \mid \mathcal{F}_\ell\Big] \le \exp\Big(-R^{3\alpha -2 - 3\eta}\Big)
		\end{equation}
		as the complementary bound for the left tail is derived similarly. First, we obtain an upper bound on the moment generating function of $M_k$, with $k> \ell$. By Proposition~\ref{prop:self_similarity_and_conditional_independence} and the tower property of the conditional expectation, for all $\la\in \bR$ we have 
		\begin{align}\label{eq:large_deviations_upper_bound_martingale_term_tower_property} \nonumber
			\bE\big[e^{\la M_{k}} \mid \mathcal{F}_{\ell} \, \big] &= \bE\big[ \bE\big[e^{\la M_{k}} \mid \mathcal{F}_{k-1} \, \big] \mid \mathcal{F}_\ell \, \big] \\ \nonumber &= \bE\bigg[ \bigg(\prod_{j=\ell}^{k-1} \prod_{Q\in \mathcal{U}_j} e^{\la\Delta_n(Q)}\bigg) \, \bE\Big[ \prod_{Q\in \mathcal{U}_{k} \cup \mathcal{V}_k} e^{\la p(Q) \Delta_n(Q)}\mid \mathcal{F}_{k-1} \, \Big] \mid \mathcal{F}_{\ell} \, \bigg] \\ &= \bE\bigg[ \bigg(\prod_{j=\ell}^{k-1} \prod_{Q\in \mathcal{U}_j} e^{\la\Delta_n(Q)}\bigg) \, \prod_{Q^\prime\in \mathcal{V}_{k-1}} \bE\Big[ \prod_{\substack{Q\subset Q^\prime \\ Q\in \mathcal{D}_{k}}} e^{\la p(Q) \Delta_n(Q)}\mid \mathcal{F}_{k-1} \, \Big] \mid \mathcal{F}_{\ell} \, \bigg]  \, .
		\end{align}
		By Claim~\ref{claim:conditional_upper_bound_on_MGF_one_level_down}, we have the bound
		\[
		\bE\Big[ \prod_{\substack{Q\subset Q^\prime \\ Q\in \mathcal{D}_{k}}} e^{\la p(Q) \Delta_n(Q)}\mid \mathcal{F}_{k-1} \, \Big] \le e^{\la p(Q^\prime) \Delta_n(Q^\prime) + C_\beta(1+\la^2)\log^2(\mu_n(Q^\prime)+1)} \, ,
		\]
		which we plug into~\eqref{eq:large_deviations_upper_bound_martingale_term_tower_property} to get
		\begin{equation*}
			\bE\big[e^{\la M_{k}} \mid \mathcal{F}_{\ell} \, \big] \le \bE\Big[ e^{\la M_{k-1}}  \prod_{Q^\prime\in \mathcal{V}_{k-1}} e^{ C_\beta(1+\la^2)\log^2(\mu_n(Q^\prime)+1)} \mid \mathcal{F}_{\ell} \, \Big] \, .
		\end{equation*}
		Since all $Q^\prime\in \mathcal{V}_{k-1}$ is contained in some dyadic square from $\mathcal{U}_\ell \cup \mathcal{V}_\ell$, we conclude from the definition of the event $\mathcal{A}\in \mathcal{F}_\ell$ that 
		\begin{equation}
			\label{eq:bound_on_point_count_on_the_good_event}
			\mu_n(Q^\prime) \le \max_{Q^{\prime\prime}\in \mathcal{U}_\ell \cup \mathcal{V}_\ell} \mu_n(Q^{\prime\prime}) \le 10R^2
		\end{equation}		
		for all $R\ge 2$. This observation implies the bound  
		\begin{equation}
			\label{eq:large_deviations_upper_bound_martingale_term_iteration_step}
			\mathbf{1}_{\mathcal{A}} \, \bE\big[e^{\la M_{k}} \mid \mathcal{F}_{\ell} \, \big] \le \mathbf{1}_{\mathcal{A}} \, \bE\big[e^{\la M_{k-1}} \mid \mathcal{F}_{\ell} \, \big]  \, e^{C|\mathcal{V}_{k-1}| (1+\la^2) \log^2(R)} 
		\end{equation}
		for all $k>\ell$. Recalling that $\mathcal{V}_k$ are the boundary squares in $\mathcal{D}_k$ of a disk $\bD_n(R)$ of radii $R/\sqrt{n} \le 2^{-\ell}$, Claim~\ref{claim:gauss_type_estimate} implies that
		\[
		|\mathcal{V}_{k}| \lesssim 2^{k-\ell}
		\]
		for all $k\ge \ell$, and hence
		\begin{equation}
			\label{eq:bound_on_sum_V_k}
			\sum_{k=\ell}^{j_1} |\mathcal{V}_k| \lesssim 2^{j_1 - \ell} \, .
		\end{equation}
		Furthermore, by our choice~\eqref{eq:def_of_ell} of $\ell\ge1$, we have $|\mathcal{U}_\ell \cup \mathcal{V}_\ell| \le 10$, which implies that on the event $\mathcal{A}$ we have
		\begin{equation}
			\label{eq:bound_on_M_l_assuming_good_event}
			|M_\ell| \le 10 \, R^{\alpha - \eta}\, .
		\end{equation}
		Thus, we can iterate the inequality~\eqref{eq:large_deviations_upper_bound_martingale_term_iteration_step} $j_1 - \ell$ times and obtain that 
		\begin{align}
			\label{eq:large_deviations_upper_bound_martingale_term_final_bound_on_MGF} \nonumber
			\mathbf{1}_{\mathcal{A}} \, \bE\big[e^{\la M_{j_1}} \mid \mathcal{F}_{\ell} \, \big] & \le   \exp\Big(C(1+\la^2) \log^2(R) \big(\sum_{k=\ell}^{j_1} |\mathcal{V}_k| \, \big) \, \Big) \, \mathbf{1}_{\mathcal{A}} \,e^{\la M_{\ell}}  \\\nonumber & \le \exp\Big(C(1+\la^2) \log^2(R) 2^{j_1 - \ell}  + 10\la R^{\alpha -\eta} \,\Big)  \\ &\le \exp\Big(C(1+\la^2) \log^2(R) R^{2-\alpha + \eta} + 10\la R^{\alpha -\eta} \, \Big) \, .
		\end{align}
		Applying Markov's inequality with $\la = R^{2\alpha -2 - 3\eta}$ (with $\eta>0$ small enough), implies that
		\begin{align*}
			\mathbf{1}_{\mathcal{A}} \, \bP\Big[ M_{j_1}  \ge \frac{1}{2} R^{\alpha} \mid \mathcal{F}_\ell\Big] &\le e^{-\la R^\alpha/2} \mathbf{1}_{\mathcal{A}} \, \bE\big[e^{\la M_{j_1}} \mid \mathcal{F}_{\ell} \, \big] \\ & \stackrel{\eqref{eq:large_deviations_upper_bound_martingale_term_final_bound_on_MGF}}{\le} \exp\Big(-\la R^\alpha/2 + C\la^2 \log^2(R) R^{2-\alpha + \eta} +10\la R^{\alpha-\eta} \Big)  \le \exp\Big(-R^{3\alpha -2 -3\eta}\Big)
		\end{align*}
		for all $R$ large enough, which proves~\eqref{eq:large_deviations_upper_bound_martingale_term_right_tail}. To conclude the lemma, we will need the upper bound for the left tail, namely,
		\begin{equation*}
			\mathbf{1}_{\mathcal{A}} \, \bP\Big[ M_{j_1}  \le - \tfrac{1}{2} R^{\alpha} \mid \mathcal{F}_\ell\Big] \le \exp\Big(-R^{3\alpha -2 - 3\eta}\Big) \, .
		\end{equation*}
		Here, Markov's inequality gives 
		\begin{equation*}
			\mathbf{1}_{\mathcal{A}} \, \bP\Big[ M_{j_1}  \le - \tfrac{1}{2} R^{\alpha} \mid \mathcal{F}_\ell\Big] \le e^{-\la R^\alpha /2} \, \mathbf{1}_{\mathcal{A}} \, \bE\big[e^{-\la M_{j_1}} \mid \mathcal{F}_{\ell} \, \big] \, .
		\end{equation*}
		and a similar application of~\eqref{eq:large_deviations_upper_bound_martingale_term_final_bound_on_MGF} gives the desired bound.
	\end{proof}
	\subsection{The boundary term}
	As we already bounded the (conditional) probability that $|M_{j_1}|$ is large, it remains to obtain a similar bound for the boundary term $|B_{j_1}|$, given by the next lemma.
	\begin{lemma}
		\label{lemma:large_deviations_upper_bound_boundary_term}
		For all $\alpha\in(1,2)$ and for all $R$ large enough we have
		\[
		\mathbf{1}_{\mathcal{A}} \, \bP\Big[\big| B_{j_1} \big| \ge \tfrac{1}{2} R^{\alpha} \mid \mathcal{F}_\ell\Big] \le \exp\Big(-R^{3\alpha -2 - 3\eta}\Big)
		\]
		uniformly as $n\to \infty$, where $\mathcal{A}\in \mathcal{F}_{\ell}$ is given by~\eqref{eq:def_of_event_no_disc_in_top_cubes_containing_the_disk} and $B_{j_1}$ is given by~\eqref{eq:large_deviation_upper_bound_def_of_boundary_term}.
	\end{lemma} 
	\begin{proof}
		By the triangle inequality, we have
		\[
		\big| B_{j_1} \big| \le \sum_{Q\in \mathcal{V}_{j_1}} \mu_n\big(Q\cap \bD_n(R)\big) + p(Q) \mu_n(Q) \le 2 \sum_{Q\in \mathcal{V}_{j_1}} \mu_n(Q)
		\]
		which implies that 
		\[
		\Big\{\big| B_{j_1} \big| \ge \tfrac{1}{2} R^{\alpha} \Big\} \subset \Big\{\sum_{Q\in \mathcal{V}_{j_1}} \mu_n(Q) \ge \tfrac{1}{4} R^\alpha \Big\} \, .
		\]
		Therefore, it suffices to prove the bound
		\begin{equation}
			\label{eq:large_deviation_upper_bound_boundary_term_only_dyadic}
			\mathbf{1}_{\mathcal{A}} \, \bP\Big[\sum_{Q\in \mathcal{V}_{j_1}} \mu_n(Q) \ge \tfrac{1}{4} R^\alpha \mid \mathcal{F}_\ell\Big] \le \exp\Big(-R^{3\alpha -2 - 3\eta}\Big) \, .
		\end{equation}
		For $Q^\prime\in \mathcal{D}_j$ with $\ell \le j \le j_1$, we denote by
		\begin{equation*}
			r_{j}(Q^\prime) = \frac{\#\big\{Q\in \mathcal{V}_{j_1} \ \text{such that } \, Q\subset Q^\prime\big\}}{4^{j_1-j}}
		\end{equation*}
		the proportion of boundary squares of $\mathcal{V}_{j_1}$ which are inside $Q^\prime$. Clearly, for $Q\in \mathcal{V}_{j_1}$ we have $r_{j_1}(Q) = 1$. Furthermore, for $Q^\prime\in \mathcal{D}_\ell$ Claim~\ref{claim:gauss_type_estimate} implies that
		\[
		\#\big\{Q\in \mathcal{V}_{j_1} \ \text{such that } \, Q\subset Q^\prime\big\} \lesssim 2^{j_1-\ell} \, ,
		\]
		which implies that
		\begin{equation}
			\label{eq:bound_on_r_ell_Q}
			r_{\ell}(Q^\prime) \lesssim 2^{\ell - j_1} \lesssim R^{\alpha-2-\eta} \, .
		\end{equation}
		Given a boundary square $Q\in \mathcal{V}_j$, then at least one of its children in $\mathcal{D}_{j+1}$ is also a boundary square and we have
		\begin{equation}
			\label{eq:average_of_proportions_is_proportion_of_top}
			r_j(Q^\prime) = \frac{1}{4} \sum_{\substack{Q\subset Q^\prime \\ Q\in \mathcal{D}_{j+1}}} r_{j+1}(Q) \, .
		\end{equation} 
		To prove~\eqref{eq:large_deviation_upper_bound_boundary_term_only_dyadic}, we derive (similarly as in the proof of Lemma~\ref{lemma:large_deviations_upper_bound_martingale_term}) an upper bound on the moment generating function. Indeed, by applying Lemma~\ref{lemma:moment_generating_funtion_upper_bound_top_level_no_assumption} with
		\[
		\boldsymbol{\la} = \la \cdot \big(r_{j}(Q_1), r_j(Q_2), r_j(Q_3), r_j(Q_4)\big) 
		\]
		together with Proposition~\ref{prop:self_similarity_and_conditional_independence}, we get the conditional lower bound 
		\begin{align*}
			\bE\Big[\prod_{\substack{Q\subset Q^\prime \\ Q\in \mathcal{D}_{j+1}}} e^{\la r_{j+1}(Q) \mu_n(Q)}\mid \mathcal{F}_{j} \, \Big] \le  \prod_{\substack{Q\subset Q^\prime \\ Q\in \mathcal{D}_{j+1}}} &e^{\la r_{j+1}(Q) \frac{\mu_n(Q^\prime)}{4} + C(1+\la^2) \log^2(\mu_{n}(Q^\prime) + 1)} \\ & \stackrel{\eqref{eq:average_of_proportions_is_proportion_of_top}}{\le} e^{\la r_j(Q^\prime) \mu_n(Q^\prime) + C(1+\la^2) \log^2(\mu_{n}(Q^\prime) + 1)} 
		\end{align*}
		for all $Q^\prime \in \mathcal{V}_j$.
		On the event $\mathcal{A}$, by~\eqref{eq:bound_on_point_count_on_the_good_event} we further have
		\[
		\mathbf{1}_{\mathcal{A}} \, \bE\Big[\prod_{\substack{Q\subset Q^\prime \\ Q\in \mathcal{D}_{j+1}}} e^{\la r_{j+1}(Q) \mu_n(Q)}\mid \mathcal{F}_{j} \, \Big] \le \mathbf{1}_{\mathcal{A}} \, e^{\la r_j(Q^\prime) \mu_n(Q^\prime) + C(1+\la^2) \log^2(R)} 
		\]
		for all $j\ge \ell$. Iterating this inequality $j_1-\ell$ times, we arrive at
		\begin{align*}
			\mathbf{1}_{\mathcal{A}}& \, \bE\Big[\prod_{Q\in \mathcal{V}_{j_1}} e^{\la \mu_n(Q)} \mid \mathcal{F}_{\ell} \, \Big]  \stackrel{\text{Prop.}~\ref{prop:self_similarity_and_conditional_independence}}{=} \mathbf{1}_{\mathcal{A}} \, \bE\Big[\prod_{Q^\prime\in \mathcal{V}_{j_1-1}} \bE\Big[\prod_{\substack{Q\subset Q^\prime \\ Q\in \mathcal{V}_{j}}} e^{\la \mu_n(Q)}\mid \mathcal{F}_{j_1-1} \,\Big] \mid \mathcal{F}_{\ell} \, \Big] \\
			& \le e^{C|\mathcal{V}_{j_1 -1}| (1+\la^2) \log^2(R)} \, \mathbf{1}_{\mathcal{A}} \, \bE\Big[\prod_{Q\in \mathcal{V}_{j_1-1}} e^{\la r_{j_1-1}(Q)\mu_n(Q)} \mid \mathcal{F}_{\ell} \, \Big] \\
			& \le \ldots \le  e^{C (1+ \la^2) \log^2(R) \sum_{j=\ell}^{j_1} |\mathcal{V}_j|} \, \mathbf{1}_\mathcal{A} \, \prod_{Q\in \mathcal{V}_{\ell}} e^{\la r_\ell(Q) \mu_n(Q)}  \, .
		\end{align*}
		Recall that by our choice~\eqref{eq:def_of_ell} of $\ell\ge 1$, we have $|\mathcal{U}_\ell \cup \mathcal{V}_\ell| \le 10$. Furthermore, combining~\eqref{eq:bound_on_point_count_on_the_good_event} with~\eqref{eq:bound_on_r_ell_Q}, we see that on the event $\mathcal{A}$
		\[
		r_\ell(Q) \mu_n(Q) \stackrel{\eqref{eq:bound_on_point_count_on_the_good_event}}{\lesssim}  r_\ell(Q) R^2  \stackrel{\eqref{eq:bound_on_r_ell_Q}}{\lesssim} \frac{R^2}{R^{2-\alpha+\eta}} = R^{\alpha -\eta}
		\] 
		for all $Q\in \mathcal{V}_\ell$. Putting everything together, we obtain the bound
		\begin{equation*}
			\mathbf{1}_{\mathcal{A}} \, \bE\Big[\prod_{Q\in \mathcal{V}_{j_1}} e^{\la \mu_n(Q)} \mid \mathcal{F}_{\ell} \, \Big] \le e^{C (1+ \la^2) \log^2(R) \sum_{j=\ell}^{j_1} |\mathcal{V}_j| + C\la R^{\alpha -\eta}}  \stackrel{\eqref{eq:bound_on_sum_V_k}}{\le}  e^{C(1+\la^2) \log^2(R) 2^{j_1-\ell} + C\la R^{\alpha -\eta}} 
		\end{equation*}
		for some $C=C(\beta)>0$. Applying this bound with $\la = R^{2\alpha -2 -3\eta}$ with $\eta>0$ arbitrarily small, Markov's inequality implies that
		\begin{align*}
			\mathbf{1}_{\mathcal{A}} \, \bP\Big[\sum_{Q\in \mathcal{V}_{j_1}} \mu_n(Q) \ge \tfrac{1}{4} R^\alpha \mid \mathcal{F}_\ell\Big] &\le \mathbf{1}_{\mathcal{A}} \, e^{-\la R^\alpha/4} \, \bE\Big[\prod_{Q\in \mathcal{V}_{j_1}} e^{\la \mu_n(Q)} \mid \mathcal{F}_{\ell} \, \Big] \\ & \le \exp\Big(-\la R^{\alpha}/4 +C(1+\la^2) \log^2(R) R^{2-\alpha +\eta} + C\la R^{\alpha -\eta}\Big) \\ & \le \exp\Big(-R^{3\alpha-2-2\eta}\Big)
		\end{align*}
		for all $R$ large enough. This proves~\eqref{eq:large_deviation_upper_bound_boundary_term_only_dyadic} and we are done.
	\end{proof}
	\subsection{Proof of Theorem~\ref{thm:Micro_JLM}: the upper bound for \texorpdfstring{$1<\alpha<2$}{1 < alpha < 2}}
	We have
	\begin{align*}
		\bP\Big[\big|\Delta_n\big(\bD_n(R)\big) \big| \ge R^{\alpha}\Big] & \le \bP\Big[\big|\Delta_n\big(\bD_n(R)\big) \big| \ge R^{\alpha}, \, \mathcal{A} \, \Big] + \bP\big[\mathcal{A}^c\big] \\
		& \stackrel{\text{Corollary}~\ref{corollary:concentration_of_four_cubes_containing_the_disk}}{\le} \bP\Big[\big|\Delta_n\big(\bD_n(R)\big) \big| \ge R^{\alpha}, \, \mathcal{A} \, \Big] + \exp\Big(-R^{2\alpha - 3\eta}\Big) \, .
	\end{align*}
	To bound the probability on the right-hand side, we recall the decomposition $\Delta_n\big(\bD_n(R)\big) = M_{j_1} + B_{j_1}$ (see~\eqref{eq:large_deviation_upper_bound_def_of_boundary_term}) and note that
	\begin{align*}
		\bP\Big[\big|\Delta_n\big(\bD_n(R)\big) \big| \ge R^{\alpha}, \, \mathcal{A} \, \Big]  &\le \bP\Big[\big|M_{j_1} \big| \ge \tfrac{1}{2}R^{\alpha}, \, \mathcal{A} \, \Big] + \bP\Big[\big|B_{j_1} \big| \ge \tfrac{1}{2}R^{\alpha}, \, \mathcal{A} \, \Big] \\ &= \bE\Big[\mathbf{1}_{\mathcal{A}} \, \bP\Big[\, \big| M_{j_1} \big| \ge \tfrac{1}{2} R^{\alpha} \mid \mathcal{F}_\ell \, \Big] \, \Big] + \bE\Big[\mathbf{1}_{\mathcal{A}} \, \bP\Big[ \, \big| B_{j_1} \big| \ge \tfrac{1}{2} R^{\alpha} \mid \mathcal{F}_\ell\, \Big]\Big] \, .
	\end{align*}
	Combining the above with Lemma~\ref{lemma:large_deviations_upper_bound_martingale_term} and Lemma~\ref{lemma:large_deviations_upper_bound_boundary_term}, we see that for $R$ large enough 
	\begin{equation*}
		\bP\Big[\big|\Delta_n\big(\bD_n(R)\big) \big| \ge R^{\alpha}\Big]  \le 3 \exp\Big(-R^{3\alpha - 2 - 3\eta}\Big) + \exp\Big(-R^{2\alpha - 3\eta}\Big) \le \exp\Big(-R^{3\alpha - 2 - 4\eta}\Big) \, .
	\end{equation*}
	As $\eta>0$ can be made arbitrarily small, we got the desired upper bound for $1<\alpha<2$.
	\qed
	\section{Upper bound for \texorpdfstring{$\frac{1}{2} < \alpha < 1$}{1/2 < alpha < 1}}
	\label{sec:UB_moderate_deviations}
	The proof of the upper bound in the regime $\frac{1}{2}<\alpha<1$ is similar to the one from Section~\ref{sec:UB_large_deviations} (upper bound for $1<\alpha<2$), in that we will use the inherent ``tree structure" to obtain upper bound on corresponding moment generating functions. To deal with the boundary term we will need a different argument, as explained below. First, we formalize the following claim which serves as the induction step in this section, analogous to Claim~\ref{claim:conditional_upper_bound_on_MGF_one_level_down} from Section~\ref{sec:UB_large_deviations}.
	\begin{claim}
		\label{claim:conditional_upper_bound_MGF_one_level_small_input}
		Let $Q^\prime\in \mathcal{D}_j$ for some $j\ge 1$. For all $\la\in \bR$ with $$|\la|\le (\log (\mu_n(Q^\prime)+1))^{-2}\, ,$$ we have
		\begin{equation*}
			\bE\Bigg[\prod_{\substack{Q\subset Q^\prime \\ Q\in \mathcal{D}_{j+1}}} \exp\Big(\la \, p(Q)  \Delta_n(Q)\Big)  \, \Big| \, \mathcal{F}_j \, \Bigg] \le \exp\bigg(\la \, p(Q^\prime) \Delta_n(Q^\prime) + 4C_\beta \la^2\log^4(\mu_n(Q^\prime)+1)\bigg)
		\end{equation*}
		where $C_\beta>0$ is the constant from Lemma~\ref{lemma:upper_bound_on_MGF_top_level_small_input}.
	\end{claim}
	\begin{proof}
		The proof is similar to that of Claim~\ref{claim:conditional_upper_bound_on_MGF_one_level_down}. By Proposition~\ref{prop:self_similarity_and_conditional_independence}, conditional of $\mathcal{F}_j$, the random variables
		\[
		\Big\{ \mu_n(Q) \Big\}_{\substack{Q\subset Q^\prime \\ Q\in \mathcal{D}_{j+1}}}
		\]
		have the same distribution as the point count at the top level of the tree, with $\mu_n(Q^\prime)$ initial particles. Furthermore, utilizing~\eqref{eq:avarage_of_relative_area_is_relative_area_above}, we see that
		\[
		\sum_{Q\subset Q^{\prime}} p(Q) {\tt Leb}(Q) = \frac{{\tt Leb}(Q^\prime)}{4}\sum_{Q\subset Q^{\prime}} p(Q) =  p(Q^\prime) {\tt Leb}(Q^\prime)\, .
		\]
		Hence, we can factor out the expectation terms, and the claim would follow once we show that
		\begin{equation*}
			\bE\Bigg[\prod_{\substack{Q\subset Q^\prime \\ Q\in \mathcal{D}_{j+1}}} \exp\Big(\la  p(Q)  \mu_n(Q)\Big)  \, \Big| \, \mathcal{F}_j \, \Bigg]  \le \exp\bigg(\la p(Q^\prime) \mu_n(Q^\prime) + 4C_\beta \la^2\log^4(\mu_n(Q^\prime)+1)\bigg) \, .
		\end{equation*}
		In turn, the inequality above follows from a simple application of Lemma~\ref{lemma:upper_bound_on_MGF_top_level_small_input}, which we may apply as $|\la| \le (\log (\mu_n(Q^\prime)+1))^{-2}$. 	
	\end{proof}
	\subsection{Decomposition of \texorpdfstring{$\Delta_n\big(\bD_n(R)\big)$}{ the discrepancy}}
	Let $j_2 > \ell$ be an integer such that 
	\begin{equation}
		\label{eq:def_of_j_2}
		n^{-1/2} \le 2^{-j_2} < 2n^{-1/2} 
	\end{equation}
	and recall that
	\begin{equation*}
		M_{j_2} = \bE\big[\Delta_n\big(\bD_n(R)\big) \mid \mathcal{F}_{j_2}\big] = \sum_{j=\ell}^{j_2} \sum_{Q\in \mathcal{U}_j} \Delta_n(Q) + \sum_{Q\in \mathcal{V}_{j_2}} p(Q) \Delta_n(Q) \, .
	\end{equation*}
	Here, as before, $\mathcal{U}_j$ are the maximal squares in $\mathcal{D}_j$ and $\mathcal{V}_j$ are the boundary squares (see Definition~\ref{definition:maximal_and_boundary_cubes}). We denote the boundary term by
	\begin{align}
		\label{eq:upper_bound_moderate_deviations_def_of_boundary_term}
		\nonumber
		B_{j_2} = \Delta_n\big(\bD_n(R)\big) - M_{j_2} &= \sum_{Q\in \mathcal{V}_{j_2}} \Delta_n\big(\bD_n(R) \cap Q\big) -  p(Q) \Delta_n(Q) \\ &= \sum_{Q\in \mathcal{V}_{j_2}} \mu_n\big(\bD_n(R) \cap Q\big) -  p(Q) \mu_n(Q) \, .
	\end{align}
	\begin{lemma}
		\label{lemma:upper_bound_moderate_deviations_martingale_term}
		For all $\alpha\in\big(\frac{1}{2},1\big)$, $\eta>0$ and for all $R$ large enough, we have
		\[
		\mathbf{1}_{\mathcal{A}} \, \bP\Big[\big| M_{j_2} \big| \ge \tfrac{1}{2} R^{\alpha} \mid \mathcal{F}_\ell\Big] \le 2\exp\Big(-R^{2\alpha -1 - 3\eta}\Big) 
		\]
		uniformly as $n\to \infty$, where $\mathcal{A}\in \mathcal{F}_{\ell}$ is given by~\eqref{eq:def_of_event_no_disc_in_top_cubes_containing_the_disk}.
	\end{lemma}
	\begin{proof}
		The idea here, analogous to the proof of Lemma~\ref{lemma:large_deviations_upper_bound_martingale_term}, is to give an upper bound on the moment generating function $\bE\big[e^{\la M_{j_2}} \mid \mathcal{F}_{\ell} \, \big]$. The difference between the lemmas is that here the parameter $\la$ will be taken small as $R\to \infty$, contrary to the proof of Lemma~\ref{lemma:large_deviations_upper_bound_martingale_term}, where $\la$ was large. This can be achieved by applying Claim~\ref{claim:conditional_upper_bound_MGF_one_level_small_input} instead of Claim~\ref{claim:conditional_upper_bound_on_MGF_one_level_down}. Indeed, on the event $\mathcal{A}$, we know by~\eqref{eq:bound_on_point_count_on_the_good_event} that for all $Q\in \mathcal{V}_j$, with $j\ge \ell$ 
		\[
		\mu_n(Q) \le 10 R^2 \, .
		\]
		Letting $\la = R^{\alpha -1 - 2\eta}$ and recalling that $\alpha<1$, we see that for all $R$ large enough we have
		\[
		\la \le \frac{1}{\log^2(10R^2)} \le \frac{1}{\log^2 \big(\mu_n(Q)\big)}
		\]
		for all $Q\in \mathcal{V}_j$ with $j\ge \ell$. Thus, we may apply Claim~\ref{claim:conditional_upper_bound_MGF_one_level_small_input} and observe that
		\begin{align*}
			\mathbf{1}_{\mathcal{A}} \, \bE\Big[ e^{\la M_{k}}\mid \mathcal{F}_\ell \, \Big] &= \mathbf{1}_{\mathcal{A}} \, \bE\Big[ \prod_{j=\ell}^{k} \prod_{Q\in \mathcal{U}_j}e^{\la \Delta_n(Q)} \prod_{Q\in \mathcal{V}_{k}} e^{\la p(Q) \Delta_n(Q)}\mid \mathcal{F}_\ell \, \Big] \\  & \stackrel{\text{Prop.}~\ref{prop:self_similarity_and_conditional_independence}}{=}  \mathbf{1}_{\mathcal{A}} \, \bE\Big[ \Big(\prod_{j=\ell}^{k-1} \prod_{Q\in \mathcal{U}_j}e^{\la \Delta_n(Q)} \Big) \prod_{Q^\prime \in \mathcal{V}_{k-1}} \bE\Big[ \prod_{\substack{Q\subset Q^\prime \\ Q\in \mathcal{D}_{k}}} e^{\la p(Q) \Delta_n(Q)} \mid \mathcal{F}_{k-1} \, \Big]\mid \mathcal{F}_\ell \, \Big] \\ & \stackrel{\text{Claim}~\ref{claim:conditional_upper_bound_MGF_one_level_small_input}}{\le} \mathbf{1}_{\mathcal{A}} \, \bE\Big[ e^{\la M_{k-1}} \prod_{Q^\prime \in \mathcal{V}_{k-1}} e^{C\la^2 \log^4(\mu_n(Q^\prime) + 1)}\mid \mathcal{F}_\ell \, \Big] \\ & \quad  \, \le e^{C\log^4(R) \la^2 |\mathcal{V}_{k-1}|} \, \mathbf{1}_{\mathcal{A}} \, \bE\Big[ e^{\la M_{k-1}} \mid \mathcal{F}_\ell \, \Big] \, .
		\end{align*}
		Iterating the above inequality $j_2 -\ell$ times, we use the bound~\eqref{eq:bound_on_M_l_assuming_good_event} to arrive at 
		\begin{align*}
			\mathbf{1}_{\mathcal{A}} \, \bE\Big[ e^{\la M_{j_2}}\mid \mathcal{F}_\ell \, \Big] &\le \exp\Big(C\la^2 \log^4(R) \sum_{k=\ell}^{j_2} |\mathcal{V}_{k}|\Big) \,  \mathbf{1}_{\mathcal{A}} \, e^{\la M_\ell} \\ &\le \exp\Big(C\la^2 \log^4(R) \sum_{k=\ell}^{j_2} 2^{k-\ell} + C\la R^{\alpha-\eta}\Big) \le \exp\Big(C\la^2 \log^4(R) R + C\la R^{\alpha -\eta}\Big) \, .
		\end{align*}
		Applying Markov's inequality while recalling that $\la = R^{\alpha-1-2\eta}$, we get
		\begin{align*}
			\mathbf{1}_{\mathcal{A}} \, \bP\Big[ M_{j_2}  \ge \tfrac{1}{2} R^{\alpha} \mid \mathcal{F}_\ell\Big] &\le \exp\Big(-\la R^\alpha/2 + C\la^2 \log^4(R) R + 4\la R^{\alpha -\eta}\Big) \\ & \le \exp\Big(-R^{2\alpha-1-3\eta}\Big)
		\end{align*} 
		for all $R$ large enough. The upper bound for the left tail
		\[
		\mathbf{1}_{\mathcal{A}} \, \bP\Big[ M_{j_2}  \le - \tfrac{1}{2} R^{\alpha} \mid \mathcal{F}_\ell\Big] \le \exp\Big(-R^{2\alpha-1-3\eta}\Big)
		\]
		is obtained in a similar way and the lemma follows.
	\end{proof}
	\subsection{Not too many points near the boundary}
	To deal with the boundary term $B_{j_2}$, we will apply Hoeffding's inequality on the tail of the sum of bounded random variables, see Section~\ref{sec:proof_thm_upper_bound_moderate} below. Before that, we need to show that with high probability, there are not too many points near the boundary of $\bD_n(R)$, as shown by the following lemma.
	\begin{lemma}
		\label{lemma:cramer_estimate_on_the_boundary}
		There exists $c_\beta>0$ so that for all $\eps>0$ we have
		\[
		\mathbf{1}_{\mathcal{A}} \, \bP\Big[\sum_{Q\in \mathcal{V}_{j_2}} \mu_n(Q)^2 \ge R^{1+\eps} \mid \mathcal{F}_\ell \, \Big] \le \exp\Big(-c_\beta R^{1+\eps}\Big)
		\]
		for all $R\ge R_0(\beta,\eps)$ large enough, where $\mathcal{A}$ is given by~\eqref{eq:def_of_event_no_disc_in_top_cubes_containing_the_disk}.
	\end{lemma}
	\begin{proof}
		By our choice of $j_2$ given by~\eqref{eq:def_of_j_2}, we have
		\[
		\bE\big[\mu_n(Q)\big] = \frac{n}{4^{j_2}} \le 2
		\]
		for all $Q\in \mathcal{D}_{j_2}$. Recalling that $|\mathcal{V}_{j_2}| \lesssim R$, we can apply the Cauchy-Schwarz inequality and get that
		\begin{align*}
			\sum_{Q\in \mathcal{V}_{j_2}} \mu_n(Q)^2 &= \sum_{Q\in \mathcal{V}_{j_2}} \big(\Delta_n(Q) + \bE\big[\mu_n(Q)\big]\big)^2 \\ &= \sum_{Q\in \mathcal{V}_{j_2}} \Big\{ \big(\Delta_n(Q)\big)^2 + 2\bE[\mu_n(Q)]  \Delta_n(Q) + \bE\big[\mu_n(Q)\big]^2 \Big\} \\ & \le C R + C \sqrt{R} \, \bigg(\sum_{Q\in \mathcal{V}_{j_2}} \big(\Delta_n(Q)\big)^2\bigg)^{1/2} + \sum_{Q\in \mathcal{V}_{j_2}} \big(\Delta_n(Q)\big)^2 \, ,
		\end{align*}
		where $C>0$ is some absolute constant. Therefore, for all $R$ large enough
		\[
		\Big\{\sum_{Q\in \mathcal{V}_{j_2}} \mu_n(Q)^2 \ge R^{1+\eps}\Big\} \subset \Big\{ \sum_{Q\in \mathcal{V}_{j_2}} \big(\Delta_n(Q)\big)^2 \ge \tfrac{1}{4} R^{1+\eps}\Big\} \, .
		\]
		The proof of the lemma would follow once we show that
		\begin{equation}
			\label{eq:cramer_estimate_on_the_boundary_after_reduction}
			\mathbf{1}_{\mathcal{A}} \, \bP\Big[\sum_{Q\in \mathcal{V}_{j_2}} \big(\Delta_n(Q)\big)^2 \ge \frac{1}{4} R^{1+\eps} \mid \mathcal{F}_\ell \, \Big] \le \exp\Big(-c_\beta R^{1+\eps}\Big)
		\end{equation}
		Indeed, Proposition~\ref{prop:self_similarity_and_conditional_independence} implies that
		\begin{align*}
			\bE\Big[\prod_{Q\in \mathcal{V}_{j_2}} e^{\frac{\beta}{10} \Delta_n(Q)^2} \mid \mathcal{F}_{\ell} \, \Big] &= \bE\bigg[\bE\Big[\prod_{Q\in \mathcal{V}_{j_2}} e^{\frac{\beta}{10} \Delta_n(Q)^2} \mid \mathcal{F}_{j_2-1} \, \Big] \mid \mathcal{F}_{\ell} \, \bigg] \\ &=  \bE\bigg[\prod_{Q^\prime\in \mathcal{V}_{j_2-1}}\bE\Big[\prod_{\substack{Q\subset Q^\prime \\ Q\in \mathcal{V}_{j_2}}} e^{\frac{\beta}{10} \Delta_n(Q)^2} \mid \mathcal{F}_{j_2-1} \, \Big] \mid \mathcal{F}_{\ell} \, \bigg] \, .
		\end{align*}
		By applying Claim~\ref{claim:upper_bound_on_moment_generating_function_disc_squared_top_level} with
		\[
		\boldsymbol{s} = \Big(\frac{\beta}{10},\frac{\beta}{10},\frac{\beta}{10},\frac{\beta}{10}\Big)
		\]
		we see that for each $Q^\prime\in \mathcal{V}_{j_2-1}$, conditionally on $\mathcal{F}_{j_2-1}$ we have
		\[
		\bE\Big[\prod_{\substack{Q\subset Q^\prime \\ Q\in \mathcal{V}_{j_2}}} e^{\frac{\beta}{10} \Delta_n(Q)^2} \mid \mathcal{F}_{j_2-1} \, \Big] \le \exp \Big(C \log^2(\mu_n(Q^\prime)+1)\Big) \, .
		\]
		On the event $\mathcal{A}$, this implies the bound
		\begin{align*}
			\mathbf{1}_{\mathcal{A}} \, \bE\Big[\prod_{Q\in \mathcal{V}_{j_2}} e^{\frac{\beta}{10} \Delta_n(Q)^2} \mid \mathcal{F}_{\ell} \, \Big] &\le \mathbf{1}_{\mathcal{A}} \, \bE\Big[\prod_{Q^\prime\in \mathcal{V}_{j_2-1}}e^{C\log^2(\mu_n(Q^\prime) + 1)} \mid \mathcal{F}_{\ell} \, \Big] \\ & \le \exp\Big(C\log^2(R) |\mathcal{V}_{j_2-1}|\Big) \le \exp\Big(CR \, \log^2(R) \Big) \, .
		\end{align*}
		Hence, Markov's inequality implies that
		\begin{align*}
			\mathbf{1}_{\mathcal{A}} \, \bP\Big[\sum_{Q\in \mathcal{V}_{j_2}} \big(\Delta_n(Q)\big)^2 \ge \tfrac{1}{4} R^{1+\eps} \mid \mathcal{F}_\ell \, \Big] &\le e^{-\frac{\beta}{40} R^{1+\eps}} \, \mathbf{1}_{\mathcal{A}} \, \bE\Big[\prod_{Q\in \mathcal{V}_{j_2}} e^{\frac{\beta}{10} \Delta_n(Q)^2} \mid \mathcal{F}_{\ell} \, \Big] \\ & \le \exp\Big(-\frac{\beta}{40} R^{1+\eps} + CR \, \log^2(R)\Big)
		\end{align*}
		which proves~\eqref{eq:cramer_estimate_on_the_boundary_after_reduction} holds for all $R$ large enough, and we conclude the proof.
	\end{proof}
	\subsection{Proof of Theorem~\ref{thm:Micro_JLM}: the upper bound for \texorpdfstring{$\frac{1}{2}<\alpha<1$}{1/2 < alpha < 1}}
	\label{sec:proof_thm_upper_bound_moderate}
	Recall Hoeffding's inequality on the tail of the sum of bounded random variables~\cite{Hoeffding}. Let $X_1,\ldots,X_N$ be independent random variables such that $a_i\le X_i \le b_i$, and set $S_N = X_1 + \ldots+ X_N$. Then for all $t> 0$ we have
	\begin{equation}
		\label{eq:Hoeffding_ineq}
		\bP\Big[\big|S_N - \bE[S_N]\big| \ge t\Big] \le 2\exp\bigg(-\frac{2t^2}{\sum_{i=1}^{n}(b_i-a_i)^2}\bigg) \, .
	\end{equation}
	Recalling the split $\Delta_n\big(\bD_n(R)\big) = M_{j_2} + B_{j_2}$, we have the inequality
	\begin{align}
		\nonumber
		\label{eq:moderate_deviation_upper_bound_first_split}
		\bP\Big[\big|\Delta_n\big(\bD_n(R)\big) \big| \ge R^{\alpha}\Big] &\le \bP\Big[\big|\Delta_n\big(\bD_n(R)\big) \big| \ge R^{\alpha}, \, \mathcal{A} \, \Big] + \bP\big[\mathcal{A}^c\big] \\ &\le \bP\Big[\big|M_{j_2} \big| \ge \tfrac{1}{2}R^{\alpha}, \, \mathcal{A} \, \Big] + \bP\Big[\big|B_{j_2} \big| \ge \tfrac{1}{2}R^{\alpha}, \, \mathcal{A} \, \Big] + \bP\big[\mathcal{A}^c\big] \, .
	\end{align}
	By Lemma~\ref{lemma:upper_bound_moderate_deviations_martingale_term} we have
	\[
	\bP\Big[\big|M_{j_2} \big| \ge \tfrac{1}{2}R^{\alpha}, \, \mathcal{A} \, \Big] = \bE\Big[\mathbf{1}_{\mathcal{A}} \, \bP\Big[\big|M_{j_2} \big| \ge \tfrac{1}{2}R^{\alpha} \mid \mathcal{F}_\ell \, \Big]\Big] \le 2\exp\Big(-R^{2\alpha -1 - 3\eta}\Big) \, .
	\]
	Furthermore, Corollary~\ref{corollary:concentration_of_four_cubes_containing_the_disk} implies that for all $R$ large enough,
	\[
	\bP\big[\mathcal{A}^c\big] \le \exp\Big(-R^{2\alpha- 3\eta}\Big) \, .
	\]
	Hence, we will conclude the upper bound in the range $\frac{1}{2}< \alpha<1$ once we deal with the term $|B_{j_2}|$ in~\eqref{eq:moderate_deviation_upper_bound_first_split}, which correspond to the discrepancy near the boundary. For this, we make a further split as
	\begin{equation*}
		\bP\Big[\big|B_{j_2} \big| \ge \tfrac{1}{2}R^{\alpha}, \, \mathcal{A} \, \Big]  \le \bP\Big[\sum_{Q\in \mathcal{V}_{j_2}} \mu_n(Q)^2 \ge R^{1+\eps} , \, \mathcal{A} \, \Big] + \bP\Big[\big|B_{j_2} \big| \ge \tfrac{1}{2}R^{\alpha}, \, \sum_{Q\in \mathcal{V}_{j_2}} \mu_n(Q)^2 < R^{1+\eps} \, \Big]\, .
	\end{equation*} 
	By Lemma~\ref{lemma:cramer_estimate_on_the_boundary}, we know that
	\[
	\bP\Big[\sum_{Q\in \mathcal{V}_{j_2}} \mu_n(Q)^2 \ge R^{1+\eps} , \, \mathcal{A} \, \Big] \le \exp\Big(-c_\beta R^{1+\eps}\Big)
	\]
	and since $2\alpha-1<1+\eps$ in this range, the probability of this event is negligible. To conclude the desired upper bound, it remains to show that
	\begin{equation}
		\label{eq:moderate_deviation_upper_bound_boundary_after_hoeffding}
		\bP\Big[\big|B_{j_2} \big| \ge \tfrac{1}{2}R^{\alpha}, \, \sum_{Q\in \mathcal{V}_{j_2}} \mu_n(Q)^2 < R^{1+\eps} \, \Big] \le \exp\Big(-R^{2\alpha - 1 - \eps/2}\Big) \, .
	\end{equation}
	Indeed, conditioned on $\mathcal{F}_{j_2}$, the random variable
	\[
	B_{j_2} = \sum_{Q\in \mathcal{V}_{j_2}} \mu_n\big(Q\cap \bD_n(R)\big) - p(Q) \mu_n(Q)
	\]
	is a sum of independent mean-zero random variables. As
	\[
	\big|\mu_n\big(Q\cap \bD_n(R)\big) - p(Q) \mu_n(Q)\big| \le 2 \mu_n(Q)
	\]
	we can apply Hoeffding's inequality~\eqref{eq:Hoeffding_ineq} (conditionally on $\mathcal{F}_{j_2}$) and see that
	\begin{equation*}
		\bP\Big[\big|B_{j_2} \big| \ge \tfrac{1}{2}R^{\alpha}\mid \mathcal{F}_{j_2} \, \Big] \le 2\exp\Big(- \frac{R^{2\alpha}}{8\sum_{Q\in \mathcal{V}_{j_2}} \mu_n(Q)^2}\Big) \, .
	\end{equation*}
	By the law of total expectation, this implies that
	\begin{align*}
		\bP\Big[\big|B_{j_2} \big| \ge \tfrac{1}{2}R^{\alpha}&, \, \sum_{Q\in \mathcal{V}_{j_2}} \mu_n(Q)^2 < R^{1+\eps} \, \Big] \\ &= \bE\Big[ \mathbf{1}_{\big\{\sum_{Q\in \mathcal{V}_{j_2}} \mu_n(Q)^2 < R^{1+\eps}\big\}} \bP\Big[\big|B_{j_2} \big| \ge \tfrac{1}{2}R^{\alpha}\mid \mathcal{F}_{j_2} \, \Big]  \Big] \\ & \stackrel{\eqref{eq:Hoeffding_ineq}}{\le} \bE\Big[ \mathbf{1}_{\big\{\sum_{Q\in \mathcal{V}_{j_2}} \mu_n(Q)^2 < R^{1+\eps}\big\}} 2\exp\Big(- \frac{R^{2\alpha}}{8\sum_{Q\in \mathcal{V}_{j_2}} \mu_n(Q)^2}\Big) \, \Big]  \le \exp\Big(-R^{2\alpha - 1 - \eps/2}\Big)
	\end{align*}
	for all $R$ large enough, which proves~\eqref{eq:moderate_deviation_upper_bound_boundary_after_hoeffding}. Plugging into~\eqref{eq:moderate_deviation_upper_bound_first_split}, we get the desired upper bound in the range $\frac{1}{2}<\alpha<1$.
	\qed
	\section{Lower bound for \texorpdfstring{$1 < \alpha < 2$}{1 < alpha < 2}}
	\label{sec:LB_large_deviations}
	Here, we will build an explicit event which is contained in $\{\Delta_n\big(\bD_n(R)\big) \ge R^\alpha\}$, whose probability we can estimate from below. Succinctly, in the event we shall ``push" all point from the boundary cubes into the disk, assuming that up until this height of the tree all discrepancies were typical.
	
	Let $\eta>0$ be a small constant and let $j_3>\ell$ be an integer so that
	\begin{equation}
		\label{eq:def_of_j_3}
		\frac{R^{\alpha - 1 + \eta}}{\sqrt{n}} \le 2^{-j_3} < 2 \, \frac{R^{\alpha-1+\eta}}{\sqrt{n}} \, .
	\end{equation}
	Recall the notion of maximal squares and boundary squares, as given by Definition~\ref{definition:maximal_and_boundary_cubes}. 
	\subsection{The good event in \texorpdfstring{$\mathcal{F}_{j_3}$}{Fj3}} Consider the event in $\mathcal{G}_1\in \mathcal{F}_{j_3}$, given by
	\begin{equation}
		\label{eq:def_of_good_event_large_deviations_lower_bound}
		\mathcal{G}_1 = \bigg\{ \forall Q\in \bigcup_{j=\ell}^{j_3} \mathcal{U}_j \cup \mathcal{V}_{j_3} \, : \, |\Delta_{n}(Q)| < R^{\eta/2} \, \bigg\}\, .
	\end{equation}
	\begin{claim}
		\label{claim:large_deviations_lower_bound_good_event_is_likely}
		We have $\displaystyle \lim_{R\to \infty} \lim_{n\to \infty} \bP\big[\mathcal{G}_1\big] = 1$.
	\end{claim}
	\begin{proof}
		By our choice of $\ell\ge 1$~\eqref{eq:def_of_ell} and $j_3\ge \ell$~\eqref{eq:def_of_j_3}, we see that for all $Q\in \bigcup_{j=\ell}^{j_3} \mathcal{U}_j \cup \mathcal{V}_{j_3}$ 
		\[
		\tfrac12 R^{\alpha-1+\eta} \le \bE\big[\mu_n(Q)\big] \le 2 R^2 \, .
		\]
		In particular, for all $Q\in \bigcup_{j=\ell}^{j_3} \mathcal{U}_j \cup \mathcal{V}_{j_3}$ we have the inclusion
		\[
		\Big\{ |\Delta_n(Q)| \ge R^{\eta/2} \Big\} \subset \Big\{ |\Delta_n(Q)| \ge 2^{-\eta/2} \, \big(\bE\big[\mu_n(Q)\big]\big)^{\eta/4} \Big\} \, .
		\]
		Letting $L_0 = L_0(\beta,\eta/4)$ be as in the assumption of Lemma~\ref{lemma:concentration_on_dyadic_cubes}, for all $R$ large enough we have $R^{1-\alpha+\eta} \ge L_0$, and Lemma~\ref{lemma:concentration_on_dyadic_cubes} implies that
		\begin{equation*}
			\bP\Big[ |\Delta_n(Q)| \ge R^{\eta/2}\Big] \le \exp\Big(-c_{\eta,\beta} \big(\bE\big[\mu_n(Q)\big]\big)^{\eta/2} \Big) \le \exp\Big(-c_{\eta,\beta} R^{\eta(\alpha-1+\eta)/2}\Big)\, .
		\end{equation*}
		Furthermore, Claim~\ref{claim:gauss_type_estimate} implies that for all $j\ge \ell$ we have $|\mathcal{U}_{j}| \lesssim 2^{j-\ell}$, hence
		\[
		\sum_{j=\ell}^{j_3} |\mathcal{U}_j| + |\mathcal{V}_{j_3}| \lesssim 2^{j_3 - \ell} \lesssim R^{2-\alpha-\eta} \, .
		\]
		Combining these observations, the union bound gives
		\[
		\bP\big[\mathcal{G}_1^c\big] \lesssim R^{2-\alpha-\eta} \, \exp\Big(-c_{\eta,\beta} R^{\eta^2/2} \Big)
		\]
		uniformly as $n\to \infty$, and we are done.
	\end{proof}
	We remark that on the event $\mathcal{G}_1$, we further have
	\begin{equation}
		\label{eq:large_deviations_lower_bound_good_event_implies_small_inner_disc}
		\Big|\sum_{j=\ell}^{j_3} \sum_{Q\in \mathcal{U}_j} \Delta_n(Q)\Big| \le R^{\alpha - \eta/4} 
	\end{equation}
	for all $R$ large enough. Indeed, the triangle inequality implies that
	\begin{align*}
		\Big|\sum_{j=\ell}^{j_3} \sum_{Q\in \mathcal{U}_j} \Delta_n(Q)\Big| &\le \sum_{j=\ell}^{j_3} \sum_{Q\in \mathcal{U}_j} \big|\Delta_n(Q)\big| \le R^{\eta/2} \sum_{j=\ell}^{j_3} |\mathcal{U}_j| \\ & \lesssim R^{\eta/2} \sum_{j=\ell}^{j_3} 2^{j-\ell} \lesssim R^{\eta/2} 2^{j_3-\ell} \lesssim R^{\eta/2} R^{2-\alpha-\eta}  \stackrel{\alpha>1}{\le} R^{\alpha-\eta/2} \, .
	\end{align*} 
	Hence,~\eqref{eq:large_deviations_lower_bound_good_event_implies_small_inner_disc} holds on the event $\mathcal{G}_1$ for all $R$ large enough.
	\subsection{Taking care of the boundary}
	We will say that a boundary square $Q\in \mathcal{V}_{j_3}$ is \emph{good} if it satisfies
	\begin{equation}
		\label{eq:def_of_good_boundary_cube}
		10^{-5} \le \frac{{\tt Leb}\big(Q\cap \bD_n(R)\big)}{{\tt Leb}(Q)} \le 1-10^{-5} \, .
	\end{equation}
	By Claim~\ref{claim:enough boundary_cubes_are_good}, there are at least $$c_1 \, 2^{j_3-\ell} \ge \frac{c_1}{2} \, R^{2-\alpha-\eta}$$ good squares in $\mathcal{V}_{j_3}$, for some absolute constant $c_1>0$. We will denote by $\mathcal{V}_{j_3}^G$ the subset of good squares in $\mathcal{V}_{j_3}$, and by $\mathcal{V}_{j_3}^B = \mathcal{V}_{j_3} \setminus \mathcal{V}_{j_3}^G$. The above observation reads
	\begin{equation}
		\label{eq:lower_and_upper_bounds_on_good_bad_boundary_squares}
		|\mathcal{V}_{j_3}^G| \gtrsim R^{2-\alpha-\eta} \, , \qquad \text{and} \qquad |\mathcal{V}_{j_3}^B|\le |\mathcal{V}_{j_3}| \lesssim R^{2-\alpha-\eta} \, .
	\end{equation}
	By writing
	\[
	X_{j_3}^G = \sum_{Q\in \mathcal{V}_{j_3}^{G}} \Delta_{n}\big(Q\cap \bD_n(R)\big)\, , \qquad X_{j_3}^B = \sum_{Q\in \mathcal{V}_{j_3}^{B}} \Delta_{n}\big(Q\cap \bD_n(R)\big)\, ,
	\]
	we have the decomposition
	\begin{equation}
		\label{eq:large_deviations_lower_bound_decomposition_of_disc}
		\Delta_n\big(\bD_n(R)\big) = \sum_{j=1}^{j_3} \sum_{Q\in \mathcal{U}_j} \Delta_n(Q) + X_{j_3}^B + X_{j_3}^G \, .
	\end{equation}
	We will built the event so that a large discrepancy is coming from the term $X_{j_3}^G$. First, we show that typically the term $X_{j_3}^B$ is not too large.
	\begin{lemma}
		\label{lemma:large_deviations_lower_bound_bad_cubes_do_not_contribute}
		For all $R$ large enough, uniformly as $n\to \infty$, we have
		\[
		\mathbf{1}_{\mathcal{G}_1} \, \bP\Big[ \, \big|X_{j_3}^B\big| < R^{\alpha} \mid \mathcal{F}_{j_3} \, \Big] \ge \frac{1}{2} \, \mathbf{1}_{\mathcal{G}_1} 
		\]
		where $\mathcal{G}_1$ is given by~\eqref{eq:def_of_good_event_large_deviations_lower_bound}.
	\end{lemma} 
	\begin{proof}
		On the event $\mathcal{G}_1$, we have
		\begin{equation*}
			\Big|\sum_{Q\in \mathcal{V}_{j_3}^{B}} p(Q) \Delta_n(Q)\Big| \leq R^{\eta/2} \,  |\mathcal{V}_{j_3}^{B}| \stackrel{\eqref{eq:lower_and_upper_bounds_on_good_bad_boundary_squares}}{\lesssim} R^{\eta/2+2-\alpha-\eta} = R^{2-\alpha-\eta/2} \le R^{\alpha-\eta/2}
		\end{equation*}
		where in the last inequality we used that $\alpha>1$. Furthermore, by Proposition~\ref{prop:self_similarity_and_conditional_independence}, we have
		\[
		\bE\big[ \Delta_{n}\big(Q\cap \bD_n(R)\big) \mid \mathcal{F}_{j_3} \, \big] = p(Q) \Delta_n(Q) 
		\]
		and the triangle inequality implies that, for all $R$ large enough, 
		\begin{align}
			\label{eq:large_deviations_lower_bound_bad_cubes_do_not_contribute_after_centering}
			\nonumber
			\mathbf{1}_{\mathcal{G}_1} \, \bP\Big[ \, \big|X_{j_3}^B\big| &\ge R^{\alpha} \mid \mathcal{F}_{j_3} \, \Big] \\ &\le \mathbf{1}_{\mathcal{G}_1} \, \bP\Big[ \, \Big|\sum_{Q\in \mathcal{V}_{j_3}^{B}} \Delta_{n}\big(Q\cap \bD_n(R)\big) - \bE\big[ \Delta_{n}\big(Q\cap \bD_n(R)\big) \mid \mathcal{F}_{j_3} \, \big] \Big| \ge \tfrac{1}{2}R^{\alpha} \mid \mathcal{F}_{j_3} \, \Big] \, .
		\end{align}
		To bound from above the right-hand side of~\eqref{eq:large_deviations_lower_bound_bad_cubes_do_not_contribute_after_centering}, we note that conditional on $\mathcal{F}_{j_3}$, Proposition~\ref{prop:self_similarity_and_conditional_independence} implies that the random variables $$\Big\{\Delta_n\big(Q\cap \bD_n(R) \big) \Big\}_{Q\in \mathcal{V}_{j_3}^B}$$
		are independent. By Markov's inequality, we have
		\begin{align}
			\label{eq:large_deviations_lower_bound_bad_cubes_do_not_contribute_after_centering_and_markov}
			\nonumber
			\bP\Big[ \, \Big|\sum_{Q\in \mathcal{V}_{j_3}^{B}} & \Delta_{n}\big(Q\cap \bD_n(R)\big) - \bE\big[ \Delta_{n}\big(Q\cap \bD_n(R)\big) \mid \mathcal{F}_{j_3} \, \big] \Big| \ge \tfrac{1}{2}R^{\alpha} \mid \mathcal{F}_{j_3} \, \Big] \\ \nonumber & \lesssim R^{-2\alpha} \, \bE\Big[\Big(\sum_{Q\in \mathcal{V}_{j_3}^{B}} \Delta_{n}\big(Q\cap \bD_n(R)\big) - \bE\big[ \Delta_{n}\big(Q\cap \bD_n(R)\big)\mid \mathcal{F}_{j_3} \, \big]\Big)^2 \mid \mathcal{F}_{j_3}\, \Big] \\ &= R^{-2\alpha} \,  \sum_{Q\in \mathcal{V}_{j_3}^{B}} \bE\Big[ \Big(\Delta_{n}\big(Q\cap \bD_n(R)\big) - \bE\big[ \Delta_{n}\big(Q\cap \bD_n(R)\big)\mid \mathcal{F}_{j_3} \, \big]\Big)^2 \mid \mathcal{F}_{j_3}\, \Big] \, .
		\end{align}
		On the event $\mathcal{G}_1$, we have for all $R$ large enough
		\begin{equation*}
			\max\Big\{ \big|\Delta_{n}\big(Q\cap \bD_n(R)\big)\big|,\big|\Delta_n(Q)\big|\Big\} \le \mu_n(Q) \lesssim R^{2\alpha-2+2\eta} \, .
		\end{equation*}
		Bounding each term in the sum~\eqref{eq:large_deviations_lower_bound_bad_cubes_do_not_contribute_after_centering_and_markov} with the above observation and recalling that $\big|\mathcal{V}_{j_3}^{B}\big|\lesssim R^{2-\alpha-\eta}$, we see that
		\begin{align*}
			\mathbf{1}_{\mathcal{G}_1} \, \bP\Big[ \, \Big|\sum_{Q\in \mathcal{V}_{j_3}^{B}} \Delta_{n}\big(Q\cap \bD_n(R)\big) &- \bE\big[ \Delta_{n}\big(Q\cap \bD_n(R)\big) \mid \mathcal{F}_{j_3} \, \big] \Big| \ge \tfrac{1}{2}R^{\alpha} \mid \mathcal{F}_{j_3} \, \Big] \\ &\lesssim R^{-2\alpha}\, \big|\mathcal{V}_{j_3}^{B}\big| \,  R^{4\alpha-4+4\eta} \lesssim R^{\alpha-2+3\eta}\, .
		\end{align*}
		Plugging this into~\eqref{eq:large_deviations_lower_bound_bad_cubes_do_not_contribute_after_centering}, we finally get that
		\[
		\mathbf{1}_{\mathcal{G}_1} \, \bP\Big[ \, \big|X_{j_3}^B\big| \ge R^{\alpha} \mid \mathcal{F}_{j_3} \, \Big] \lesssim R^{\alpha-2+3\eta}
		\]
		As $\alpha<2$, for small enough $\eta>0$ the right-hand side tends to zero as $R\to \infty$. In particular, for $R$ large enough, it is smaller than $\frac{1}{2}$ and we are done.
	\end{proof}
	Finally, we estimate from below the (conditional) probability that $X_{j_3}^G$ is large. 
	\begin{lemma}
		\label{lemma:large_deviations_lower_bound_good_cubes_give_correct_probability_for_disc}
		For all $R$ large enough, uniformly as $n\to \infty$, we have
		\[
		\mathbf{1}_{\mathcal{G}_1} \, \bP\Big[ \, X_{j_3}^G \ge R^{\alpha+\eta/2} \mid \mathcal{F}_{j_3} \, \Big] \ge \mathbf{1}_{\mathcal{G}_1} \, \exp\Big(-R^{3\alpha-2+4\eta}\Big)
		\]
		where $\mathcal{G}_1$ is given by~\eqref{eq:def_of_good_event_large_deviations_lower_bound}.
	\end{lemma} 
	\begin{proof}
		We will consider the event
		\begin{equation*}
			\mathcal{A}^\prime = \Big\{ \forall Q\in \mathcal{V}_{j_3}^{G} \, : \, \mu_n(Q) = \mu_n\big(Q\cap \bD_{n}(R)\big) \Big\}\, .
		\end{equation*}
		By the definition~\eqref{eq:def_of_good_boundary_cube} of good boundary squares, for each $Q\in \mathcal{V}_{j_3}^{G}$ there exists a dyadic square of side-length $$\frac{1}{4} \times 10^{-5} \times 2^{-j_3}$$ which is strictly contained in $Q \cap  \bD_n(R)$. Thus, by combining the (conditional) self-similarity property (Proposition~\ref{prop:self_similarity_and_conditional_independence}) with Lemma~\ref{lemma:overcrowding_probability_in_dyadic_cube}, we have the lower bound
		\begin{equation*}
			\bP\Big[ \mu_n(Q) = \mu_n\big(Q\cap \bD_{n}(R)\big) \mid \mathcal{F}_{j_3} \, \Big] \ge \exp\Big(-c_\beta \big(\mu_n(Q)\big)^2\Big) \, .
		\end{equation*}
		Furthermore, Proposition~\ref{prop:self_similarity_and_conditional_independence} also asserts that conditional on $\mathcal{F}_{j_3}$, the random variables $\big\{\mu_n\big(Q\cap \bD_n(R)\big)\big\}_{Q\in \mathcal{V}_{j_3}^{G}}$ are independent, which implies that
		\begin{equation}
			\label{eq:large_deviations_lower_bound_good_cubes_give_correct_probability_for_disc_lower_bound_all_in}
			\bP\big[\mathcal{A}^\prime \mid \mathcal{F}_{j_3} \, \big] \ge \prod_{Q\in \mathcal{V}_{j_3}^{G}} \exp\Big(-c_\beta \big(\mu_n(Q)\big)^2\Big) = \exp\Big(-c_\beta \sum_{Q\in \mathcal{V}_{j_3}^{G}}\big(\mu_n(Q)\big)^2\Big) \, .
		\end{equation}
		In view of~\eqref{eq:large_deviations_lower_bound_good_cubes_give_correct_probability_for_disc_lower_bound_all_in}, the lemma would follow once we show that both
		\begin{equation}
			\label{eq:large_deviations_lower_bound_good_cubes_give_correct_probability_for_disc_lower_bound_sum_of_squares}
			\mathbf{1}_{\mathcal{G}_1} \sum_{Q\in \mathcal{V}_{j_3}^{G}}\big(\mu_n(Q)\big)^2 \ge \mathbf{1}_{\mathcal{G}_1}  \exp\Big(-R^{3\alpha-2 + 4\eta}\Big)
		\end{equation}
		and
		\begin{equation}
			\label{eq:large_deviations_lower_bound_good_cubes_give_correct_probability_for_disc_inclusion_of_events}
			\mathcal{G}_1 \cap \mathcal{A}^\prime \subset \mathcal{G}_1 \cap \Big\{X_{j_3}^G \ge R^{\alpha+\eta/2}  \Big\}
		\end{equation}
		hold for all $R$ large enough, uniformly as $n\to \infty$. Indeed, on the event $\mathcal{G}_1$ we have the lower bound
		\begin{equation*}
			\mu_n(Q) \ge n \, {\tt Leb} (Q) - R^{\eta/2} \ge R^{2\alpha-2+2\eta} - R^{\eta/2}
		\end{equation*}
		for all $Q\in \mathcal{V}_{j_3}$. In particular, as $\alpha>1$, we get that for all $R$ large enough
		\[
		\sum_{Q\in \mathcal{V}_{j_3}^{G}}\big(\mu_n(Q)\big)^2 \gtrsim |\mathcal{V}_{j_3}^{G}| \, R^{4\alpha-4 + 4\eta} \stackrel{\eqref{eq:lower_and_upper_bounds_on_good_bad_boundary_squares}}{\gtrsim} R^{2-\alpha-\eta} \, R^{4\alpha-4 + 4\eta} = R^{3\alpha-2 + 3\eta}
		\]
		which gives~\eqref{eq:large_deviations_lower_bound_good_cubes_give_correct_probability_for_disc_lower_bound_sum_of_squares}. To verify~\eqref{eq:large_deviations_lower_bound_good_cubes_give_correct_probability_for_disc_inclusion_of_events}, we note that on the event $\mathcal{G}_1 \cap \mathcal{A}^\prime$ we have
		\begin{align*}
			X_{j_3}^G & = \sum_{Q\in \mathcal{V}_{j_3}^{G}} \Delta_{n}\big(Q\cap \bD_n(R)\big) = \sum_{Q\in \mathcal{V}_{j_3}^{G}} \Big(\mu_{n}\big(Q\cap \bD_n(R)\big) - n \, {\tt Leb}\big(Q\cap \bD_n(R)\big)\Big) \\ &\stackrel{\text{by}~\mathcal{A}^\prime}{=} \sum_{Q\in \mathcal{V}_{j_3}^{G}} \Big(\mu_{n}(Q) - n \, {\tt Leb}\big(Q\cap \bD_n(R)\big)\Big) \stackrel{{\text{by}~\mathcal{G}_1}}{\ge} \sum_{Q\in \mathcal{V}_{j_3}^{G}} \Big( n \, {\tt Leb}\big(Q\cap \bD_n(R)^{c}\big) - R^{\eta/2}\Big) \, .
		\end{align*}
		By the definition of good boundary squares~\eqref{eq:def_of_good_boundary_cube}, we have
		\[
		\sum_{Q\in \mathcal{V}_{j_3}^{G}} n \, {\tt Leb}\big(Q\cap \bD_n(R)^{c}\big) \gtrsim |\mathcal{V}_{j_3}^{G}| \, R^{2\alpha-2+2\eta} \stackrel{\eqref{eq:lower_and_upper_bounds_on_good_bad_boundary_squares}}{\gtrsim} R^{2-\alpha-\eta} \, R^{2\alpha-2+2\eta} = R^{\alpha+\eta} \, .
		\]
		Recalling that $\alpha\ge 1$, we can combine the above with the naive bound
		\[
		\sum_{Q\in \mathcal{V}_{j_3}^{G}} R^{\eta/2} \leq R^{\eta/2} \, |\mathcal{V}_{j_3}^B| \stackrel{\eqref{eq:lower_and_upper_bounds_on_good_bad_boundary_squares}}{\lesssim}  R^{2-\alpha - \eta/2} \le R^{\alpha}
		\]
		and get that on the event $\mathcal{G}_1 \cap \mathcal{A}^\prime$, for all $R$ large enough we have 
		$
		X_{j_3}^G \ge R^{\alpha+\eta/2}\, .
		$
		This proves~\eqref{eq:large_deviations_lower_bound_good_cubes_give_correct_probability_for_disc_inclusion_of_events}, and together with~\eqref{eq:large_deviations_lower_bound_good_cubes_give_correct_probability_for_disc_lower_bound_all_in} and~\eqref{eq:large_deviations_lower_bound_good_cubes_give_correct_probability_for_disc_lower_bound_sum_of_squares} we are done.
	\end{proof}
	\subsection{Proof of Theorem~\ref{thm:Micro_JLM}: the lower bound for \texorpdfstring{$1<\alpha<2$}{1 < alpha < 2}}
	Consider the events 
	\[
	\mathcal{E}_B = \Big\{\big|X_{j_3}^B\big| \le R^{\alpha}\Big\}  \, , \qquad  \mathcal{E}_{G} = \Big\{X_{j_3}^G \ge R^{\alpha+\eta/2}\Big\} \, ,
	\]
	and recall the decomposition of $\Delta_{n}\big(\bD_n(R)\big)$ given by~\eqref{eq:large_deviations_lower_bound_decomposition_of_disc}. By~\eqref{eq:large_deviations_lower_bound_good_event_implies_small_inner_disc}, on the event $\mathcal{G}_1$ we have for all $R$ large enough
	\[
	\Big|\sum_{j=1}^{j_3} \sum_{Q\in \mathcal{U}_j} \Delta_n(Q)\Big| \le R^{\alpha-\eta/2} \, ,
	\]
	which implies that on $\mathcal{G}_1 \cap \mathcal{E}_B \cap \mathcal{E}_{G}$, we have 
	\[
	\Delta_n\big(\bD_n(R)\big) \ge R^{\alpha+\eta/2} \, .
	\]
	In particular, for all $R$ large, we have the inclusion		
	\begin{equation}
		\label{eq:large_deviations_lower_bound_event_of_large_disc}
		\Big\{\Delta_{n}\big(\bD_n(R)\big) \ge R^{\alpha} \Big\} \supseteq \mathcal{G}_1 \cap \mathcal{E}_B \cap \mathcal{E}_{G} \, .
	\end{equation}
	Furthermore, conditionally on $\mathcal{F}_{j_3}$, the random variables $X_{j_3}^B$ and $X_{j_3}^G$ are independent. Therefore, by the law of total expectation
	\begin{align*}
		\bP\Big[\mathcal{G}_1 \cap \mathcal{E}_B \cap \mathcal{E}_{G}\Big] &= \bE\Big[\mathbf{1}_{\mathcal{G}_1} \, \bP\big[ \mathcal{E}_B \cap \mathcal{E}_{G} \mid \mathcal{F}_{j_3} \, \big] \, \Big] = \bE\Big[\mathbf{1}_{\mathcal{G}_1} \, \bP\big[ \mathcal{E}_B \mid \mathcal{F}_{j_3} \, \big] \, \bP\big[ \mathcal{E}_G \mid \mathcal{F}_{j_3} \, \big] \, \Big] \\ &\stackrel{\text{Lemma}~\ref{lemma:large_deviations_lower_bound_bad_cubes_do_not_contribute}}{\ge} \frac{1}{2} \bE\Big[\mathbf{1}_{\mathcal{G}_1} \, \bP\big[ \mathcal{E}_G \mid \mathcal{F}_{j_3} \, \big] \, \Big]  \\ &\stackrel{\text{Lemma}~\ref{lemma:large_deviations_lower_bound_good_cubes_give_correct_probability_for_disc}}{\ge} \frac{1}{2}  \exp\Big(-R^{3\alpha-2+4\eta}\Big) \, \bP[\mathcal{G}_1] \, .
	\end{align*} 
	Claim~\ref{claim:large_deviations_lower_bound_good_event_is_likely} implies that for all $R$ large enough, 
	$
	\bP[\mathcal{G}_1] \ge \frac{1}{2}
	$
	uniformly as $n\to \infty$, and hence
	\[
	\bP\Big[\mathcal{G}_1 \cap \mathcal{E}_B \cap \mathcal{E}_{G}\Big] \ge \frac{1}{4} \exp\Big(-R^{3\alpha-2+4\eta}\Big) \, .
	\]
	In view of~\eqref{eq:large_deviations_lower_bound_event_of_large_disc}, this gives the desired lower bound in the range $1<\alpha<2$. 
	\qed
	\section{Lower bound for \texorpdfstring{$\frac{1}{2} < \alpha < 1$}{1/2< alpha < 1}}
	\label{sec:LB_moderate_deviations}
	Contrary to the lower bound from Section~\ref{sec:LB_large_deviations}, here we cannot just estimate from below the probability of a single explicit event that we construct, as the entropy contribution to the discrepancy is non-trivial in the range $\tfrac{1}{2}<\alpha<1$ (recall the basic intuition from Section~\ref{sec:ideas_from_the_proof}). To overcome this difficulty, we will reduce the desired lower bound to a question about ``moderate deviations'' for a sum on independent bounded random variables, and to deal with the latter we shall apply the a standard change of measure argument. The following proposition is tailor-made for our needs.
	\begin{proposition}
		\label{prop:moderate_deviations_lower_bound_bounded_random_variables}
		Let $X_1,\ldots,X_N$ be independent, mean-zero random variables and denote by $S_N = X_1+\ldots+X_N$. Assume there exists $\La<\infty$ and $\sigma>0$ so that
		\begin{equation*}
			\bP\big[|X_j| \le \La\big] = 1 \quad \text{and } \quad\bE\big[X_j^2\big]\ge \sigma>0,
		\end{equation*}
		for all $j=1,\ldots,N$. Then for all $\gamma\in\big(\frac{1}{2},1\big)$ and for all $\eps>0$ there exists $N_0=N_0(\eps,\gamma,\La,\sigma)$ so that for all $N\ge N_0$ we have
		\[
		\bP\Big[S_N \ge N^\gamma\Big] \ge \exp\Big(-N^{2\gamma-1+\eps}\Big) \, .
		\]
	\end{proposition} 
	As we already mentioned, the proof follows from a change of measure argument, similar to the one applied in the proof of the lower bound in Cram\'er's theorem from large deviations theory, see for instance~\cite[Section~2.2]{DemboZeitouniLDP}. For completeness, we provide the proof of Proposition~\ref{prop:moderate_deviations_lower_bound_bounded_random_variables} below in an appendix, see Section~\ref{sec:moderate_deviations_lower_bound}.
	\subsection{Pushing the discrepancy to the boundary}
	Following the intuition presented in Section~\ref{sec:ideas_from_the_proof}, we expect the excess in the discrepancy to be present in a thin layer around the boundary of $\bD_n(R)$. We start by recalling some notations from Section~\ref{sec:UB_moderate_deviations} which we will use here. Let $\ell\ge 1$ be the integer~\eqref{eq:def_of_ell} and recall that $j_2>\ell$ is an integer such that
	\[
	\frac{1}{\sqrt{n}} \le 2^{-j_2} < \frac{2}{\sqrt{n}} \, ,
	\]
	see~\eqref{eq:def_of_j_2}. We further recall the exposure martingale with respect to the filtration $\{\mathcal{F}_k\}$, given by
	\begin{equation*}
		M_k = \bE\big[\Delta_n\big(\bD_n(R)\big) \mid \mathcal{F}_k\big] = \sum_{j=\ell}^{k} \sum_{Q\in \mathcal{U}_j} \Delta_{n}(Q) + \sum_{Q\in \mathcal{V}_{k}} p(Q)\Delta_n(Q) \, ,
	\end{equation*}
	and the decomposition $\Delta_{n}\big(\bD_n(R)\big) = M_{k} + B_{k}$, where
	\begin{equation*}
		B_{k} = \sum_{Q\in \mathcal{V}_{k}} \mu_n\big(Q\cap \bD_n(R)\big) - p(Q) \mu_n(Q) \, .
	\end{equation*}
	First, we show that with high probability, the martingale term $M_{j_2}$ does not contribute to the excess in the discrepancy.
	\begin{lemma}		\label{lemma:moderate_deviations_lower_bound_push_disc_to_boundary}
		We have
		\[
		\lim_{R\to \infty} \lim_{n\to \infty} \bP\Big[\big|M_{j_2}\big| \ge \sqrt{R} \, \log^2(R) \, \Big] = 0 \, .
		\]
	\end{lemma}
	\begin{proof}
		By Chebyshev's inequality, the lemma would follow once we show that
		\begin{equation}
			\label{eq:moderate_deviations_lower_bound_martingale_term_variance_bound}
			\Var\big[M_{j_2}\big] \le C_\beta R\,  \log^2(R) \, ,
		\end{equation}
		uniformly as $n\to \infty$. To get~\eqref{eq:moderate_deviations_lower_bound_martingale_term_variance_bound}, we recall the variance bound on dyadic squares, proved by Chatterjee (see~\cite[Lemma~3.8]{chatterjee}); For any $j\ge 1$ let $Q\in \mathcal{D}_j$, then
		\begin{equation}
			\label{eq:variance_bound_on_dyadic_cubes}
			\Var\big[\mu_n(Q)\big] \le C_\beta \log^2\big(4^{-j+1} n + 3\big) \, .
		\end{equation}
		By the orthogonality of martingale increments, we have
		\begin{align}
			\nonumber \label{eq:martingale_propoerty_for_conditional_variance}
			\Var\big[M_k\big] &= \bE\big[\Var\big[M_k \mid \mathcal{F}_{k-1}\big]\big] + \Var\big[\bE\big[M_k \mid \mathcal{F}_{k-1}\big]\big] \\ & =\bE\big[\Var\big[M_k \mid \mathcal{F}_{k-1}\big]\big] + \Var\big[M_{k-1}\big] \, .
		\end{align}
		For all $k>\ell$, conditional on $\mathcal{F}_{k-1}$ the random variables $\{\mu_n(Q)\}_{Q\in \mathcal{U}_{k} \cup \mathcal{V}_k}$ are independent by Proposition~\ref{prop:self_similarity_and_conditional_independence}, unless they are both subsets of the same boundary square $Q^\prime \in \mathcal{V}_{k-1}$. Hence, the Cauchy-Schwarz inequality yields
		\begin{align*}
			\Var\big[M_k \mid \mathcal{F}_{k-1}\big] & = \sum_{Q^\prime\in \mathcal{V}_{k-1}} \Var \Big[\sum_{\substack{Q\subset Q^\prime \\ Q\in \mathcal{U}_{k} \cup \mathcal{D}_{k}}} p(Q) \mu_n(Q) \mid \mathcal{F}_{k-1} \, \Big] \\
			& \leq \sum_{Q^\prime\in \mathcal{V}_{k-1}} \bigg(\sum_{\substack{Q\subset Q^\prime \\ Q\in \mathcal{U}_{k} \cup \mathcal{D}_{k}}} \sqrt{\Var \Big[ \mu_n(Q) \mid \mathcal{F}_{k-1} \, \Big]} \, \bigg)^2 \, .
		\end{align*}
		Combining Proposition~\ref{prop:self_similarity_and_conditional_independence} with the bound~\eqref{eq:variance_bound_on_dyadic_cubes}, we see that
		\[
		\Var \Big[ \mu_n(Q) \mid \mathcal{F}_{k-1} \, \Big] \le C_\beta \log^2\big( \mu_n(Q^\prime) + 3 \big)  \, .
		\]
		Using the above together with the concavity of the map $x\mapsto \log^2(x+3)$, we arrive at the bound
		\begin{align*}
			\bE\big[\Var\big[M_k \mid \mathcal{F}_{k-1}\big]\big] & \le C_\beta \sum_{Q^\prime\in \mathcal{V}_{k-1}} \bE\big[\log^2\big( \mu_n(Q^\prime) + 3 \big)\big] \\
			& \le  C_\beta \sum_{Q^\prime\in \mathcal{V}_{k-1}} \log^2\big(4^{-k+1} n + 3\big) \le C_\beta \log^2(R) \, |\mathcal{V}_{k-1}|
		\end{align*}
		for all $k>\ell$. Plugging into~\eqref{eq:martingale_propoerty_for_conditional_variance} and iterating for $\ell<k\le j_2$, we arrive at the bound
		\begin{align*}
			\Var\big[M_k\big] &\le C_\beta \log^2(R) \sum_{k=\ell+1}^{j_2} |\mathcal{V}_{k}| + \Var\big[M_\ell\big]  \\ &\le C_\beta \log^2(R) \sum_{k=\ell+1}^{j_2} 2^{k-\ell} + \Var\big[M_\ell\big] \le C_\beta \log^2(R) R + \Var\big[M_\ell\big] \, . 
		\end{align*}
		Finally, by our choice of $\ell \ge 1$ given by~\eqref{eq:def_of_ell}, we have $|\mathcal{U}_\ell \cup \mathcal{V}_{\ell}|\le 10$ which, by combining~\eqref{eq:variance_bound_on_dyadic_cubes} and the Cauchy-Schwarz inequality, implies that
		\begin{align*}
			\Var\big[M_\ell\big] &= \bE\Big[\Big(\sum_{Q\in \mathcal{U}_\ell \cup \mathcal{V}_\ell} p(Q) \Delta_n(Q)\Big)^2\Big] \\ &\le \bE\Big[|\mathcal{U}_\ell \cup \mathcal{V}_\ell| \sum_{Q\in \mathcal{U}_\ell \cup \mathcal{V}_\ell} \Delta_n(Q)^2 \Big] = |\mathcal{U}_\ell \cup \mathcal{V}_\ell| \sum_{Q\in \mathcal{U}_\ell \cup \mathcal{V}_\ell} \Var\big[\mu_n(Q)\big] \le C_\beta \log^2(R) \, .
		\end{align*}
		Altogether, we have proved the variance bound~\eqref{eq:moderate_deviations_lower_bound_martingale_term_variance_bound} and the lemma follows. 
	\end{proof}
	\subsection{Boundary squares}
	To deal with the discrepancy coming from boundary squares, it will be convenient to work under the typical event that none of them contain too many points. Indeed, considering the event
	\begin{equation}
		\label{eq:moderate_deviations_lower_bound_boundary_cubes_are_typical}
		\mathcal{G}_2 = \Big\{ \forall Q\in \mathcal{V}_{j_2} \, : \, \mu_n(Q) \le R^{\eps} \Big\}\, ,
	\end{equation}
	we have the following simple claim.
	\begin{claim}
		\label{claim:moderate_deviations_lower_bound_boundary_cubes_are_typical_with_high_probability}
		For all $\eps>0$ we have $\displaystyle \lim_{R\to \infty} \lim_{n\to \infty} \bP\big[\mathcal{G}_2\big] = 1$.
	\end{claim}
	\begin{proof}
		By our choice of $j_2$, for all $Q\in \mathcal{D}_{j_2}$ we have
		$
		\bE\big[\mu_n(Q)\big] \in [1,2)
		$
		and hence, if $\widetilde{Q}$ denotes the ancestor of $Q$ with 
		\[
		\bE\big[\mu_n\big(\widetilde{Q}\big)\big]\in \Big[\frac{R^{\eps}}{16},\frac{R^{\eps}}{4}\Big) \, ,
		\]
		which always exists since the expectations grow as multiples of four as we move up in the tree. We have the simple exclusions
		\[
		\Big\{ \mu_n(Q) \ge R^{\eps} \Big\} \subset \Big\{ \Delta_{n}(\widetilde{Q}) \ge \frac{1}{2}R^{\eps} \Big\} \subset \Big\{ \Delta_{n}(\widetilde{Q}) \ge R^{\eps/2} \Big\}
		\]
		for all $R$ large enough. By Lemma~\ref{lemma:concentration_on_dyadic_cubes}, we can bound the probability of the event on the right-hand side as
		\[
		\bP\Big[\Delta_{n}(\widetilde{Q}) \ge R^{\eps/2} \Big] \le \exp\Big(-c_{\eps,\beta} \, R^{\eps}\Big) \, .
		\]
		As $|\mathcal{V}_{j_2}| \lesssim R$, the desired claim follows from the union bound.
	\end{proof}
	Borrowing the terminology from Section~\ref{sec:LB_large_deviations}, we say that a square $Q\in \mathcal{V}_{j_2}$ is a \emph{good} boundary square if it satisfies~\eqref{eq:def_of_good_boundary_cube}. By Claim~\ref{claim:enough boundary_cubes_are_good}, there are at least
	$
	c_1 R
	$
	good boundary squares, with $c_1>0$ being an absolute constant.
	We will denote by $\mathcal{V}_{j_2}^{G}$, and by $\mathcal{V}_{j_2}^B = \mathcal{V}_{j_2} \setminus \mathcal{V}_{j_2}^G$. 
	
	To proceed, we will import a simple claim from~\cite{chatterjee}, essentially saying that typically there are many squares in $\mathcal{V}_{j_2}^G$ with a bounded number of points. By~\cite[Lemma~3.8]{chatterjee}, for all $Q\in \mathcal{D}_{j_2}$ we have
	\[
	\bE\big[\mu_n(Q)^2\big] \le L(\beta)
	\]
	where $L(\beta)$ is a positive integer that depends only on $\beta$. Denote by $\La = 1000 L(\beta)$, and consider the (random) set
	\begin{equation*}
		\mathcal{C}_0 = \Big\{ Q\in \mathcal{V}_{j_2}^G \, : \, 0< \mu_n(Q) \le \La \Big\} \, .
	\end{equation*} 
	We will also denote by $\mathcal{C}_1 = \mathcal{V}_{j_2}^G \setminus \mathcal{C}_0$, and so we have the decomposition
	\begin{equation}
		\label{eq:decomposition_of_boundary_cubes}
		\mathcal{V}_{j_2} = \mathcal{C}_0 \cup \mathcal{C}_1 \cup \mathcal{V}_{j_2}^B \, .
	\end{equation}
	\begin{lemma}
		\label{lemma:moderate_deviations_lower_bound_lemma_from_chatterjee}
		There exists $\delta=\delta(\beta)>0$ so that 
		\[
		\bP\big[|\mathcal{C}_0| \ge \delta R\big] \ge \frac{1}{4}
		\]
		for all $R$ large enough.
	\end{lemma}
	Lemma~\ref{lemma:moderate_deviations_lower_bound_lemma_from_chatterjee} was essentially proved in Chatterjee~\cite{chatterjee}, see relation (3.8) therein (although the formulation there is slightly different). For the convenience of the reader, we provide the proof below.
	\begin{proof}
		Recall that $j_2\ge \ell$ is given by~\eqref{eq:def_of_j_2} and that $L(\beta)$ is a constant such that $$\bE\big[\mu_n(Q)^2\big] \le L(\beta)$$ for all $Q\in \mathcal{D}_{j_2}$. Recalling that $\mathcal{V}_{j_2}^G$ are all the good boundary squares in $\mathcal{D}_{j_2}$, we set
		\begin{align*}
			&p_1 = \frac{1}{|\mathcal{V}_{j_2}^G|} \sum_{Q\in \mathcal{V}_{j_2}^G} \mu_n(Q) \, , \qquad \qquad \quad \ \,  p_2 = \frac{1}{|\mathcal{V}_{j_2}^G|} \sum_{Q\in \mathcal{V}_{j_2}^G} \mu_n(Q)^2 \, , \\ & q_1 = \frac{\#\big\{Q\in \mathcal{V}_{j_2}^G \, : \, \mu_n(Q)>0 \big\}}{|\mathcal{V}_{j_2}^G|} \, , \qquad q_2 = \frac{\#\big\{Q\in \mathcal{V}_{j_2}^G \, : \, \mu_n(Q) > \La \big\}}{|\mathcal{V}_{j_2}^G|}\, , 
		\end{align*}
		and note the relation $|\mathcal{C}_0| = |\mathcal{V}_{j_2}^G| \big(q_1-q_2\big)$. By our choice of $j_2$, we have $$\bE[p_1] = n 4^{-j_2} \ge 1 \, , $$
		whereas Theorem 3.10 in~\cite{chatterjee} implies that 
		\[
		\Var\big[p_1\big] \stackrel{\text{Claim}~\ref{claim:enough boundary_cubes_are_good}}{\lesssim} \frac{1}{R^2} \Var\Big[\mu_n\Big(\bigcup_{Q\in \mathcal{V}_{j_2}^G} Q\Big)\Big] \lesssim \frac{\log^2(R)}{R}\, .
		\] 
		With that, Chebyshev's inequality implies that for all large enough $R$ 
		\[
		\bP\Big[p_1 \ge \frac{1}{2} \, \Big] \ge 0.99 \, .
		\]
		Furthermore, we have $\bE\big[p_2\big] \le L(\beta)$, which implies that
		$
		\bP\big[p_2 \ge 4 L(\beta)\big] \le \frac{1}{4} \, .
		$
		We conclude that
		\begin{equation*}
			\bP\Big[p_1 \ge \frac{1}{2} \, , p_2 < 4L(\beta) \Big] \ge \frac{1}{2} 
		\end{equation*}
		for all $R$ large enough. By the Cauchy-Schwarz inequality, we see that
		\[
		q_1 \ge \frac{p_1^2}{p_2}
		\]
		which in turn implies that
		\begin{equation}
			\label{eq:lower_bound_on_tail_of_q_1}
			\bP\Big[q_1 \ge \frac{1}{16 L(\beta)} \Big] \ge \bP\Big[p_1 \ge \frac{1}{2} \, , p_2 < 4L(\beta) \Big] \ge \frac{1}{2}\, . 
		\end{equation}
		Next, we turn to show that typically $q_2$ is not too large. Indeed, Markov's inequality implies that
		\[
		\bE[q_2] = \frac{1}{|\mathcal{V}_{j_2}^G|} \sum_{Q\in \mathcal{V}_{j_2}^G} \bP\big[\mu_n(Q) > \La\big] \le \frac{2}{\La}\, ,
		\] 
		which gives
		\[
		\bP\Big[q_2 \ge \frac{10}{\La} \, \Big] \le \frac{1}{5} \, .
		\]
		Combining with~\eqref{eq:lower_bound_on_tail_of_q_1}, we see that
		\[
		\bP\Big[q_1 \ge \frac{1}{16 L(\beta)} \, , \, q_2 < \frac{10}{\La} \, \Big]  \ge \frac{1}{4}\, .
		\]
		To conclude the proof of the lemma, it remains to observe that by our choice of $\La = 1000 L(\beta)$, on the event $\big\{q_1 \ge \frac{1}{16 L(\beta)} \, , \, q_2 < \frac{10}{\La} \big\}$ we have
		\[
		|\mathcal{C}_0| = |\mathcal{V}_{j_2}^G| \big(q_1-q_2\big) \stackrel{\text{Claim}~\ref{claim:enough boundary_cubes_are_good}}{\gtrsim} R \, \Big(\frac{1}{16 L(\beta)} - \frac{10}{\La}\Big) \ge \frac{R}{32 L(\beta)} 
		\]
		and we are done.
	\end{proof}
	\subsection{Creating the discrepancy from the boundary}
	We wish to estimate from below the probability that the boundary term $B_{j_2}$ is large. To do so, we split the boundary contribution into two different sums
	\begin{equation*}
		B_{j_2}^\prime = \sum_{Q\in \mathcal{C}_0} \mu_n\big(Q\cap \bD_n(R)\big) - p(Q) \mu_n(Q) \qquad \text{and} \qquad B_{j_2}^{\prime\prime} = \sum_{Q\in \mathcal{C}_1\cup \mathcal{V}_{j_2}^B} \mu_n\big(Q\cap \bD_n(R)\big) - p(Q) \mu_n(Q) \, .
	\end{equation*}
	By~\eqref{eq:decomposition_of_boundary_cubes}, we have $B_{j_2} = B_{j_2}^\prime + B_{j_2}^{\prime\prime}$. Furthermore, by Proposition~\ref{prop:self_similarity_and_conditional_independence}, conditional on $\mathcal{F}_{j_2}$ the random variables $B_{j_2}^\prime$ and $B_{j_2}^{\prime\prime}$ are independent. 
	\begin{lemma}
		\label{lemma:moderate_deviations_lower_bound_bad_boundary_does_not_contribute}
		For all $\eps>0$ and for all $R$ large enough, we have
		\[
		\mathbf{1}_{\mathcal{G}_2} \, \bP\Big[ \big|B_{j_2}^{\prime\prime}\big| < R^{\frac{1}{2} + 2\eps} \mid \mathcal{F}_{j_2} \, \Big] \ge \frac{1}{2} \, \mathbf{1}_{\mathcal{G}_2}
		\]
		where $\mathcal{G}_2$ is given by~\eqref{eq:moderate_deviations_lower_bound_boundary_cubes_are_typical}.
	\end{lemma}
	\begin{proof}
		By Proposition~\ref{prop:self_similarity_and_conditional_independence}, conditional on $\mathcal{F}_{j_2}$ the random variables 
		\[
		\Big\{\mu_n\big(Q\cap \bD_n(R)\big) - p(Q) \mu_n(Q)\Big\}_{Q\in \mathcal{V}_{j_2}}
		\]
		are independent and all having mean-zero. In particular, we get the bound
		\begin{multline*}
			\Var\Big[\sum_{Q\in \mathcal{C}_1\cup \mathcal{V}_{j_2}^B} \mu_n\big(Q\cap \bD_n(R)\big) - p(Q) \mu_n(Q) \mid \mathcal{F}_{j_2} \, \Big] \\  = \sum_{Q\in \mathcal{C}_1\cup \mathcal{V}_{j_2}^B} \Var\Big[ \mu_n\big(Q\cap \bD_n(R)\big) - p(Q) \mu_n(Q) \mid \mathcal{F}_{j_2} \, \Big] \le \sum_{Q\in \mathcal{C}_1\cup \mathcal{V}_{j_2}^B} \mu_n(Q)^2\, ,
		\end{multline*}
		which, by the definition of the event $\mathcal{G}_2$, gives
		\[
		\mathbf{1}_{\mathcal{G}_2} \, \Var\Big[\sum_{Q\in \mathcal{C}_1\cup \mathcal{V}_{j_2}^B} \mu_n\big(Q\cap \bD_n(R)\big) - p(Q) \mu_n(Q) \mid \mathcal{F}_{j_2} \, \Big] \le R^{2\eps} \, |\mathcal{V}_{j_2}| \lesssim R^{1+2\eps}\, .
		\]
		Chebyshev's inequality gives
		\[
		\mathbf{1}_{\mathcal{G}_2} \, \bP\Big[ \big|B_{j_2}^{\prime\prime}\big| \ge R^{\frac{1}{2} + 2\eps} \mid \mathcal{F}_{j_2} \, \Big] \leq R^{-1-4\eps} \, \mathbf{1}_{\mathcal{G}_2} \, \Var\big[B_{j_2}^{\prime\prime}\mid \mathcal{F}_{j_2}\big] \lesssim R^{-2\eps}\, ,
		\]
		and the lemma follows by taking $R$ to be large enough.
	\end{proof}
	Finally, we apply our moderate deviations lower bound (Proposition~\ref{prop:moderate_deviations_lower_bound_bounded_random_variables}) to show that $B_{j_2}^\prime$ contributes the correct probability to the desired lower bound. 
	\begin{lemma}
		\label{lemma:moderate_deviations_lower_bound_good_boundary_correct_probabaility}
		For all $\eps>0$ and for all $R$ large enough, uniformly as $n\to \infty$, we have
		\[
		\mathbf{1}_{\{|\mathcal{C}_0| \ge \delta R \}} \, \bP\Big[ B_{j_2}^{\prime} \ge  R^{\alpha+\eps}  \mid \mathcal{F}_{j_2} \, \Big] \ge \mathbf{1}_{\{|\mathcal{C}_0| \ge \delta R \}} \, \exp\Big(-R^{2\alpha-1+6\eps}\Big)
		\]
		where $\delta>0$ is the constant from Lemma~\ref{lemma:moderate_deviations_lower_bound_lemma_from_chatterjee}.
	\end{lemma}
	\begin{proof}
		When $|\mathcal{C}_0| \ge \delta R$, we have for $R$ large enough that $|\mathcal{C}_0|^{\alpha+2\eps} \ge R^{\alpha+\eps}$, and hence
		\begin{equation}
			\label{eq:moderate_deviations_lower_bound_lower_bound_on_good_boundary_cubes_contribution_with_random_size}
			\mathbf{1}_{\{|\mathcal{C}_0| \ge \delta R \}} \, \bP\Big[ B_{j_2}^{\prime} \ge  R^{\alpha+\eps}  \mid \mathcal{F}_{j_2} \, \Big] \ge \mathbf{1}_{\{|\mathcal{C}_0| \ge \delta R \}} \, \bP\Big[ B_{j_2}^{\prime} \ge  |\mathcal{C}_0|^{\alpha+2\eps}  \mid \mathcal{F}_{j_2} \, \Big] \, .
		\end{equation}
		To give a lower bound on the right-hand side, we note once more that conditional on $\mathcal{F}_{j_2}$, 
		\[
		B_{j_2}^\prime = \sum_{Q\in \mathcal{C}_0} \mu_n\big(Q\cap \bD_n(R)\big) - p(Q) \mu_n(Q) 
		\]
		is a sum of independent mean-zero random variables, all uniformly bounded by $\La$. Furthermore, since $\mathcal{C}_0\subset \mathcal{V}_{j_2}^G$, we have the simple lower bound (see~\cite[Lemma~3.15]{chatterjee})
		\[
		\Var\big[\mu_n\big(Q\cap \bD_n(R)\big) \mid \mathcal{F}_{j_2} \, \big] \ge \sigma(\beta) >0,
		\]
		where $\sigma(\beta)$ depends only on $\beta>0$. Bringing everything together, we apply Proposition~\ref{prop:moderate_deviations_lower_bound_bounded_random_variables} with $N = |\mathcal{C}_0|$ and $\gamma= \alpha+2\eps$ for $\eps>0$ small, and obtain that for $R$ large enough
		\[
		\bP\Big[ B_{j_2}^{\prime} \ge  |\mathcal{C}_0|^{\alpha+2\eps}  \mid \mathcal{F}_{j_2} \, \Big] \ge \exp\Big(-|\mathcal{C}_0|^{2\alpha - 1 + 5\eps}\Big) \, .
		\]
		As $|\mathcal{C}_0| \le |\mathcal{V}_{j_2}| \lesssim R$, plugging the above into~\eqref{eq:moderate_deviations_lower_bound_lower_bound_on_good_boundary_cubes_contribution_with_random_size} gives
		\[
		\mathbf{1}_{\{|\mathcal{C}_0| \ge \delta R \}} \, \bP\Big[ B_{j_2}^{\prime} \ge  R^{\alpha+\eps}  \mid \mathcal{F}_{j_2} \, \Big] \ge \mathbf{1}_{\{|\mathcal{C}_0| \ge \delta R \}} \, \exp\Big(-R^{2\alpha-1+6\eps}\Big)
		\]
		for all $R$ large, as desired.
	\end{proof}
	\subsection{Proof of Theorem~\ref{thm:Micro_JLM}: the lower bound for \texorpdfstring{$\frac{1}{2}<\alpha<1$}{1/2 < alpha < 1}} We denote by
	\begin{equation*}
		\mathcal{E}_1 = \big\{ |M_{j_2}| < \sqrt{R} \, \log^2(R)\big\} \, , \quad  \mathcal{E}_2 = \big\{ |B_{j_2}^{\prime\prime}| < R^{\frac{1}{2} + 2\eps} \big\}\, , \quad  \mathcal{E}_3 = \big\{B_{j_2}^{\prime} \ge R^{\alpha+\eps} \big\} \, , 
	\end{equation*}
	and note that $\mathcal{E}_1 \in \mathcal{F}_{j_2}$, while conditional on $\mathcal{F}_{j_2}$, $\mathcal{E}_2$ and $\mathcal{E}_3$ are independent events.
	By the decomposition
	\[
	\Delta_{n}\big(\bD_n(R)\big) = M_{j_2} + B_{j_2}  = M_{j_2} + B_{j_2}^\prime + B_{j_2}^{\prime\prime} \, ,
	\] 
	we have, for $\alpha>\frac{1}{2}$, the inclusion
	\begin{equation}
		\label{eq:moderate_deviations_lower_bound_inclusion_of_events}
		\Big\{\Delta_{n}\big(\bD_n(R)\big) \ge R^{\alpha} \Big\} \supseteq \mathcal{E}_1\cap \mathcal{E}_2 \cap \mathcal{E}_3 
	\end{equation}
	so it remains to give a lower bound on the probability of the event on the right-hand side. Indeed, the law of total expectation gives
	\begin{align}
		\nonumber \label{eq:moderate_deviations_lower_bound_after_total_expectation}
		\bP\Big[\mathcal{E}_1\cap \mathcal{E}_2 \cap \mathcal{E}_3 \Big] &= \bE\Big[ \mathbf{1}_{\mathcal{E}_1} \bP\big[\mathcal{E}_2 \mid \mathcal{F}_{j_2} \, \big] \, \bP\big[\mathcal{E}_3 \mid \mathcal{F}_{j_2} \, \big]\Big] \\ & \ge \bE\Big[ \mathbf{1}_{\mathcal{E}_1} \, \mathbf{1}_{\mathcal{G}_2} \, \bP\big[\mathcal{E}_2 \mid \mathcal{F}_{j_2} \, \big] \, \mathbf{1}_{\{|\mathcal{C}_0| \ge \delta R\}} \, \bP\big[\mathcal{E}_3 \mid \mathcal{F}_{j_2} \, \big]\Big]\, .
	\end{align}
	By Lemma~\ref{lemma:moderate_deviations_lower_bound_bad_boundary_does_not_contribute} we know that $$\mathbf{1}_{\mathcal{G}_2} \, \bP\big[\mathcal{E}_2 \mid \mathcal{F}_{j_2} \, \big] \ge \frac{1}{2} \, \mathbf{1}_{\mathcal{G}_2}\, ,$$ and Lemma~\ref{lemma:moderate_deviations_lower_bound_good_boundary_correct_probabaility} yields
	\begin{equation*}
		\mathbf{1}_{\{|\mathcal{C}_0| \ge \delta R\}} \, \bP\big[\mathcal{E}_3 \mid \mathcal{F}_{j_2} \, \big] \ge \mathbf{1}_{\{|\mathcal{C}_0| \ge \delta R \}} \, \exp\Big(-R^{2\alpha-1+6\eps}\Big)\, .
	\end{equation*}
	Plugging both bounds into~\eqref{eq:moderate_deviations_lower_bound_after_total_expectation}, we arrive at the inequality
	\[
	\bP\Big[\mathcal{E}_1\cap \mathcal{E}_2 \cap \mathcal{E}_3 \Big] \ge \frac{1}{2}  \, \exp\Big(-R^{2\alpha-1+6\eps}\Big) \, \bP\big[\mathcal{E}_1\cap \mathcal{G}_2 \cap \big\{|\mathcal{C}_0| \ge \delta R\big\}\big] \, .
	\]
	Combining Lemma~\ref{lemma:moderate_deviations_lower_bound_push_disc_to_boundary} with Claim~\ref{claim:moderate_deviations_lower_bound_boundary_cubes_are_typical_with_high_probability}, we see that
	\[
	\lim_{R\to \infty} \lim_{n\to \infty} \bP\big[\mathcal{E}_1\cap \mathcal{G}_2\big] = 1\, ,
	\]
	and together with Lemma~\ref{lemma:moderate_deviations_lower_bound_lemma_from_chatterjee}, we get that for all $R$ large enough
	\[
	\bP\big[\mathcal{E}_1\cap \mathcal{G}_2 \cap \big\{|\mathcal{C}_0| \ge \delta R\big\}\big]  \ge \frac{1}{5} \, ,
	\] 
	uniformly as $n\to\infty$. We found that
	\[
	\bP\Big[\mathcal{E}_1\cap \mathcal{E}_2 \cap \mathcal{E}_3 \Big] \ge \frac{1}{10}  \, \exp\Big(-R^{2\alpha-1+6\eps}\Big) 
	\]
	which, in view of~\eqref{eq:moderate_deviations_lower_bound_inclusion_of_events}, concludes the proof of the lower bound in this range.
	\qed
	\section{Overcrowding probabilities}
	\label{sec:overcrowding}
	In this section, we conclude the proof of Theorem~\ref{thm:Micro_JLM}. It remains to deal with the extreme overcrowding probabilities, that is, the regime $\alpha>2$. 
	\subsection{Proof of Theorem~\ref{thm:Micro_JLM}: the upper bound for \texorpdfstring{$\alpha>2$}{alpha > 2}}
	Since
	$$
	\bE\big[\mu_n\big(\bD_n(R)\big)\big] = \pi R^2 \, ,
	$$
	for all $R$ large enough we have the inclusion
	\begin{equation*}
		\Big\{ |\Delta_n\big(\bD_n(R)\big)| \ge R^\alpha \Big\}\subset \Big\{ \mu_n\big(\bD_n(R)\big) \ge \tfrac{1}{2} R^{\alpha} \Big\}\, .
	\end{equation*}
	Recalling that $\ell> 1$ is an integer such that
	\[
	\frac{R}{\sqrt{n}} \le 2^{-\ell} < 2\frac{R}{\sqrt{n}} \, ,
	\]
	and that $|\mathcal{V}_\ell |\le 4$, the union bound gives
	\begin{align*}	
		\bP\Big[|\Delta_n\big(\bD_n(R)\big)| \ge R^\alpha\Big] & \le \bP\Big[\mu_n\big(\bD_n(R)\big) \ge \frac{1}{2} R^{\alpha}\Big] \le \sum_{Q\in \mathcal{V}_\ell } \bP\Big[ \mu_n\big(Q\cap \bD_n(R)\big) \ge \frac{1}{8} R^\alpha \Big] \\
		& \le 4 \times \bP\Big[ \mu_n\big(Q\big) \ge \tfrac{1}{8} R^\alpha \Big]
	\end{align*}
	where $Q\in \mathcal{D}_\ell$ is some dyadic square. In view of the above, the desired upper bound would follow once we show that
	\begin{equation}
		\label{eq:overcrowding_upper_bound_desired_bound}
		\bP\Big[ \mu_n\big(Q\big) \ge \tfrac{1}{8} R^\alpha \Big] \le \exp\Big(-R^{2\alpha-3\eps}\Big)
	\end{equation}
	for $Q\in \mathcal{D}_\ell$ and for all $\eps>0$. Indeed, let $1\le j_4<\ell$ be an integer such that
	\begin{equation}
		\label{eq:def_of_j_4}
		\frac{R^{\alpha/2}}{\sqrt{n}} \le 2^{-j_4} < 2\, \frac{R^{\alpha/2}}{\sqrt{n}}
	\end{equation}
	and denote by $Q^\prime\in \mathcal{D}_{j_4}$ the ancestor of $Q\in \mathcal{D}_\ell$. We note that 
	\begin{equation}
		\label{eq:overcrowding_upper_bound_expectation_of_top_cube}
		\bE\big[\mu_n(Q^\prime)\big] \in \big[R^\alpha,2R^\alpha\big) \, .
	\end{equation}
	Furthermore, Lemma~\ref{lemma:concentration_on_dyadic_cubes} implies that 
	\begin{align}\label{eq:overcrowding_upper_bound_split_of_probabilities}
		\bP\big[ \mu_n\big(Q\big) \ge \tfrac{1}{8} R^\alpha \big] & \le \bP\Big[ \mu_n\big(Q\big) \ge \tfrac{1}{8} R^\alpha \, , \ |\Delta_{n}(Q^\prime)| \le R^{\alpha-\eps} \Big] + \bP\Big[|\Delta_n(Q^\prime)| \ge R^{\alpha-\eps} \Big]\nonumber\\
		& \le \bP\Big[ \mu_n\big(Q\big) \ge \tfrac{1}{8} R^\alpha \, , \ |\Delta_{n}(Q^\prime)| \le R^{\alpha-\eps} \Big] + \exp\Big(-c \, R^{2\alpha-2\eps}\, \Big) \, .
	\end{align}
	To deal with the remaining probability on the right-hand side of~\eqref{eq:overcrowding_upper_bound_split_of_probabilities}, we note that~\eqref{eq:overcrowding_upper_bound_expectation_of_top_cube} implies that
	\[
	\Big\{\mu_n\big(Q\big) \ge \tfrac{1}{8} R^\alpha \, , \ |\Delta_{n}(Q^\prime)| \le R^{\alpha-\eps}\Big\} \subset \Big\{ \mu_n(Q) \ge \tfrac{1}{20}\, \mu_n(Q^\prime) \, , \  |\Delta_n(Q^\prime)| \le R^{\alpha-\eps} \Big\}
	\]
	for all $R$ large enough. Noting that $\ell - j_4 \ge c_\alpha \log R$, Proposition~\ref{prop:self_similarity_and_conditional_independence} combined with Lemma~\ref{lemma:overcrowding_probability_in_dyadic_cube} (applied with $\delta=1/20$) implies that
	\[
	\bP\Big[\mu(Q) \ge \frac{1}{20} \mu_n(Q^\prime) \mid \mathcal{F}_{j_4} \, \Big] \le \exp\Big(-c \, (\log R) \, \mu_n(Q^\prime)^2 \Big)\, .
	\]
	Altogether, we get that
	\begin{align*}
		\bP\Big[\mu_n\big(Q\big) \ge \tfrac{1}{8} R^\alpha \, , \ |\Delta_{n}(Q^\prime)| \le R^{\alpha-\eps}\Big]  &\le \bE\Big[\mathbf{1}_{\{|\Delta_{n}(Q^\prime)| \le R^{\alpha-\eps} \}} \bP\big[\mu(Q) \ge \tfrac{1}{20} \mu_n(Q^\prime) \mid \mathcal{F}_{j_4} \, \big]\Big] \\ &\le \bE\Big[\mathbf{1}_{\{|\Delta_{n}(Q^\prime)| \le R^{\alpha-\eps} \}} \exp\Big(-c(\log R) \, \mu_n(Q^\prime)^2 \Big)\Big]  \\ &\le \exp\Big(-c \, (\log R)\, R^{2\alpha}\Big) 
	\end{align*}
	which, together with~\eqref{eq:overcrowding_upper_bound_split_of_probabilities}, proves the bound~\eqref{eq:overcrowding_upper_bound_desired_bound} and we are done.
	\qed
	\subsection{Proof of Theorem~\ref{thm:Micro_JLM}: the lower bound for \texorpdfstring{$\alpha>2$}{alpha > 2}}
	Let $j_5\ge \ell$ be an integer such that
	\[
	\frac{1}{8}\frac{R}{\sqrt{n}} \le 2^{-j_5} < \frac{1}{4}\frac{R}{\sqrt{n}}\, .
	\]
	As $\bD_n(R)$ is a disk of radius $R/\sqrt{n}$, it necessarily contains a dyadic square $Q\in \mathcal{D}_{j_5}$, and we have the inclusion
	\begin{equation}
		\label{eq:overcrowding_lower_bound_inclusion_of_events}
		\Big\{ \mu_n(Q) \ge 2R^\alpha\Big\} \subset \Big\{ \Delta_n\big(\bD_n(R)\big)\ge R^\alpha \Big\} \, .
	\end{equation} 
	For $\alpha>2$, let $j_6<j_5$ be an integer such that
	\[
	2 \, \frac{R^{\alpha/2}}{\sqrt{n}} \le 2^{-j_6} < 4 \, \frac{R^{\alpha/2}}{\sqrt{n}}
	\]
	and denote by $Q^\prime\in \mathcal{D}_{j_6}$ the ancestor of $Q$. Since $\bE\big[\mu_n(Q^\prime)\big] \in (4R^\alpha,16 R^\alpha)$, Lemma~\ref{lemma:concentration_on_dyadic_cubes} implies in particular that
	\[
	\bP\Big[\mu_n(Q^\prime) \in \big(2R^\alpha,20 R^\alpha\big) \Big] \ge \frac{1}{2} \, ,
	\]
	for all $R$ large enough, uniformly as $n\to \infty$. Furthermore, by Lemma~\ref{lemma:overcrowding_probability_in_dyadic_cube} we have
	\begin{equation*}
		\bP\Big[\mu_n(Q) = \mu_n(Q^\prime) \mid \mathcal{F}_{j_6} \, \Big] \ge \exp\Big(-c (j_5- j_6) \mu_n(Q^\prime)^2\Big) \ge \exp\Big(-c \, (\log R )\, \mu_n(Q^\prime)^2\Big) 
	\end{equation*}
	and altogether we get
	\begin{align*}
		\bP\Big[\mu_n(Q) \ge 2R^\alpha \Big] &\ge \bE\Big[\mathbf{1}_{\{\mu_n(Q^\prime) \in (2R^\alpha,20 R^\alpha)\}} \bP\big[\mu_n(Q) = \mu_n(Q^\prime)  \mid \mathcal{F}_{j_6} \, \big]\Big]  \\ &\ge \exp\Big(-c \, (\log R) \, R^{2\alpha} \Big) \, \bP\Big[\mu_n(Q^\prime) \in \big(2R^\alpha,20 R^\alpha\big)\Big]  \ge \frac{1}{2}\exp\Big(-c \, (\log R) \, R^{2\alpha} \Big) \, .
	\end{align*}
	The inequality above, together with~\eqref{eq:overcrowding_lower_bound_inclusion_of_events}, concludes the lower bound and we are done.
	\qed 
	\appendix
	\section{Moderate deviations lower bound : Proof of Proposition~\ref{prop:moderate_deviations_lower_bound_bounded_random_variables}}
	\label{sec:moderate_deviations_lower_bound}
	The proof of Proposition~\ref{prop:moderate_deviations_lower_bound_bounded_random_variables} follows the standard scheme of the change of measure technique, see for instance the book by Dembo and Zeitouni~\cite{DemboZeitouniLDP}. Denote by
	\[
	\varphi_j (t) = \bE \big[e^{tX_j}\big] \qquad t\in \bR 
	\]
	the moment generating function of $X_j$. As $X_j$ have zero mean and $|X_j| \le \La$ almost surely, we have the asymptotic expansion as $t\to 0$
	\begin{equation}
		\label{eq:expansion_of_generating_function}
		\big|\varphi_j(t) - 1 - \frac{t^2}{2} \bE\big[X_j^2\big]\big| \lesssim  t^3
	\end{equation}
	where the error term depends only on $\La$. Denote by $\mu_j$ the law of $X_j$, and consider a new probability measure $\widetilde{\mu}_j$ given by
	\begin{equation}
		\label{eq:def_of_tilted_measure}
		\frac{{\rm d} \widetilde{\mu}_j(x)}{{\rm d}\mu_j (x)} = \frac{e^{\xi x}}{\varphi_j(\xi)} \qquad x\in \bR 
	\end{equation}
	where $\xi>0$ is a small parameter that we choose later. Denote by $\widetilde{X}_1,\ldots,\widetilde{X}_N$ a realization of independent random variables, where $\widetilde{X}_j$ has law $\widetilde{\mu}_j$. We also denote by $$\widetilde{S}_N = \widetilde{X}_1+\ldots+\widetilde{X}_N \, .$$	
	\begin{proof}[Proof of Proposition~\ref{prop:moderate_deviations_lower_bound_bounded_random_variables}]
		Let $\eta>0$ be a small parameter. Our starting point is the change of measure~\eqref{eq:def_of_tilted_measure}, which gives that
		\begin{align*}
			\bP\Big[S_N \ge N^\gamma\Big] &\ge \bP\Big[S_N\in \big(N^\gamma, N^{\gamma+2\eta}\big)\Big] = \int_{\{S_N\in (N^\gamma, N^{\gamma+2\eta})\}} {\rm d} \mu_1(x_1) \cdots {\rm d}\mu_N(x_N) \\ &= \int_{\{S_N\in (N^\gamma, N^{\gamma+2\eta})\}} e^{-\xi S_N} \prod_{j=1}^{N} \varphi_j(\xi) \ {\rm d} \widetilde\mu_1(x_1) \cdots {\rm d}\widetilde{\mu}_N(x_N) \\ & \ge \exp\bigg(\sum_{j=1}^{N} \log \varphi_j(\xi) - \xi N^{\gamma+2\eta}\bigg) \,  \bP_\xi \Big[\widetilde{S}_N \in \big(N^\gamma, N^{\gamma+2\eta}\big)\Big] \, .
		\end{align*}
		Therefore, to conclude the proof, we take $\xi = N^{\gamma-1+\eta}$ for $\eta<1-\gamma$ and show that both
		\begin{equation}
			\label{eq:moderate_deviation_lower_bound_after_tilte_exponential_term}
			\exp\bigg(\sum_{j=1}^{N} \log \varphi_j(\xi) - \xi N^{\gamma+2\eta}\bigg) \ge \exp\Big(-N^{2\gamma-1 + 3\eta}\Big)
		\end{equation}
		and 
		\begin{equation}
			\label{eq:moderate_deviation_lower_bound_after_tilte_probability_term}
			\lim_{N\to \infty} \bP_\xi \Big[\widetilde{S}_N \in \big(N^\gamma, N^{\gamma+2\eta}\big)\Big] = 1 
		\end{equation}
		holds. Indeed, since $\bE\big[X_j^2\big]\ge \sigma >0$ for all $j$, the inequality~\eqref{eq:expansion_of_generating_function} implies that
		\[
		\varphi_j(\xi) \ge 1 + \frac{\sigma}{4} \xi^2
		\]
		for all $N$ large enough, which in turn shows $\log \varphi_j(\xi) \ge 0$
		for all $1\le j \le N$. We get that
		\begin{align*}
			\sum_{j=1}^{N} \log \varphi_j(\xi) - \xi N^{\gamma+2\eta} &\ge - \xi N^{\gamma+2\eta} = - N^{2\gamma -1 + 3\eta} 
		\end{align*}
		for all $N$ large enough, which proves~\eqref{eq:moderate_deviation_lower_bound_after_tilte_exponential_term}. To deal with the probability term~\eqref{eq:moderate_deviation_lower_bound_after_tilte_probability_term}, we first need to understand the effect of the change of measure on the expectation. For all $j\in \{1,\ldots,N\}$, we have
		\begin{align*}
			\bE_\xi\big[\widetilde{X}_j\big] = \int_{\bR} x \, {\rm d}\widetilde\mu_j(x) = \int_{\bR} \frac{xe^{\xi x}}{\varphi_j(\xi)} \, {\rm d} \mu_j(x) = \frac{\partial}{\partial t} \bigg|_{t=\xi} \log \varphi_j(t) \, , 
		\end{align*}
		and hence, by~\eqref{eq:expansion_of_generating_function}, we see that
		\[
		\bE_\xi\big[\widetilde{X}_j\big] = \xi \, \bE\big[X_j^2\big] + \mathcal{O} (\xi^2)
		\]
		where the error term is uniform in $1\le j\le N$. Plugging $\xi = N^{\gamma-1+\eta}$ and summing over $j$, we arrive at the asymptotic equality
		\begin{equation*}
			\bE_\xi\big[\widetilde{S}_N\big] = N^{\gamma+\eta} \Big(\frac{1}{N} \sum_{j=1}^{N} \bE[X_j^2]\Big) + \mathcal{O}\big(N^{2\gamma-1+2\eta}\big) \, .
		\end{equation*}
		As $\eta<1-\gamma$ and with our assumptions on $\{X_j\}$, we see that for all $N$ large enough
		\begin{equation}
			\label{eq:asymptotic_of_expectation_after_change_of_measure}
			\frac{\sigma}{2} N^{\gamma+\eta} \le \bE_\xi\big[\widetilde{S}_N\big] \le 2\La^2 N^{\gamma+\eta} \, .
		\end{equation}
		Finally, since $\{\widetilde{X}_j\}$ are independent and bounded almost surely, we have
		$ \Var_\xi \big[\widetilde{S}_N\big] \lesssim N .
		$ 
		Together with~\eqref{eq:asymptotic_of_expectation_after_change_of_measure}, we apply Chebyshev's inequality and see that
		\begin{equation*}
			\bP_\xi \Big[\widetilde{S}_N \not\in \big(N^\gamma, N^{\gamma+2\eta}\big)\Big] \le \bP_\xi \Big[ \big|\widetilde{S}_N - \bE[\widetilde{S}_N]\big| \ge N^{\gamma + \eta/2}\Big] \le \frac{\Var_\xi \big[\widetilde{S}_N\big]}{N^{2\gamma +\eta}} \lesssim N^{1-2\gamma -\eta} \xrightarrow{N\to \infty} 0
		\end{equation*}
		which proves~\eqref{eq:moderate_deviation_lower_bound_after_tilte_probability_term} and hence the proposition.
	\end{proof}
	
	\printbibliography[heading=bibliography]

	\hspace{10mm}
\end{document}